\documentclass[12pt]{article}
\usepackage{verbatim,amsmath,amssymb,mathrsfs,amsthm,amsfonts,bm}
\usepackage{varwidth}
\usepackage{extarrows}
\usepackage{enumitem}
\usepackage{fancyhdr}
\usepackage{booktabs}
\usepackage[dvipsnames]{xcolor}
\usepackage{graphicx,wrapfig}
\usepackage{subfigure}
\usepackage{caption}
\usepackage{amsmath}
\usepackage{makecell}
\usepackage{multirow}
\usepackage{array}
\usepackage{setspace}
\usepackage{float} 
\usepackage{cite}
\usepackage{color,mathtools}
\usepackage{epstopdf}
\usepackage{lipsum}
\usepackage{appendix}
\usepackage{algorithm}
\usepackage{algorithmic}
\usepackage{blindtext}
\usepackage{ulem}

\numberwithin{equation}{section}
\numberwithin{figure}{section}
\numberwithin{table}{section}
\allowdisplaybreaks

\newtheorem{theorem}{Theorem}[section]

\newtheorem{remark}{Remark}[section]
\newtheorem{example}{Example}[section]


\usepackage{hyperref}
\hypersetup{
    colorlinks=true, linkcolor=black,
    urlcolor=cyan, citecolor=black,
}

\begin{document}

\title{Equilibrium preserving space in discontinuous Galerkin methods for hyperbolic balance laws}
\date{}

\author{}
\author{Jiahui Zhang\thanks{School of Mathematical Sciences, University of Science and Technology of China, Hefei, Anhui 230026, P.R. China.  E-mail: zjh55@mail.ustc.edu.cn.}
	\and Yinhua Xia\thanks{Corresponding author. School of Mathematical Sciences, University of Science and Technology of China, Hefei, Anhui 230026, P.R. China.  E-mail: yhxia@ustc.edu.cn.
	}
	\and Yan Xu\thanks{School of Mathematical Sciences, University of Science and Technology of China, Hefei, Anhui 230026, P.R. China.  E-mail: yxu@ustc.edu.cn.
	}
}

\maketitle

\begin{abstract}
  In this paper, we develop a general framework for the design of the arbitrary high-order well-balanced discontinuous Galerkin (DG) method for hyperbolic balance laws, including the compressible Euler equations with gravitation and the shallow water equations with horizontal temperature gradients (referred to as the Ripa model). Not only the hydrostatic equilibrium including the more complicated isobaric steady state in Ripa system, but our scheme is also well-balanced for the exact preservation of the moving equilibrium state. The strategy adopted is to approximate the equilibrium variables in the DG piecewise polynomial space, rather than the conservative variables, which is pivotal in the well-balanced property. Our approach provides flexibility in combination with any consistent numerical flux, and it is free of the reference equilibrium state recovery and the special source term treatment. This approach enables the construction of a well-balanced method for non-hydrostatic equilibria in Euler systems. Extensive numerical examples such as moving or isobaric equilibria validate the high order accuracy and exact equilibrium preservation for various flows given by hyperbolic balance laws. With a relatively coarse mesh, it is also possible to capture small perturbations at or close to steady flow without numerical oscillations.

\end{abstract}

\section{Introduction}\label{se:in}
A class of hyperbolic conservation laws with source terms, also known as hyperbolic balance laws, is considered in this paper. They have been widely used in various fields, including chemistry, biology, fluid dynamics, astrophysics, meteorology, etc. In the one-dimensional case, such a model takes the form of the following equation
\begin{equation}\label{hbl}
  u_t + f(u)_x = r(u),
\end{equation}
where $u,f(u),r(u)$ denote the conservative variable, the physical flux, and the source term, respectively. An important feature of such models is that they usually admit non-trivial steady-state solutions, in which the non-zero flux gradient exactly balances the source term.

As a typical example of hyperbolic balance laws, the shallow water equations incorporating the horizontal temperature gradients were introduced in \cite{ripa1993conservation} for modeling ocean currents. We often refer to it as the Ripa model. The shallow water equations are derived from the incompressible Navier-Stokes equations, assuming the density is constant. The Ripa system is a generalized model of the shallow water equations. Based on multi-layered ocean models with several layers of different constant densities, the Ripa system is obtained by vertically averaging each layer's density, horizontal pressure gradient, and velocity fields. The introduction of horizontal temperature gradients is advantageous to represent the variations in the fluid density within each layer.

Due to the presence of the source term, many standard numerical methods may only capture small perturbations of the near-equilibrium flow if extremely refined meshes are used. As a result, the well-balanced schemes were first proposed by Bermudez and Vazquez \cite{bermudez1994upwind} in the shallow water equations to preserve the exact equilibrium state solutions at a discrete level. There has been tremendous interest in constructing well-balanced schemes for the Ripa system in the last few years. Chertock et al. \cite{chertock2014central} first proposed a well-balanced and positivity-preserving central-upwind scheme using an interface tracking technique. A second-order accurate unstaggered central finite volume scheme was developed by Touma and Klingenberg \cite{touma2015well}. Based on the path-conservative approximate Riemann solver, \cite{sanchez2016hllc} derived an HLLC type scheme for the Ripa model, which possessed the well-balanced and positivity-preserving property at the same time. To achieve high order accuracy as well as capture the small perturbations of the hydrostatic state well without numerical oscillations near the discontinuity, several works \cite{han2016well,qian2018high,li2020high} under the framework of finite difference and discontinuous Galerkin methods were proposed. They extended the well-balanced approach based on modifying the numerical fluxes with the hydrostatic reconstruction and the special source term splitting used in the shallow water equations to the Ripa model. These schemes above mainly focus on the still water steady state (with zero velocity). Similar to the shallow water equations, there are some limitations in these schemes developed for the still water equilibrium that cannot be generalized to the moving water equilibrium (with non-zero velocity), which is much more difficult. The second-order surface reconstruction schemes were introduced in \cite{dong2022well} to maintain the still water and moving water equilibrium state solutions and guarantee the positivity of the water height and the temperature. The well-balanced discontinuous Galerkin method, which decomposed the numerical results into the equilibrium and perturbation part, was first introduced by \cite{xing2014exactly} for arbitrary equilibria of the shallow water equations. Britton and Xing \cite{britton2020high} extended the analogous idea to the Ripa model with a new and simple choice of Radau projection operator to recover the reference equilibrium values.

Another extensively considered example is the Euler equations under gravitational fields, which has many interesting phenomena in astrophysical and atmospheric, such as core-collapse supernova simulation, stellar evolution, numerical weather forecasting, etc. The first work was developed by LeVeque and Bale \cite{LeVeque1998wave}, which used a quasi-steady wave-propagation algorithm for the near-steady flow of an ideal gas subject to a static gravitational field. Later, many well-balanced schemes for the Euler equations with gravity have been designed in the literature. They may be broadly divided into three categories based on the steady states of interest to maintain.
The first category is that the steady state is assumed to be globally known and given as analytical formula. Based on a special source term treatment, Li and Xing introduced high order well-balanced finite volume, finite difference, and discontinuous Galerkin schemes, respectively, in \cite{li2016highfve,li2016welldge,li2018wellfde}. The $\alpha$-$\beta$ method presented by Klingenberg et al. \cite{klingenberg2019arbitrary} was well-balanced for a generally known hydrostatic equilibrium state as well as one type of moving equilibria and can arrive at arbitrary order of accuracy. For further recent development and discussion, we refer to the articles \cite{ghosh2016well,touma2016well,li2021simple,zhang2021high,wu2021uniformly}.
The second type is aimed at maintaining the barotropic hydrostatic state, where a certain barotropic relation is satisfied to impose a thermal stratification of the equilibrium state effectively. Isentropic conditions in which the specific entropy is constant and isothermal conditions in which the temperature is constant are prominent examples considered by many researchers, see \cite{kappeli2014well,li2018welldge,grosheintz2019high,chandrashekar2015second,botta2004well} for further information.
The third category is designed to preserve a consistent discretization of the partial differential equations underlying the steady state of interest. Some second-order schemes were studied in \cite{kappeli2016well,varma2019second,chertock2018well} without additional assumptions and a priori knowledge of the hydrostatic solution. A high-order method was generalized by Berberich et al. \cite{berberich2020high}. The only caveat is, of course, that this classification is more or less not strict. Moreover, the schemes above are mainly concerned with hydrostatic stationary state preservation. A generalized moving equilibrium with the non-zero velocity finite volume method was developed in \cite{grosheintz2020well}, where a discrete equivalent of steady adiabatic flow was preserved.
In addition, a local projection step was adopted in \cite{guerrero2021well} to develop a novel well-balanced high order DG method for systems of balance laws, and this strategy can be adapted to both Runge-Kutta and ADER time discretization.
C. Par{\'e}s et.al. developed arbitrary order accuracy well-balanced finite volume methods \cite{castro2020well} and finite difference methods \cite{pares2021well} based on a well-balanced reconstruction operator for systems of balance laws. The main difficulty in the application of this technique comes from the calculation of the stationary solutions at every cell or node at each time step, which equivalent to seek for a solution of an ODE system with prescribed average or point value at the computational interval. And in \cite{gomez2021high,gomez2021collocation}, the authors gave a general implementation of this technique in order to apply to any systems of balance laws.
The well-balanced schemes for more complex systems of hyperbolic equations such as the multi-layer shallow water equations, the general relativistic magnetohydrodynamics and the Einstein field equations were designed in \cite{gaburro2021well,dumbser2023wellbalanced,castro2012central}.

Recently, the discontinuous Galerkin method coupled with Runge-Kutta time discretization (dubbed as RKDG method) has been developed and analyzed by Cockburn et al. \cite{cockburn1989tvb2,cockburn1989tvb3,cockburn1990runge4,cockburn1998runge5} for solving hyperbolic conservation laws. It has an essential advantage in simulating unstructured meshes without losing high-order accuracy and conservation properties.
In this paper, we develop a well-balanced DG method for the Euler equations with gravitation and the Ripa model based on the equilibrium variables, where we can preserve the generalized moving equilibria in one dimension and hydrostatic equilibria in two dimensions.
The idea of change of variables was adopted by Hughes et.al. \cite{hughes1986new}, where they derived symmetric forms of the Euler and Navier-Stokes equations using the entropy variables. By utilizing these entropy variables, the resulting finite element methods are able to satisfy the second law of thermodynamics and a discrete entropy inequality.
Mantri et.al. \cite{mantri2021well} developed a well-balanced DG scheme for 2$\times$2 hyperbolic balance law by solving a linear system of equations for the equilibrium variables given by the global fluxes. In \cite{zhang2023moving}, the approximation of the equilibrium variables in the DG piecewise polynomial space yields a moving water equilibria preserving DG method for the shallow water equations.
Suitable equilibrium variables must be chosen properly for different hyperbolic systems, based on the ideas from \cite{zhang2023moving}, in order to ensure the well-balance and accuracy of the proposed approach. Starting from the steady state provided by the steady adiabatic flow \cite{grosheintz2020well}, we define the specific entropy, non-zero momentum, and the Bernoulli equation to be our equilibrium variables for Euler equations. The well-balanced property is achieved by reformulating the DG space in terms of the equilibrium variables rather than conservative variables, together with the modified numerical fluxes based on the generalized hydrostatic reconstruction. To our knowledge, this is the first attempt at the discontinuous Galerkin method for the Euler equations with gravitation involving non-hydrostatic equilibria. Analogously, the energy, water discharge, and temperature are viewed as the equilibrium variables for the generalized non-hydrostatic equilibria in Ripa model. More complicated zero-velocity steady states which will be explained in detail in Section \ref{se:mo} lead to the main difference from the shallow water equations. The same framework for the equilibrium preserving schemes can also be extended to the constant water height and isobaric steady state, with special treatment of the numerical flux in isobaric steady state. In contrast to schemes \cite{xing2014exactly,gomez2021high}, our approach is free of the reference equilibrium state recovery or the ODE solver for a stationary solution. The theoretical analysis and numerical tests verify that the proposed scheme is well-balanced. Meanwhile, the DG method is naturally high-order accurate and shock-capturing with appropriate limiters.


An outline of the rest paper is as follows.
Section \ref{se:mo} presents the mathematical model of the Euler equations with gravitation and the Ripa system, as well as the corresponding steady-state solutions. We also discuss the transformation between the conservative and equilibrium variables.
In Section \ref{se:wb}, by introducing some necessary notations for discontinuous Galerkin methods, we propose a new well-balanced DG method that can maintain both the hydrostatic equilibria and the generalized one-dimensional moving equilibria for hyperbolic balance laws.
Numerical results in different circumstances are illustrated in Section \ref{se:nu} to validate that the DG scheme is well-balanced, high-order accurate and shock-capturing. Finally,  concluding remarks are presented in the last section.

\section{Model description}\label{se:mo}
To begin with, we introduce two prototype examples of hyperbolic balance laws considered in this paper, the corresponding stationary state solutions and the conservative and equilibrium variables transformation. 

\subsection{Euler equations with gravitation}\label{subse:euler}
The compressible Euler equations with the effect of the gravitational force have the following form  in $d$-dimensional case
\begin{equation}\label{euler2D}
    \left\{\begin{array}{l}
    \rho_t+\nabla \cdot(\rho  \boldsymbol{u})=0, \\
    (\rho  \boldsymbol{u})_t+ \nabla \cdot (\rho  \boldsymbol{u}\otimes \boldsymbol{u} + p \mathbf{I}_d )  =-\rho \nabla\phi,\\
    E_t+\nabla \cdot((E+p) \boldsymbol{u}) = -\rho  \boldsymbol{u}\cdot \nabla\phi,
    \end{array}\right.
\end{equation}
where $\rho, \boldsymbol{u},p$ represent the mass density, the velocity of the gas, and the pressure, respectively. The total fluid energy $E=\rho e + \frac{1}{2}\rho \|\boldsymbol{u}\|^2$ is the sum of the internal and kinetic energy, with $e$ being the internal energy. $\rho \boldsymbol{u}$ denotes the momentum vector and $\phi=\phi(\boldsymbol{x})$ is the time-independent gravitational potential. $\mathbf{I}_d$ is the identity matrix and the operators $\nabla$, $\nabla\cdot$, $\otimes$ denote the gradient, divergence and tensor product. An equation of state (EoS) which relates the specific internal energy to the density and pressure, is important to close this system of equations. For the ideal gas law, it is given by
\begin{equation}\label{eos}
  p = p(\rho,e) = (\gamma-1)\rho e = (\gamma-1)\left(E - \dfrac{1}{2}\rho \|\boldsymbol{u}\|^2\right),
\end{equation}
with $\gamma$ being the ratio of specific heats.

\subsubsection{Steady state solutions}
This model allows for the hydrostatic stationary solutions with $\boldsymbol{u}= 0$ and describes the mechanical balance between pressure and external forces such as gravity:
\begin{equation}\label{hydro:euler}
  \rho = \rho(\boldsymbol{x}), \quad  \boldsymbol{u}= 0, \quad \nabla p = -\rho\nabla \phi.
\end{equation}
It is an underdetermined system of equations. To fully characterize the equilibrium, additional assumptions such as a thermal equilibrium stratification need to be made.
As a concrete astrophysically relevant example of stationary states, isentropic equilibrium is a special case of barotropic hydrostatic equilibria,
in which the pressure is only a function of density. The thermodynamic relation for isentropic conditions is
\begin{equation}\label{is_re}
  d h = Tds+\dfrac{dp}{\rho},
\end{equation}
where the enthalpy $h$ is defined by $h=e+\frac{p}{\rho}$, $T$ is the temperature and $s$ is the specific entropy. For the constant specific entropy ($ds = 0$), (\ref{is_re}) can easily be integrated into the following isentropic hydrostatic equilibrium state
\begin{equation}\label{isen:euler}
  s = \text{constant}, \quad  \boldsymbol{u}= 0, \quad  h + \phi = \text{constant}.
\end{equation}
An equivalent formulation is the polytropic equilibrium state given by
\begin{equation}\label{pol_euler}
  \rho = \left( \dfrac{\gamma-1}{K\gamma}(c-\phi) \right)^\frac{1}{\gamma-1}, \quad \boldsymbol{u}=0, \quad p = \dfrac{1}{K^{\frac{1}{\gamma-1}}}\left(\dfrac{\gamma-1}{\gamma}(c-\phi)\right)^\frac{\gamma}{\gamma-1},
\end{equation}
with $c$ being a constant and $K$ being a constant and depends on the entropy alone, i.e. $K = K(s)$.

Another steady state with non-zero velocity  $\boldsymbol{u}\neq 0$ is referred to as moving or adiabatic equilibrium. It arises from the fact that the entropy inside a fluid particle remains constant for smooth flows, while the entropy does not have to be constant globally. It is advected with the flow. With the help of Bernoulli's equation
\begin{equation}\label{be_eq}
  \dfrac{1}{2}u^2 + h + \phi = \text{constant},
\end{equation}
which is derived in \cite{LANDAU1987313},
the steady adiabatic flow in one dimension takes the following form
\begin{equation}\label{adi:euler}
  s = \text{constant}, \quad \rho u= \text{constant}, \quad  \varepsilon := \dfrac{1}{2}u^2 + h + \phi = \text{constant}.
\end{equation}
Notice that Bernoulli's equation (\ref{be_eq}) is valid along each streamline, lines tangent to the velocity of the flow. It may differ from streamline to streamline. And the extension to multi-dimensional steady adiabatic flow is only exactly well-balanced with streamlines aligned along a computational grid axis. More details can be found in \cite{LANDAU1987313,grosheintz2020well}. We are mainly concerned with one- and two-dimensional isentropic hydrostatic equilibria (\ref{isen:euler}) and moving equilibrium state  (\ref{adi:euler}) in one dimension.

\subsubsection{Conservative and equilibrium variables}

For simplicity of presentation, we take the one-dimensional case as an example to describe the sets of conservative and equilibrium variables on which our well-balanced scheme depends. Before proceeding further, we comment that the ideal gas law is considered in this paper, but the methods presented below are applicable to any convex equation of state.
The conservative variables are defined by $\boldsymbol{U} = (\rho,\rho u,E)^T$, here the superscript $T$ denoting the transpose. For the ideal gas, the polytropic relation which relates the density and pressure at constant entropy, is characterized by
\begin{equation}\label{po_re}
  p=K\rho^\gamma,
\end{equation}
where $K$ is a constant only dependent on the entropy. Thus, using the EoS (\ref{eos}) together with the polytropic relation (\ref{po_re}), we can write the enthalpy as follows
\begin{equation}\label{h_ex}
h = e+\dfrac{p}{\rho} = \dfrac{\gamma}{\gamma-1}\dfrac{p}{\rho} = \dfrac{\gamma}{\gamma-1}K\rho^{\gamma-1}.
\end{equation}
For convenience, when defining the equilibrium variables, we use $K$ as a substitute for specific entropy $s$ in (\ref{adi:euler}), owing to no requirement for the explicit expression of the entropy and constant remaining of $K$ in space and time. Thereby, we can define the one-dimensional equilibrium variables $\boldsymbol{V} = (K,\rho u,\varepsilon)^T$. Given $\boldsymbol{U}$ and gravitational potential $\phi$, we can easily obtain the transform function from conservative variables to equilibrium variables, which is denoted as
\begin{equation}\label{uv:euler}
\boldsymbol V=\boldsymbol{V}(\boldsymbol{U},\phi) =
              \left(\begin{array}{c}
               \displaystyle \frac{p}{\rho^\gamma} \\
                \rho u\\
               \displaystyle \frac{1}{2}\frac{(\rho u)^2}{\rho^2}+\frac{\gamma}{\gamma-1}\frac{p}{\rho} + \phi
              \end{array}\right).
\end{equation}
Here $p$ is acquired from (\ref{eos}). Conversely, given $\boldsymbol {V}$ and $\phi$, the inverse function $\boldsymbol{U}(\boldsymbol{V},\phi)$ cannot be directly computed due to the nonlinearity of the Bernoulli's constant $\varepsilon$. Plugging (\ref{h_ex}) back into (\ref{be_eq}), we have
\begin{equation}\label{nl_solver}
  \dfrac{(\rho u)^2}{2\rho^2}+\dfrac{\gamma}{\gamma-1}K\rho^{\gamma-1} + \phi - \varepsilon = 0,
\end{equation}
from which we need to seek the correct root denoted by $\rho = \rho(K,\rho u,\varepsilon,\phi)$.
If the momentum $\rho u$ is zero, the solution of (\ref{nl_solver}) is trivial
\begin{equation*}
  \rho = \left( \dfrac{\gamma-1}{K\gamma}(\varepsilon-\phi) \right)^\frac{1}{\gamma-1}.
\end{equation*}
For the nonzero momentum $\rho u$, slightly different from the shallow water equations\cite{zhang2023moving}, (\ref{nl_solver}) cannot reduce to a cubic equation solver with explicit expressions thanks to the value of $\gamma$. By introducing the function which is convex for the ideal gas law:
\begin{equation*}
  \xi(\rho) = \dfrac{(\rho u)^2}{2\rho^2}+\dfrac{\gamma}{\gamma-1}K\rho^{\gamma-1},
\end{equation*}
we can get its unique minimum $\xi_{\star}$ at point
\begin{equation}\label{rho_point}
\rho_{\star} = \left(\dfrac{(\rho u)^2}{\gamma K}\right)^\frac{1}{\gamma+1},\quad \xi_{\star} =  \xi(\rho_{\star}).
\end{equation}
The value $\rho_{\star}$  is referred to as the critical density, which separates the solution into a subsonic and supersonic branch.

We are now in a position to claim that
\begin{itemize}
\item If $\varepsilon - \phi<\xi_{\star}$, no equilibrium exists.
\item If $\varepsilon - \phi=\xi_{\star}$, one equilibrium exists and the critical density $\rho_{\star}$ is the solution of  (\ref{nl_solver}). We can also calculate the velocity at the critical point
\begin{equation}\label{u_point}
  u_{\star}^2 = \dfrac{(\rho u)^2}{\rho_{\star}^2} = \gamma K \rho_{\star}^{\gamma-1}. 
\end{equation}
\item If $\varepsilon - \phi>\xi_{\star}$, two equilibrium solutions exist.
\end{itemize}
And in order to label the different flow regimes, we introduce the Mach number
\begin{equation}\label{mach_euler}
 M := \dfrac{|u|}{c}, \quad \text{with} \quad  c^2 = \left(\dfrac{\partial p}{\partial \rho}\right)_s = \gamma K \rho^{\gamma-1}.
\end{equation}
For the subsonic flow where $M <1$, we select the subsonic solution $\rho> \rho_{\star}$, otherwise the supersonic solution $\rho<\rho_{\star}$ in the supersonic flows $M >1$ is chosen. Moreover, from (\ref{u_point}) and the definition of the sound speed $c$ in (\ref{mach_euler}), we can observe the Mach number equals 1 when $\rho=\rho_{\star}$. Similar issues arise in the shallow water equations, and Noelle et al. \cite{noelle2007high} give a thorough discussion. In our calculation, Newton iteration coupled with the analogous initial guess choice in \cite{noelle2007high} is applied to solve (\ref{nl_solver}).
Once a correct root $\rho$ is found, the conservative variables are simply given by
\begin{equation}\label{vu:euler}
\boldsymbol{U} = \boldsymbol{U}(\boldsymbol{V},\phi) = \left(\begin{array}{c}
                \rho(K,\rho u,\varepsilon,\phi) \\
                \rho u\\
                \displaystyle\frac{(\rho u)^2}{2\rho(K,\rho u,\varepsilon,\phi)}+\frac{1}{\gamma-1}K\rho^\gamma
              \end{array}\right).
\end{equation}
As can be seen, the variable transformation between $\boldsymbol{U}$ and $\boldsymbol{V}$ is a one-to-one correspondence under the determined Mach number.

For two-dimensional isentropic hydrostatic equilibria preserving, we define $\boldsymbol U=(\rho,\rho u,\rho v,E)^T$ to be the conservative variables, and the following equilibrium variables $\boldsymbol V=(K,\rho u,\rho v,\varepsilon)$ with
\begin{equation}\label{nl_solver2d}
\varepsilon =\dfrac{1}{2}(u^2+v^2)+h+\phi= \dfrac{(\rho u)^2+(\rho v)^2}{2\rho^2}+\dfrac{\gamma}{\gamma-1}K\rho^{\gamma-1} + \phi.
\end{equation}
A similar variable transform as one dimension case can be taken. Such a definition is reasonable owing to the fact that (\ref{nl_solver2d}) reduces to $\varepsilon=h + \phi$ with zero velocity, which is consistent with the isentropic equilibrium state.

\subsection{Ripa model}\label{subse:ripa}
The shallow water equations with horizontal temperature gradients take the following form in a two-dimensional case
\begin{equation}\label{ripa}
    \left\{\begin{array}{l}
    h_t+(hu)_{x}+(hv)_{y}=0, \\
    (hu)_t+(hu^2+\frac{1}{2}gh^2\theta)_x+(huv)_y=-gh\theta b_x,\\
    (hv)_t+(huv)_x+(hv^2+\frac{1}{2}gh^2\theta)_y=-gh\theta b_y,\\
    (h\theta)_t+(hu\theta)_x+(hv\theta)_y = 0,
    \end{array}\right.
\end{equation}
where $h,b,\theta,g$ stands for the water height, the bottom topography function, the potential temperature field, and the gravitational constant, respectively, $u,v$ is the depth-averaged velocity of the fluid.  Here $\theta = \frac{g\Delta \Theta}{\Theta_{\text{ref}}}$ is the reduced gravity with $\Delta \Theta$ denotes the difference in potential temperature from a reference value $\Theta_{\text{ref}}$. Additionally, $hu,hv$ is the water discharge, and $p = \frac{1}{2}gh^2\theta$ represents the pressure dependent on the water temperature.

\subsubsection{Steady state solutions}
Similar to the shallow water equations, no general form of the moving water equilibria exists in two dimensions. Providing the velocity does not vanish, the well-known moving water equilibrium state solutions in one dimension can be derived as
\begin{equation}\label{moving_water:Ripa}
  hu = \text{constant}, \quad E := \frac{1}{2}u^2 + g\theta(h + b) = \text{constant}, \quad \theta=\text{constant}.
\end{equation}
When the velocity vanishes, the system admits several non-trivial steady-state solutions defined by the following unsolvable PDE system
\begin{equation}\label{ripa_zerov}
    \left\{\begin{array}{l}
    u=0, \ v=0,\\
    \left(\dfrac{1}{2}h^2\theta \right)_x=-h\theta b_x, \left(\dfrac{1}{2}h^2\theta \right)_y=-h\theta b_y.
    \end{array}\right.
\end{equation}
Nevertheless, additional assumptions can be enforced on $\theta,b,h$ to solve (\ref{ripa_zerov}).
The first steady state is the still water equilibrium which is obtained with constant temperature
\begin{equation}\label{still_water:Ripa}
  u = 0,\quad v=0,\quad \theta=\text{constant},\quad h + b = \text{constant}.
\end{equation}
It is a special case of the generalized steady state (\ref{moving_water:Ripa}). The second is the isobaric equilibrium
\begin{equation}\label{isobaric:Ripa}
  u = 0,\quad v=0, \quad b = \text{constant},\quad h^2\theta=\text{constant},
\end{equation}
where the height and temperature jump but velocity and pressure remain constant. The third case is the constant water height equilibrium
\begin{equation}\label{con_water:Ripa}
  u = 0, \quad v=0,\quad h = \text{constant},\quad b+\frac{1}{2}h\ln{\theta}=\text{constant}.
\end{equation}
Such more complicated zero-velocity steady states yield the main discrepancy from the shallow water equations.  All three steady states can be treated within the current framework, only the isobaric steady state (\ref{isobaric:Ripa}) requires special treatment of the numerical flux.

\subsubsection{Conservative and equilibrium variables}
In one dimension, we define the conservative variables $\boldsymbol U=(h,hu,h\theta)^T$. The equilibrium variables for moving water equilibrium are denoted by
\begin{equation}\label{mo_ev:ripa}
\boldsymbol V=\left(\begin{array}{c}
                E \\
                hu\\
                \theta
              \end{array}\right) =
              \left(\begin{array}{c}
                \frac{(hu)^2}{2h^2} + g\theta (h + b) \\
                hu\\
                \frac{h\theta}{h}
              \end{array}\right).
\end{equation}
On the one hand, given $\boldsymbol U$ and the bottom function $b$, the transform function denoted as $\boldsymbol{V} = \boldsymbol{V}(\boldsymbol{U},b)$ can be directly acquired from (\ref{mo_ev:ripa}).
On the other hand, suppose $\boldsymbol V$ and $b$ are given, the inverse transform function
\begin{equation}\label{tr_rel}
\boldsymbol{U} = \boldsymbol{U}(\boldsymbol{V},b) = \left(\begin{array}{c}
                h(E,hu,\theta,b) \\
                hu\\
                h(E,hu,\theta,b)\cdot \theta
              \end{array}\right)
\end{equation}
can be obtained by solving an equivalent cubic polynomial
\begin{equation}\label{cubic:ripa}
   g\theta h^3 - (E-gb\theta)h^2 + \frac{1}{2}(hu)^2 = 0,
\end{equation}
and what we need is to find the correct root of this cubic equation. For this reason, we introduce the Froude number
\begin{equation}\label{fr_ripa}
 Fr := |u|/\sqrt{gh\theta}
\end{equation}
to label the different flow regimes and distinguish between the following three state
\begin{equation*}
\left\{\begin{array}{lll}
    Fr<1,&  \text{subcritical flow}, \\
    Fr=1,&  \text{sonic flow}, \\
    Fr>1,&  \text{supercritical flow}.
    \end{array}\right.
\end{equation*}
It plays the same role as the Mach number (\ref{mach_euler}).
Assuming that the admissible values $E,hu,\theta,b$ satisfy $$E\geq \frac{3}{2}(g\theta|hu|)^{\frac{2}{3}} + gb\theta,$$
we can find the solutions analytically, which read as
\begin{itemize}
\setlength{\parskip}{0.0em}
  \item [1)] If  $hu=0$, the unique physically relevant root is $h = \frac{E}{g\theta} - b$.
  \item [2)] If $hu\neq 0$ and $E>\frac{3}{2}(g\theta|hu|)^{\frac{2}{3}} + gb\theta,$ (\ref{cubic:ripa}) will return three real-valued roots denoted by $$h_1<0<h_2<h_3.$$ They are written explicitly as
             \begin{equation}\label{roothe}
              \begin{aligned}
               &h_1 = \dfrac{E-gb\theta}{3g\theta}\left(1- 2\cos{\dfrac{t}{3}}\right),\\
               &h_2 = \dfrac{E-gb\theta}{3g\theta}\left(1+ \cos{\dfrac{t}{3}} - \sqrt{3} \sin{\dfrac{t}{3}}\right),\\
               &h_3 = \dfrac{E-gb\theta}{3g\theta}\left(1+ \cos{\dfrac{t}{3}} + \sqrt{3} \sin{\dfrac{t}{3}}\right),\\
               &t = \arccos{L}, \ L = -1 + \dfrac{27}{4}\dfrac{g^2\theta^2(hu)^2}{(E-gb\theta)^3}.
              \end{aligned}
             \end{equation}
             Here, the two positive real roots correspond to the supercritical ($h_2$) and subcritical ($h_3$) cases, respectively.
  \item [3)] If $hu\neq 0$ and $E=\frac{3}{2}(g\theta|hu|)^{\frac{2}{3}} + gb\theta,$ (\ref{cubic:ripa}) will return three real-valued roots denoted by $$h_1<0<h_2=h_3.$$
  The unique physically relevant double root is $h_2=h_3= \left(\dfrac{(hu)^2}{g\theta}\right)^{\frac{1}{3}}$, which is the sonic point.
\end{itemize}
It shows that for the Ripa system, the transform relationship between $\boldsymbol{U}$ and $\boldsymbol{V}$ is also a one-to-one correspondence under the determined Froude number.

In two dimensions, the conservative variables $\boldsymbol{U} = (h,hu,hv,h\theta)^T$, the equilibrium variables $\boldsymbol{V} = (E,hu,hv,\theta)^T$, as well as the variable transform are obtained in a similar way. Here $E$ is defined by
\begin{equation}\label{E_solver2d}
E=\dfrac{1}{2}(u^2+v^2) + g\theta (h + b) = \dfrac{(hu)^2+(hv)^2}{2h^2} + g\theta (h + b),
\end{equation}
which will reduce to the still water equilibrium state (\ref{still_water:Ripa}) when the velocities $u,v$ vanish.

\section{The well-balanced discontinuous Galerkin method}\label{se:wb}
This section develops the well-balanced discontinuous Galerkin method for hyperbolic balance laws, in quest of preserving both the hydrostatic and moving equilibrium state. The extension of the developed scheme to the two-dimensional case is only restricted to the hydrostatic equilibria.

\subsection{Notations}\label{subse:no}
For the computational domain $\Omega\in \mathbb{R}^d$, let $\mathcal {K}$ be a family of partitions such that
$$\Omega = \bigcup\{\mathcal {T}|\mathcal {T}\in \mathcal {K}\}.$$
For any element $\mathcal {T} \in \mathcal {K}$, we denote $\Delta_{\mathcal {T}}$ as the area of the cell $\mathcal {T}$.
Specifically speaking, supposing that the computational domain in one dimension is divided into $nx$ subintervals, 
namely the element $\mathcal {T}$ is taken as the cell $\mathcal I_j$ and
$$\Omega = \bigcup   _{j=1}^{nx} \mathcal I_j = \bigcup  _{j=1}^{nx}[x_{j-\frac{1}{2}}, x_{j+\frac{1}{2}}].$$
The $x_{j\pm \frac{1}{2}}$ is the left or right cell interface.
For two-dimensional case, we denote the mesh $\mathcal {K}$ by rectangular cells $\mathcal I_{ij}$, i.e.
$$\Omega = \bigcup  _{i=1}^{nx}\bigcup _{j=1}^{ny}{\mathcal I_{ij}} = \bigcup  _{i=1}^{nx}\bigcup _{j=1}^{ny} [x_{i-\frac{1}{2}},x_{i+\frac{1}{2}}] \times [y_{j-\frac{1}{2}},y_{j+\frac{1}{2}}].$$

We define the following piecewise polynomial finite element space
 \begin{equation}\label{fes1}
\mathcal{W}_h^k:= \left\{ w: w|_{\mathcal {T}}\in P^k(\mathcal {T}), \forall \mathcal {T}\in \mathcal {K}  \right\},
\end{equation}
where $P^k(\mathcal {T})$ denotes the space of polynomials in cell $\mathcal {T}$ of degree at most $k$, and
 \begin{equation}\label{fes2}
\boldsymbol{\mathcal{V}}_h^k:= \left\{ \boldsymbol v: \boldsymbol v = (v_1,\dots, v_{d+2})^T | \ v_1,\dots,v_{d+2} \in  \mathcal{W}_h^k  \right\},
\end{equation}
which consists of $d+2$ components in the $d$-dimensional equation systems we considered.
In the DG methods, for any unknown variable $u$, we project it into the finite element space $\mathcal{W}_h^k$ and still denote the numerical approximation as $u$ with abuse of notation.
The values $u^{int_{\mathcal T}},u^{ext_{\mathcal T}}$ denote the approximations at the cell boundary $\partial {\mathcal T}$, and they are defined by
\begin{equation}\label{ub}
u^{int_{\mathcal T}}(\boldsymbol{x}) := \lim _{\epsilon \rightarrow 0^+} u(\boldsymbol{x} - \epsilon\boldsymbol{n}), \
u^{ext_{\mathcal T}}(\boldsymbol{x}) := \lim _{\epsilon \rightarrow 0^+} u(\boldsymbol{x} + \epsilon\boldsymbol{n}), \ \forall \boldsymbol{x}\in \partial {\mathcal T}.
\end{equation}
Here $\boldsymbol{n}$ denotes the outward unit normal vector.

\subsection{The semi-discrete scheme with hydrostatic reconstruction}
Different from the traditional high-order well-balanced DG methods, instead of seeking an approximation to the conservative variables $\boldsymbol {U}$, we approximate the equilibrium variables $\boldsymbol {V}$ in the DG piecewise polynomial space $\boldsymbol{\mathcal{V}}_h^k$. Similar to the procedure in \cite{zhang2023moving}, for simplicity of presentation, we rewrite the hyperbolic balance laws as
\begin{equation}\label{re_hyper2d}
\begin{aligned}
\boldsymbol {U}(\boldsymbol{V})_t + \nabla \cdot\boldsymbol F(\boldsymbol {U}(\boldsymbol{V})) = \boldsymbol r(\boldsymbol {U}(\boldsymbol{V}),\omega).
 \end{aligned}
 \end{equation}
Here $\omega$ represents the gravitational potential $\phi(x)$ in Euler equations or the bottom topography $b(x)$ in Ripa systems.
Notice that we do not change the form of the equation itself; only the conservative variables are treated as a nonlinear function of the equilibrium variables, thus maintaining the conservation of the hyperbolic conservation laws. 

Based on the hydrostatic reconstruction technology \cite{li2018welldge}, we design our well-balanced semi-discrete DG method, which reads as:
Find $\boldsymbol{V}\in \boldsymbol{\mathcal{V}}_h^k$, s.t. for any test function $\boldsymbol{\varphi} = ({\varphi}_1,\dots,{\varphi}_4)\in \boldsymbol{\mathcal{V}}_h^k$ we have
\begin{equation}\label{scheme2d}
\dfrac{d}{dt}\int_{\mathcal T}\boldsymbol{U}(\boldsymbol{V})\cdot \boldsymbol{\varphi} \ d\boldsymbol{x} = \text{RHS}(\boldsymbol{V},\omega,\boldsymbol{\varphi})
\end{equation}
with
\begin{equation*}
\begin{aligned}
 \text{RHS}(\boldsymbol{V},\omega,\boldsymbol{\varphi})
 =\int_{\mathcal T}\boldsymbol{F}(\boldsymbol{U}(\boldsymbol{V})): \nabla\boldsymbol{\varphi} \ d\boldsymbol{x} -
  \int_{\partial \mathcal T} \widehat{\boldsymbol{F}}^{*} \cdot\boldsymbol{\varphi}^{int_{\mathcal T}} \ ds +\int_{\mathcal T}\boldsymbol r(\boldsymbol{U}(\boldsymbol{V}),\omega) \cdot \boldsymbol{\varphi} \ d\boldsymbol{x}.
 \end{aligned}
\end{equation*}
Here $``\cdot"$ is the Euclidean dot product, $``:"$ is the Frobenius inner product. 
The well-balanced numerical flux $\widehat{\boldsymbol{F}}^{*}$ is computed as follows. At each time step, we first obtain the boundary values by transform relationship (\ref{vu:euler}) or \eqref{tr_rel}
\begin{equation}\label{hy_u}
   {\boldsymbol{U}}^{int_{\mathcal T}} = {\boldsymbol{U}}\left({\boldsymbol{V}}^{int_{\mathcal T}},\omega^{int_{\mathcal T}}\right), \quad
   {\boldsymbol{U}}^{ext_{\mathcal T}} = {\boldsymbol{U}}\left({\boldsymbol{V}}^{ext_{\mathcal T}},\omega^{ext_{\mathcal T}}\right).
\end{equation}
For another, we evaluate their modification
\begin{equation}\label{hy_up}
   {\boldsymbol{U}}^{*,int_{\mathcal T}} = {\boldsymbol{U}}\left({\boldsymbol{V}}^{int_{\mathcal T}},\omega^{*}\right), \quad
   {\boldsymbol{U}}^{*,ext_{\mathcal T}} = {\boldsymbol{U}}\left({\boldsymbol{V}}^{ext_{\mathcal T}},\omega^{*}\right)
\end{equation}
with
\begin{equation}\label{hy_phi}
\begin{aligned}
   & \omega^{*} = \max{(\omega^{int_{\mathcal T}},\omega^{ext_{\mathcal T}})}.
\end{aligned}
\end{equation}
Then, the well-balanced numerical fluxes are taken as
\begin{equation}\label{flux2d}
\begin{aligned}
\widehat{\boldsymbol{F}}^{*}  =  \widehat{\boldsymbol{F}}({\boldsymbol{U}}^{*,int_{\mathcal T}},{\boldsymbol{U}}^{*,ext_{\mathcal T}},\boldsymbol{n}) - \boldsymbol{F}({\boldsymbol{U}}^{*,int_{\mathcal T}})\cdot \boldsymbol{n} + \boldsymbol{F}({\boldsymbol{U}}^{int_{\mathcal T}})\cdot \boldsymbol{n}.
\end{aligned}
\end{equation}
Considering  the stability of the scheme, we choose the simplest monotone Lax-Friedrichs numerical flux for  $\widehat {\boldsymbol {F}}(\boldsymbol U^{int_{\mathcal T}},\boldsymbol U^{ext_{\mathcal T}},\boldsymbol{n})$, that is
\begin{equation}\label{LF_FLUX}
\begin{aligned}
  \widehat {\boldsymbol {F}}(\boldsymbol U^{int_{\mathcal T}},\boldsymbol U^{ext_{\mathcal T}},\boldsymbol{n}) &= \dfrac{1}{2}\left(\boldsymbol  {F}(\boldsymbol U^{int_{\mathcal T}})\cdot \boldsymbol{n}  + \boldsymbol {F}( \boldsymbol U^{ext_{\mathcal T}}) \cdot \boldsymbol{n}  -{\alpha}(\boldsymbol U^{ext_{\mathcal T}} - \boldsymbol U^{int_{\mathcal T}})\right),\\
  \alpha &= \max\limits_{\boldsymbol {U}} \left( \left|\dfrac{\partial \boldsymbol F}{\partial\boldsymbol U} \cdot\boldsymbol n\right|\right),
\end{aligned}
\end{equation}
the maximum is taken globally in our computation. If the isolated flow discontinuities exist, the Roe flux \cite{roe1981approximate} or HLLC flux \cite{toro1994restoration} with carefully chosen waves speeds is applied for the sake of preserving the well-balanced property.

\subsection{Special treatment for the isobaric equilibrium}
Schemes for the moving equilibria preserving fail to maintain the isobaric equilibrium state exactly, even employing a substantial mesh refinement, which will be shown in Example \ref{Ripa:pof}. Therefore, we design a new well-balanced DG method for the isobaric equilibrium state \eqref{isobaric:Ripa}. Note that the isobaric equilibrium preserving DG method does not preserve the moving water equilibrium due to the different equilibrium variables adopted.

We define the conservative variables $\boldsymbol U=(h,hu,hv,h\theta)^T$ and the equilibrium variables $\boldsymbol V=(h,u,v,p)^T$ from the isobaric steady state solution.
The variable transformation between $\boldsymbol U$ and $\boldsymbol V$ is much easier, which reads as
\begin{equation}\label{iso_vt}
\begin{aligned}
  \boldsymbol{V} = \boldsymbol{V}(\boldsymbol{U}) =
  \left(\begin{array}{c}
          h \\
         \displaystyle \frac{hu}{h} \\
       \displaystyle   \frac{hv}{h} \\
         \displaystyle \frac{g}{2}h\cdot h\theta
        \end{array}
  \right), \quad
  \boldsymbol{U} = \boldsymbol{U}(\boldsymbol{V}) =
  \left(\begin{array}{c}
          h \\
          h\cdot u \\
          h\cdot v \\
        \displaystyle  \frac{2}{g}\frac{p}{h}
  \end{array}
  \right).
\end{aligned}
\end{equation}
Based on the approximation of the equilibrium variables in the DG piecewise polynomial space, we get the following well-balanced semi-discrete DG scheme: Find $\boldsymbol{V}\in \boldsymbol{\mathcal{V}}_h^k$, s.t. for any test function $\boldsymbol{\varphi}\in \boldsymbol{\mathcal{V}}_h^k$ we have
\begin{equation}\label{scheme:ripaiso2d}
\dfrac{d}{dt}\int_{\mathcal T}\boldsymbol{U}(\boldsymbol{V})\cdot \boldsymbol{\varphi} \ d\boldsymbol{x} = \text{RHS}^{\text{iso}}(\boldsymbol{V},\omega,\boldsymbol{\varphi}),
\end{equation}
where
\begin{equation*}
\text{RHS}^{\text{iso}}(\boldsymbol{V},\omega,\boldsymbol{\varphi})
 =\int_{\mathcal T}\boldsymbol{F}(\boldsymbol{U}(\boldsymbol{V})): \nabla\boldsymbol{\varphi} \ d\boldsymbol{x} -
  \int_{\partial \mathcal T} \widehat{\boldsymbol{F}}^{\text{mod}} \cdot\boldsymbol{\varphi}^{int_{\mathcal T}} \ ds +\int_{\mathcal T}\boldsymbol r(\boldsymbol{U}(\boldsymbol{V}),\omega) \cdot \boldsymbol{\varphi} \ d\boldsymbol{x}.
\end{equation*}
Inspired from \cite{dong2022well}, we make a modification to the Lax-Friedrichs numerical flux \eqref{LF_FLUX} by introducing a parameter
\begin{equation}\label{delta}
\begin{aligned}
   & \delta= \max\left( \dfrac{-\min(\boldsymbol {u}^{int_{\mathcal T}} \cdot\boldsymbol n,0)}{\max(|\boldsymbol {u}^{int_{\mathcal T}} \cdot\boldsymbol n|,1)},  \dfrac{\max(\boldsymbol {u}^{ext_{\mathcal T}} \cdot\boldsymbol n,0)}{\max(|\boldsymbol {u}^{ext_{\mathcal T}} \cdot\boldsymbol n|,1)} \right).
\end{aligned}
\end{equation}
Some straightforward calculation yields $\delta\in [0,1]$.
Hence, the modified well-balanced numerical flux becomes
\begin{equation}\label{flux:ripaiso2d}
\begin{aligned}
&\widehat{\boldsymbol{F}}^{\text{mod}} = \widehat{\boldsymbol{F}}^{\text{mod}}(\boldsymbol U^{int_{\mathcal T}},\boldsymbol U^{ext_{\mathcal T}},\boldsymbol{n}) = \dfrac{1}{2}\left(\boldsymbol  {F}(\boldsymbol U^{int_{\mathcal T}})\cdot \boldsymbol{n}  + \boldsymbol {F}( \boldsymbol U^{ext_{\mathcal T}}) \cdot \boldsymbol{n}  -{\alpha}\delta(\boldsymbol U^{ext_{\mathcal T}} - \boldsymbol U^{int_{\mathcal T}})\right),
\end{aligned}
\end{equation}
with the boundary values $\boldsymbol U^{int_{\mathcal T}},\boldsymbol U^{ext_{\mathcal T}}$ obtained from the variable transformation \eqref{iso_vt}. This modification does not affect accuracy.

\subsection{The fully-discrete DG scheme}
For hyperbolic conservation laws, the semi-discrete scheme is usually completed by combining high-order strong stability preserving Runge-Kutta (SSP-RK) or multistep time discretization methods. Herein, we apply the following third-order SSP-RK  method for our time discretization:
\begin{equation}\label{3rk}
\begin{aligned}
& \int_{\mathcal {T}} \boldsymbol{U}(\boldsymbol{V}^{(1)}) \cdot \boldsymbol{\varphi} \ dx = \int_{\mathcal T} \boldsymbol{U}(\boldsymbol{V}^{n}) \cdot \boldsymbol{\varphi} \ dx + \Delta t  \ \text{RHS}(\boldsymbol{V}^n,\omega,\boldsymbol{\varphi}), \\
& \int_{\mathcal {T}} \boldsymbol{U}(\boldsymbol{V}^{(2)}) \cdot \boldsymbol{\varphi} \ dx = \dfrac{3}{4}\int_{\mathcal T} \boldsymbol{U}(\boldsymbol{V}^{n}) \cdot \boldsymbol{\varphi} \ dx + \dfrac{1}{4}\left(\int_{\mathcal T} \boldsymbol{U}(\boldsymbol{V}^{(1)}) \cdot \boldsymbol{\varphi} \ dx + \Delta t  \ \text{RHS}(\boldsymbol{V}^{(1)},\omega,\boldsymbol{\varphi}) \right),\\
& \int_{\mathcal {T}} \boldsymbol{U}(\boldsymbol{V}^{n+1}) \cdot \boldsymbol{\varphi} \ dx = \dfrac{1}{3}\int_{\mathcal T} \boldsymbol{U}(\boldsymbol{V}^{n}) \cdot \boldsymbol{\varphi} \ dx + \dfrac{2}{3}\left(\int_{\mathcal T} \boldsymbol{U}(\boldsymbol{V}^{(2)}) \cdot \boldsymbol{\varphi} \ dx + \Delta t  \ \text{RHS}(\boldsymbol{V}^{(2)},\omega,\boldsymbol{\varphi}) \right),
\end{aligned}
\end{equation}
where $\text{RHS}(\boldsymbol{V}^n,\omega,\boldsymbol{\varphi})$ is the spatial operator in (\ref{scheme2d}) or \eqref{scheme:ripaiso2d}. The caveat is that Newton iteration is necessary for solving a nonlinear system of equations at each inner stage of Runge-Kutta time stepping since the solutions are in terms of the equilibrium variables $\boldsymbol{V}$ which are nonlinear with respect to $\boldsymbol{U}$.
Taking Euler forward as an example, see the first equation in (\ref{3rk}),
\begin{equation}\label{1rk}
 \int_{\mathcal {T}} \boldsymbol{U}(\boldsymbol{V}^{n+1})\cdot \boldsymbol{\varphi} \ dx = \mathcal{L}_{\mathcal {T}}(\boldsymbol{V}^{n},\boldsymbol{\varphi})
\end{equation}
with
\begin{equation}\label{res}
 \mathcal{L}_{\mathcal {T}}(\boldsymbol{V}^{n},\boldsymbol{\varphi}) = \int_{\mathcal {T}} \boldsymbol{U}(\boldsymbol{V}^{n}) \cdot \boldsymbol{\varphi} \ dx + \Delta t  \ \text{RHS}(\boldsymbol{V}^n,\omega,\boldsymbol{\varphi}).
\end{equation}
We denote $\mathcal{L}_{\mathcal {T}}^{[l]} = \mathcal{L}_{\mathcal {T}}(\boldsymbol{V}^{n},{\varphi}_l \boldsymbol {e}_l)$ as the right hand side of equations satisfied by the $l$-th equilibrium variable to be solved, here 
$\boldsymbol {e}_l:=(0,\dots,0,\mathop{1}\limits_{\mathop{\uparrow}\limits_{l\text{-th}}},0,\dots,0)^T$
represents the unit vector. 
For Euler systems, the momentum vector $\rho \boldsymbol u$ defined in equilibrium variables can be acquired directly from the second and third components of (\ref{1rk}). The coefficients of the polynomial $K^{n+1}, \varepsilon^{n+1} \in \mathcal{W}_h^k$ are obtained by iteratively solving the nonlinear systems
\begin{equation}\label{newiter}
\begin{aligned}
& G_1(K^1,\dots,K^l,\varepsilon^1,\dots,\varepsilon^l)=\displaystyle{\int_{\mathcal {T}} \rho\left(\sum _{i=1}^l K^i \varphi_1^i,\sum _{i=1}^l \varepsilon^i \varphi_4^i,\rho \boldsymbol u^{n+1},\omega \right) {\varphi}_1 \ dx} - \mathcal{L}_{\mathcal {T}}^{[1]}=0,\\
& G_2(K^1,\dots,K^l,\varepsilon^1,\dots,\varepsilon^l)=\displaystyle{\int_{\mathcal {T}} E\left(\sum _{i=1}^l K^i \varphi_1^i,\sum _{i=1}^l \varepsilon^i \varphi_4^i,\rho \boldsymbol u^{n+1},\omega \right) {\varphi}_4 \ dx} - \mathcal{L}_{\mathcal {T}}^{[4]}=0.
\end{aligned}
\end{equation}
Here $l$ denotes the number of polynomial basis functions of degree $k$, and $l=k+1$ in one-dimensional and $(k+1)(k+2)/2$ in two-dimensional. The same is true for Ripa model.
For more details and a good initial guess from which Newton iteration starts, we refer to \cite{zhang2023moving} for a thorough discussion.

Slope limiters might be needed for the DG methods to control spurious numerical oscillations when containing discontinuous solutions. In our experiments, we adopt total variation bounded (TVB) limiter \cite{cockburn1989tvb3,cockburn1998runge5}.
In order not to destroy the moving equilibrium state in a one-dimensional case and hydrostatic equilibria in two dimensions, we apply the limiter procedure on the local characteristic fields of the equilibrium function $\boldsymbol {V}(\boldsymbol {x})$ with the same modification as \cite{zhang2023moving}.
Specifically, the TVB limiter procedure involves two parts. First, we check whether the limiter is needed in each cell based on the equilibrium variables $\boldsymbol{V}$ thanks to the equilibrium variables in the DG piecewise polynomial space at each time step.  Then, we apply this limiter for those  `trouble' cells on the polynomials $\boldsymbol{V}(\boldsymbol{x})$. When the steady state $\boldsymbol{V}=\text{constant}$ is reached, it indicates no limiting is needed. So the limiting procedure will not destroy the well-balanced property.

\begin{remark}
Minor modification exists for the Ripa model. Taking one-dimensional as an instance, the Jacobian ${\boldsymbol A}_{j+1/2}$ in \cite{cockburn1989tvb3} is replaced by
\begin{equation*}
\begin{aligned}
{\boldsymbol A}_{j+1/2}=\left[  \dfrac{\partial \boldsymbol V}{\partial\boldsymbol W}
                                \dfrac{\partial \boldsymbol F}{\partial\boldsymbol W}
                          \left(\dfrac{\partial \boldsymbol V}{\partial\boldsymbol W}\right)^{-1} \right]_{\boldsymbol U= \boldsymbol U_{j+1/2}},
\end{aligned}
\end{equation*}
where $\boldsymbol W = (h,hu,\theta)^T$ and $\boldsymbol U= (h,hu,h\theta)^T$ is our conservative variable. The average state $\boldsymbol {U}_{j+1/2} $ is chosen as the simple mean of the left and right values. The two-dimensional case can be dealt with in a similar approach.
\end{remark}

\begin{remark}
Despite the scheme in \cite{mantri2021well} being still in terms of equilibrium variables, a linear system of equations for $\boldsymbol V_t$ obtained by chain rule yields the loss of mass conservation. In our proposed approach, we can preserve the conservation of the hyperbolic balance laws. The formulation of the target equation itself does not change and we still solve the conservative form of the problem, leading to the locally nonlinear iteration in each cell. Moreover, in contrast to the scheme \cite{xing2014exactly}, our method is easily accomplished without the reference equilibrium state recovery and a special source term approximation.
\end{remark}


\subsection{Well-balanced property}
We outline the well-balanced property for the semi-discrete DG method in the following theorem.
\begin{theorem}\label{wbp_semi}
The semi-discrete DG method (\ref{scheme2d}) with hydrostatic reconstruction numerical flux (\ref{flux2d}) is exact for one-dimensional moving equilibria and two-dimensional hydrostatic equilibrium states.
\end{theorem}


\begin{proof}
We suppose that the initial values are in an exact equilibrium state, that is $\boldsymbol{V}(\boldsymbol{x}) = \text{constant} := c$ with $\boldsymbol{x}\in \mathbb{R}^d, d=1,2$. Here we comment that for a two-dimensional isentropic hydrostatic equilibrium state, (\ref{nl_solver2d}) reduces to $\varepsilon = h+\phi$, and the equilibrium variables remain constant. The same is true for the still water equilibrium state. It is straightforward that the approximation denoted by $\boldsymbol {V}(\boldsymbol{x})$ again is constant and equals $c$ as well.
From the updated boundary values (\ref{hy_up}), we have $\boldsymbol U^{*,int_{\mathcal T}} = \boldsymbol U^{*,ext_{\mathcal T}}$. Plugging back into (\ref{flux2d}) together with the consistent property of the numerical flux, we get
\begin{equation}\label{fwb}
\begin{aligned}
 \widehat{\boldsymbol{F}}^{*}  &= \widehat{\boldsymbol{F}}({\boldsymbol{U}}^{*,int_{\mathcal T}},{\boldsymbol{U}}^{*,int_{\mathcal T}},\boldsymbol{n}) - \boldsymbol{F}({\boldsymbol{U}}^{*,int_{\mathcal T}})\cdot \boldsymbol{n} + \boldsymbol{F}({\boldsymbol{U}}^{int_{\mathcal T}})\cdot \boldsymbol{n} = \boldsymbol F(\boldsymbol U^{int_{\mathcal T}})\cdot \boldsymbol{n}.
\end{aligned}
\end{equation}
For the stationary shocks separating the smooth regions are all located at the cell boundaries, the same results $ \widehat{\boldsymbol{F}}^{*} = \boldsymbol F(\boldsymbol U^{int_{\mathcal T}})\cdot \boldsymbol{n}$ can be obtained thanks to the exact Riemann solver by Roe flux or HLLC flux.
Hence we have
\begin{equation*}
\begin{aligned}
\text{RHS} &= \int_{\mathcal T}\boldsymbol{F}(\boldsymbol{U}(\boldsymbol{V})): \nabla\boldsymbol{\varphi} \ d\boldsymbol{x} -
  \int_{\partial \mathcal T} \widehat{\boldsymbol{F}}^{*} \cdot\boldsymbol{\varphi}^{int_{\mathcal T}} \ ds +\int_{\mathcal T}\boldsymbol r(\boldsymbol{U}(\boldsymbol{V}),\omega) \cdot \boldsymbol{\varphi} \ d\boldsymbol{x} \\
& =  \int_{\mathcal {T}}\boldsymbol{F}(\boldsymbol{U}(\boldsymbol{V})): \nabla\boldsymbol{\varphi} \ d\boldsymbol{x} -
\int_{\partial \mathcal T} \left(\boldsymbol F(\boldsymbol U^{int_{\mathcal T}})\cdot \boldsymbol{n}\right) \cdot \boldsymbol{\varphi}^{int_{\mathcal T}} \ ds
+ \int_{\mathcal {T}}\boldsymbol r(\boldsymbol{U}(\boldsymbol{V}),\omega)\cdot \boldsymbol{\varphi} \ d\boldsymbol{x} \\
& = -\int_{\mathcal {T}} (\nabla \cdot\boldsymbol{F}(\boldsymbol{U}(\boldsymbol{V}))) \cdot \boldsymbol{\varphi} \ d\boldsymbol{x}+\int_{I_j}\boldsymbol r(\boldsymbol{U}(\boldsymbol{V}),\omega)\cdot \boldsymbol{\varphi} \ d\boldsymbol{x} \\
& = -\int_{\mathcal {T}} \left(\nabla \cdot \boldsymbol{F}(\boldsymbol{U}(\boldsymbol{V})) - \boldsymbol r(\boldsymbol{U}(\boldsymbol{V}),\omega)\right)\cdot \boldsymbol{\varphi} \ d\boldsymbol{x}=0.
\end{aligned}
\end{equation*}
The last equality follows from the equilibrium state solution.  This completes the proof of well-balanced property.
\end{proof}

\begin{theorem}
For the Ripa model, the semi-discrete DG method (\ref{scheme:ripaiso2d})-(\ref{flux:ripaiso2d}) is exact for the isobaric equilibrium state (\ref{isobaric:Ripa}).
\end{theorem}

\begin{proof}
Assuming that the initial values are in an exact equilibrium state (\ref{isobaric:Ripa}), the velocity $\boldsymbol u$ vanishes.
From the definition \eqref{delta}, we have $\delta=0$. Hence the numerical flux \eqref{flux:ripaiso2d} reduces to
\begin{equation*}
\widehat{\boldsymbol{F}}^{\text{mod}} = \dfrac{1}{2}\left(\boldsymbol  {F}(\boldsymbol U^{int_{\mathcal T}})\cdot \boldsymbol{n}  + \boldsymbol {F}( \boldsymbol U^{ext_{\mathcal T}}) \cdot \boldsymbol{n}\right),
\end{equation*}
and owing to the constant bottom function in an equilibrium state, we have
\begin{equation}\label{rhs_iso}
\begin{aligned}
\text{RHS}^{\text{iso}}
  = \int_{\mathcal T}\boldsymbol{F}(\boldsymbol{U}(\boldsymbol{V})): \nabla\boldsymbol{\varphi} \ d\boldsymbol{x} -
  \int_{\partial \mathcal T} \dfrac{1}{2}\left(\boldsymbol  {F}(\boldsymbol U^{int_{\mathcal T}})\cdot \boldsymbol{n}  + \boldsymbol {F}( \boldsymbol U^{ext_{\mathcal T}}) \cdot \boldsymbol{n}\right) \cdot \boldsymbol{\varphi}^{int_{\mathcal T}} \ ds.
\end{aligned}
\end{equation}
We can rewrite \eqref{rhs_iso} to be the following four parts
$$\text{RHS}^{\text{iso}} = \sum_{l=1}^4\text{RHS}^{\text{iso},[l]} :=  \sum_{l=1}^4\text{RHS}^{\text{iso}}(\boldsymbol{V},\omega,{\varphi}_l \boldsymbol {e}_l).$$
It is trivial that the first and fourth components of $\text{RHS}^{\text{iso}}$ are equal to zero. We denote the outward unit normal vector $\boldsymbol{n}=(n_x,n_y).$ Using the fact that $p(\boldsymbol x) = \text{constant} := c$ and the approximation denoted by ${p}(\boldsymbol x)$ again is constant and equals $c$ as well, we can obtain
\begin{equation*}
\begin{aligned}
  \text{RHS}^{\text{iso},[2]}
  &=\int_{\mathcal T}\boldsymbol{F}^{[2]}(\boldsymbol{U}(\boldsymbol{V})): \nabla{\varphi}_2 \ d\boldsymbol{x} -
  \int_{\partial \mathcal T} \dfrac{1}{2}\left(\boldsymbol  {F}^{[2]}(\boldsymbol U^{int_{\mathcal T}})\cdot \boldsymbol{n}  + \boldsymbol {F}^{[2]}( \boldsymbol U^{ext_{\mathcal T}}) \cdot \boldsymbol{n}\right){\varphi}_2^{int_{\mathcal T}} \ ds\\
  &= \int_{\mathcal T} p({\varphi_2})_x \ d\boldsymbol{x} - \int_{\partial \mathcal T} \dfrac{1}{2}(p^{int_{\mathcal T}} n_x + p^{ext_{\mathcal T}} n_x){\varphi_2}^{int_{\mathcal T}} \ ds \\
  &= c\left(\int_{\mathcal T} (\varphi_2)_x \ d\boldsymbol{x} - \int_{\partial \mathcal T} n_x {\varphi}_2^{int_{\mathcal T}} \ ds\right) = 0.
\end{aligned}
\end{equation*}
The superscript $\mbox{}^{[2]}$ means the second component of the fluxes. The second equality is due to the steady state $\boldsymbol u=0$, and the last equality follows from the divergence theorem. Similarly,
\begin{equation*}
\begin{aligned}
  \text{RHS}^{\text{iso},[3]}
  &= \int_{\mathcal T} p (\varphi_3)_y \ d\boldsymbol{x} - \int_{\partial \mathcal T} \dfrac{1}{2}(p^{int_{\mathcal T}} n_y + p^{ext_{\mathcal T}} n_y){\varphi}_3^{int_{\mathcal T}} \ ds \\
  &= c\left(\int_{\mathcal T} (\varphi_3)_y \ d\boldsymbol{x} - \int_{\partial \mathcal T} n_y {\varphi}_3^{int_{\mathcal T}} \ ds\right) = 0.
\end{aligned}
\end{equation*}
 This completes the proof of well-balanced property.
\end{proof}

The same property holds for the fully-discrete DG method. Some trivial induction gives rise to the following property, and we omit the proof here.
\begin{theorem}\label{wbp_fully}
The fully-discrete DG method,  space discretization (\ref{scheme2d})-(\ref{flux2d})  together with the third-order SSP-RK time discretization (\ref{3rk}), is exact for one-dimensional moving equilibria and two-dimensional hydrostatic equilibrium state.
\end{theorem}

\begin{theorem}
For the Ripa model, the fully-discrete DG method, space discretization (\ref{scheme:ripaiso2d})-(\ref{flux:ripaiso2d}) together with the third-order SSP-RK time discretization (\ref{3rk}), is exact for the isobaric equilibrium state (\ref{isobaric:Ripa}).
\end{theorem}

\section{Numerical examples}\label{se:nu}
This section implements our proposed well-balanced discontinuous Galerkin method for the one- and two-dimensional Ripa model and Euler equations with gravitation.
A quadratic ($P^2$) polynomial basis is adopted for our spatial discretization in all numerical experiments. The time step constraint is acquired by
  $$\Delta t = \text{CFL}\frac{ h_m }{\max_{l,\boldsymbol U}{|\lambda_l(\boldsymbol U)|}},$$ where $h_m$ denotes the mesh size and $\lambda_l(\boldsymbol U)$ is the eigenvalues of the systems. The CFL number is taken as 0.1 for one and two dimensions. The integrals of fluxes and source terms were computed numerically using $(2k+1)^d$ Gauss-Legendre points, where $d$ is the dimensionality of the system, in order to accurately compute the integrals. The constant in the TVB slope limiter is set as 0, except for the accuracy tests in which no slope limiter is applied. Unless otherwise specified, in the Ripa model, the gravitation constant $g$ is taken as 9.812 $m/s^2$, and we use the moving equilibria preserving DG method \eqref{scheme2d}-\eqref{flux2d} for simulation.
As stated in \cite{zhang2023moving}, most of the extra cost is spent on solving the nonlinear equation systems, which are local in each cell.  The solution of these local nonlinear equations in a fully-discrete scheme (\ref{3rk}) is implemented by the Newton-Raphson method.

\subsection{Euler equations with gravity}
\subsubsection{One-dimensional tests}
\begin{example}{\bf Accuracy test}\label{euler:accuracy1D}
\end{example}
To test the high order accuracy of the developed well-balanced DG scheme for one-dimensional Euler equations with gravity, we consider the following stationary solution taken from \cite{chandrashekar2015second}
\begin{equation}\label{euler:inismoo}
 \rho(x,t)=\exp(-x), \quad u(x,t) = 0, \quad p(x,t) = (1+x)\exp(-x),
\end{equation}
on the computational domain $[0,1]$. The gravitational field is given by $\phi(x) = \frac{1}{2}x^2$, and the ratio of specific heats is $\gamma=1.4$. We run the simulations up to $t=0.1$, and calculate the $L^1$ errors of the numerical solution relative to the exact stationary solution. The errors and orders of accuracy for the conservative variables $\rho$, $\rho u$, $E$ and the equilibrium variables $K$, $\varepsilon$ are shown in Table \ref{euler:DGacc1D}, from which we can observe the excepted third-order accuracy.
\begin{table}[htb]
  \centering
  \caption{Example \ref{euler:accuracy1D}:  $L^1$ errors and numerical orders of accuracy with initial condition (\ref{euler:inismoo}), $t = 0.1$.}\label{euler:DGacc1D}
  \begin{tabular}{ c c c c c c c c c c c}
  \toprule
    \multirow{2}{*} {$nx$} & \multicolumn{2}{c}{$\rho$}&\multicolumn{2}{c}{$\rho u $}&\multicolumn{2}{c}{$ E$}& \multicolumn{2}{c}{$K$} &\multicolumn{2}{c}{$ \varepsilon $}\\
   \cmidrule(lr){2-3} \cmidrule(lr){4-5} \cmidrule(lr){6-7} \cmidrule(lr){8-9} \cmidrule(lr){10-11}
    ~ &$L^1$ error &order&$L^1$ error &order&$L^1$ error &order&$L^1$ error &order&$L^1$ error &order \\
   \midrule
   20 & 2.52E{-07} &     -- &  3.10E{-07} &  --  &  7.77E{-07} &  --  & 4.94E{-07} &     -- & 7.92E{-07} &     --\\
   40 & 3.14E{-08} &   3.00 &  3.92E{-08} & 2.98 &  9.70E{-08} & 3.00 & 6.19E{-08} &   3.00 & 9.97E{-08} &   2.99\\
   80 & 3.92E{-09} &   3.00 &  4.94E{-09} & 2.99 &  1.21E{-08} & 3.00 & 7.75E{-09} &   3.00 & 1.25E{-08} &   2.99\\
   160& 4.91E{-10} &   3.00 &  6.20E{-10} & 2.99 &  1.51E{-09} & 3.00 & 9.70E{-10} &   3.00 & 1.57E{-09} &   3.00\\
   320& 6.13E{-11} &   3.00 &  7.76E{-11} & 3.00 &  1.88E{-10} & 3.01 & 1.21E{-10} &   3.00 & 1.97E{-10} &   2.99\\
  \bottomrule
  \end{tabular}
\end{table}

\begin{example}{\bf Test for one-dimensional isentropic equilibrium}\label{euler:ise_1d}
\end{example}
To verify the well-balanced property, we consider the isentropic hydrostatic atmosphere studied in \cite{chandrashekar2015second} on the computational domain $[0,1]$. The steady state solutions with zero velocity are given by
\begin{equation}\label{euler:iseini_1d}
\begin{aligned}
  \rho(x) =  \left( 1-\frac{\gamma-1}{\gamma} \phi(x)\right)^{\frac{1}{\gamma-1}}, \ u(x)=0, \ p(x) = \left( 1-\frac{\gamma-1}{\gamma} \phi(x)\right)^{\frac{\gamma}{\gamma-1}},
\end{aligned}
\end{equation}
where $\gamma = 1.4$ and three different potential functions $\phi(x) = x, \frac{1}{2}x^2, \sin(2\pi x)$ are tested. The boundaries are treated as solid walls, and the final time is set as $t=2$. The $L^1$ and $L^{\infty}$ errors with 100 uniform cells for both the conservative and equilibrium variables are listed in Table \ref{euler:isewb_1d}. Errors in the order of machine accuracy can be observed under different gravitational fields.
\begin{table}[htb]
  \centering
  \caption{Example \ref{euler:ise_1d}: $L^1$ and $L^{\infty}$ errors for one-dimensional isentropic hydrostatic atmosphere.}\label{euler:isewb_1d}
  \begin{tabular}{ c c c c c c c }
   \toprule
    $\phi(x)$ &{error}&{$\rho$}&{$\rho u$}&{$E$}&{$K$}&{$\varepsilon$}\\
    \midrule
    \multirow{2}{*}{$x$}       &$L^1$        & 1.53E{-14} &  4.73E{-15} & 2.32E{-14} &  3.20E{-14} & 7.37E{-14} \\
                         ~     &$L^{\infty}$ & 4.75E{-14} &  1.20E{-14} & 5.31E{-14} &  9.78E{-14} & 2.13E{-13} \\
    \midrule
    \multirow{2}{*}{$ \frac{1}{2}x^2$}&$L^1$        & 2.22E{-14} &  6.35E{-15} & 2.74E{-14} &  2.99E{-14} & 7.07E{-14} \\
                                ~     &$L^{\infty}$ & 5.37E{-14} &  1.85E{-14} & 9.59E{-14} &  7.45E{-14} & 1.60E{-13} \\
    \midrule
    \multirow{2}{*}{$\sin(2\pi x)$} &$L^1$        & 2.89E{-14} &  1.14E{-14} & 3.28E{-14} &  3.33E{-14} & 8.49E{-14} \\
                              ~     &$L^{\infty}$ & 1.16E{-13} &  2.62E{-14} & 1.21E{-13} &  1.01E{-13} & 2.78E{-13} \\
    \bottomrule
  \end{tabular}
\end{table}

\begin{example}{\bf Well-balanced property test for the steady adiabatic flow}\label{euler:wb_adi}
\end{example}
In this example, we demonstrate the well-balanced property for the steady adiabatic flow, which consists of non-hydrostatic equilibria. We consider the linear gravitational field $\phi(x)=x$. Following the same setup in \cite{grosheintz2020well}, the initial conditions on a domain $[0,2]$ are given by
\begin{equation}\label{euler:adi_ini}
(\rho,u,p)(x,0) = (\rho_{eq}(x), u_{eq}(x), p_{eq}(x)),
\end{equation}
where $ (\rho_{eq}, u_{eq}, p_{eq})$ is the equilibrium defined by the point values
\begin{equation}\label{euler:adi_data}
(\rho_0,u_0,p_0) = (1, -M c_{s,0},1)
\end{equation}
at $x_0 = 0$. The ratio of specific heats is taken as $\gamma=5/3$ and $c_{s,0} = \gamma^{1/2}$ is the speed of sound at $x_0$. Three cases, including a hydrostatic ($M=0$), a subsonic ($M=0.01$), and a supersonic ($M=2.5$) equilibrium flow, are considered. Fig. \ref{euler:iniplot} plots the initial density, velocity, and pressure for different Mach numbers $M$. The values of the ghost points are equal to the initial conditions. We compute the solutions up to $t = 4$ for the hydrostatic and subsonic cases, and $t=1$ for the supersonic case. Table \ref{euler:adi_10021D} shows the $L^1$ and $L^{\infty}$ errors for all values of $M$ with 100 uniform cells. Both the conservative and the equilibrium variables are in equilibrium within machine precision.

\begin{figure}[htb!]
  \centering
  \subfigure[density $\rho$]{
  \centering
  \includegraphics[width=4.9cm,scale=1]{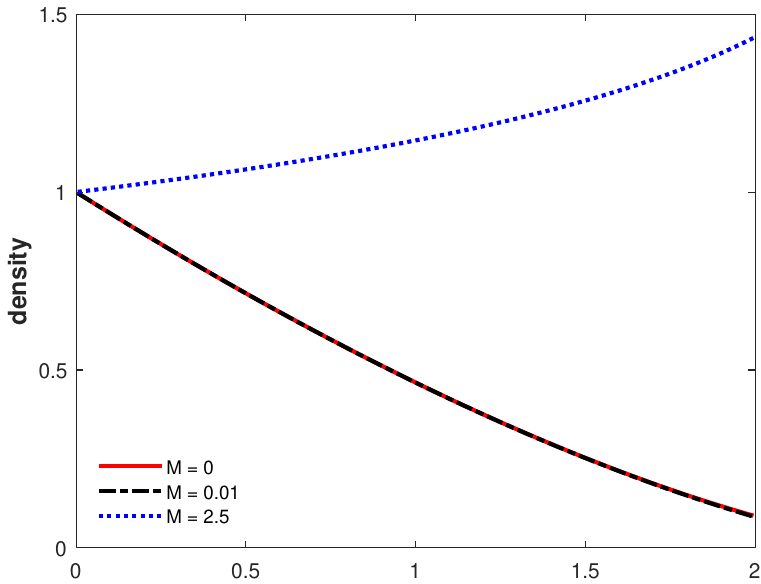}
  }
  \subfigure[velocity $u$]{
  \centering
  \includegraphics[width=4.9cm,scale=1]{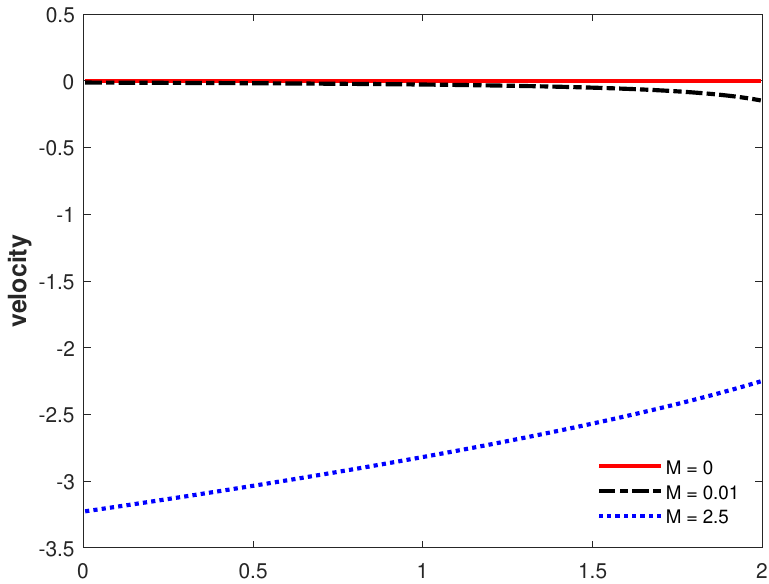}
  }
  \subfigure[pressure $p$]{
  \centering
  \includegraphics[width=4.9cm,scale=1]{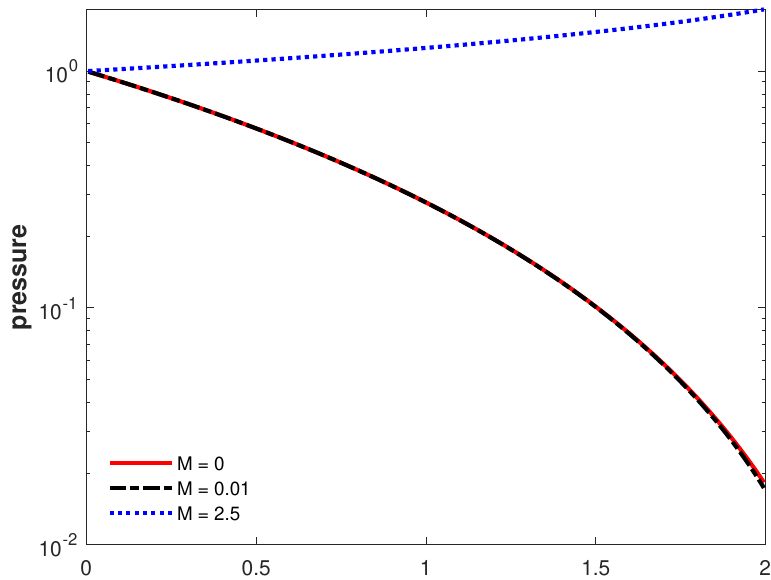}
  }
  \caption{Example \ref{euler:wb_adi}: From left to right: the initial density $\rho$, velocity $u$ and pressure $p$ computed by (\ref{euler:adi_ini}) for three values of $M$. }\label{euler:iniplot}
\end{figure}

\begin{table}[htb]
  \centering
  \caption{Example \ref{euler:wb_adi}: $L^1$ and $L^{\infty}$ errors for one-dimensional steady adiabatic flow.}\label{euler:adi_10021D}
  \begin{tabular}{ c c c c c c c }
   \toprule
    case &{error}&{$\rho$}&{$\rho u$}&{$E$}&{$K$}&{$\varepsilon$}\\
    \midrule
    \multirow{2}{*}{$M=0$} &$L^1$        & 3.37E{-14} &  1.29E{-14} & 3.06E{-14} &  1.07E{-13} & 1.29E{-13} \\
                         ~     &$L^{\infty}$ & 6.28E{-14} &  1.48E{-14} & 2.75E{-14} &  1.78E{-13} & 1.56E{-13} \\
    \midrule
    \multirow{2}{*}{$M=0.01$}&$L^1$        & 3.43E{-14} &  1.46E{-14} & 3.06E{-14} &  1.11E{-13} & 1.45E{-13} \\
                           ~     &$L^{\infty}$ & 6.14E{-14} &  1.52E{-14} & 2.69E{-14} &  1.02E{-13} & 1.28E{-13} \\
    \midrule
    \multirow{2}{*}{$M=2.5$} &$L^1$        & 2.38E{-13} &  3.21E{-13} & 5.72E{-13} &  1.84E{-13} & 8.73E{-13} \\
                           ~     &$L^{\infty}$ & 2.76E{-13} &  3.71E{-13} & 6.73E{-13} &  2.41E{-13} & 9.34E{-13} \\
    \bottomrule
  \end{tabular}
\end{table}

\begin{example}{\bf Small perturbation of the steady adiabatic flow}\label{euler:per_adi}
\end{example}
In this test, we investigate the capability to capture small perturbations of the equilibrium state (\ref{adi:euler}). Almost the same setup in Example \ref{euler:wb_adi}, we first impose a periodic velocity perturbation
\begin{equation}\label{euler:vp}
u(0,t) = A\sin(4\pi t)
\end{equation}
on the hydrostatic ($M=0$) equilibrium flow with a big amplitude  $A = 0.1$ and a small amplitude $A=10^{-6}$, respectively. The stopping time is $t=1.5$ before the waves propagate to the upper boundary. We calculate the results on a coarse mesh with 100 cells and a finer mesh of 2000 cells for comparison. The pressure perturbation and velocity for the small amplitude are presented in Fig. \ref{euler:persmall}.  The traditional third-order non-well-balanced scheme (abbreviated as `nwb') is also employed, and we plot the results in the same figure. Fig. \ref{euler:perbig} displays similar solutions for the big amplitude. Although the non-well-balanced scheme behaves well in the propagation of large perturbation waves, it is difficult to simulate the smaller disturbance due to the existence of gravity. In contrast, our well-balanced DG method has its advantages in capturing the small perturbation of the hydrostatic equilibrium state, and the solutions are free of spurious numerical oscillations near the discontinuity.

\begin{figure}[htb!]
  \centering
  \subfigure[pressure perturbation]{
  \centering
  \includegraphics[width=5.5cm,scale=1]{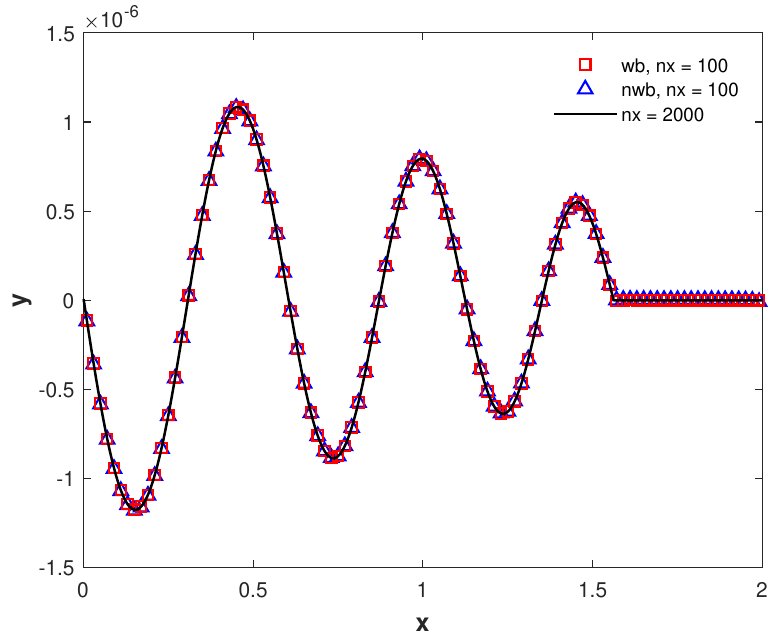}
  }
  \subfigure[velocity ]{
  \centering
  \includegraphics[width=5.5cm,scale=1]{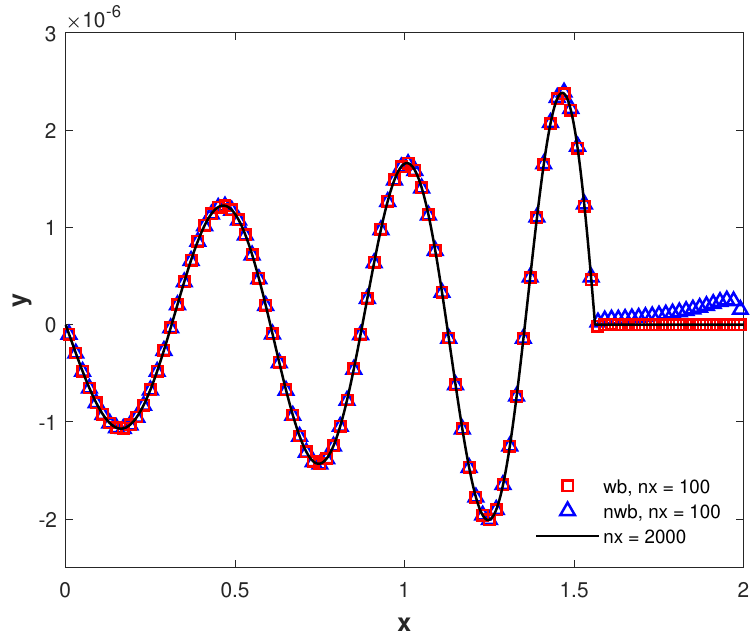}
  }
  \caption{Example \ref{euler:per_adi}: Pressure perturbation and velocity of the hydrostatic atmosphere ($M=0$) for the small amplitude wave propagation $A=10^{-6}$ in (\ref{euler:vp}), using the well-balanced scheme with 100 and 2000 cells, the non-well-balanced scheme with 100 cells for comparison. }\label{euler:persmall}
\end{figure}

\begin{figure}[htb!]
  \centering
  \subfigure[pressure perturbation]{
  \centering
  \includegraphics[width=5.5cm,scale=1]{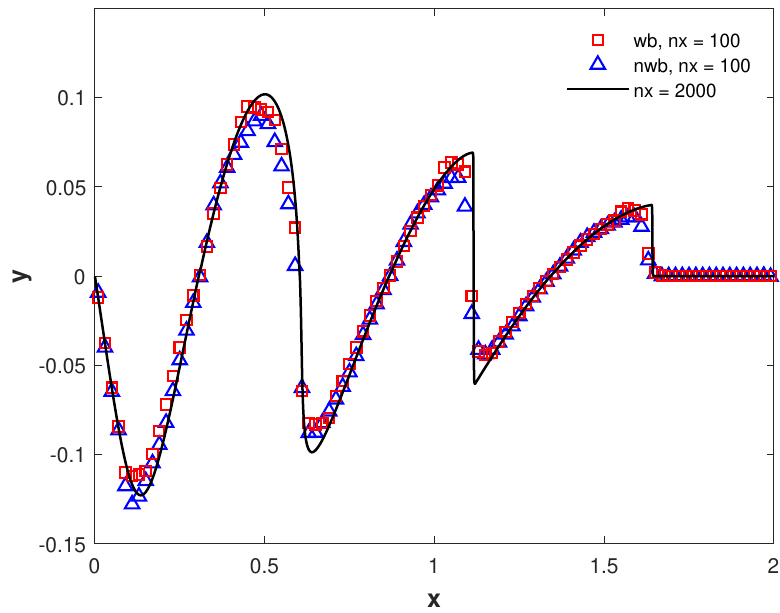}
  }
  \subfigure[velocity]{
  \centering
  \includegraphics[width=5.5cm,scale=1]{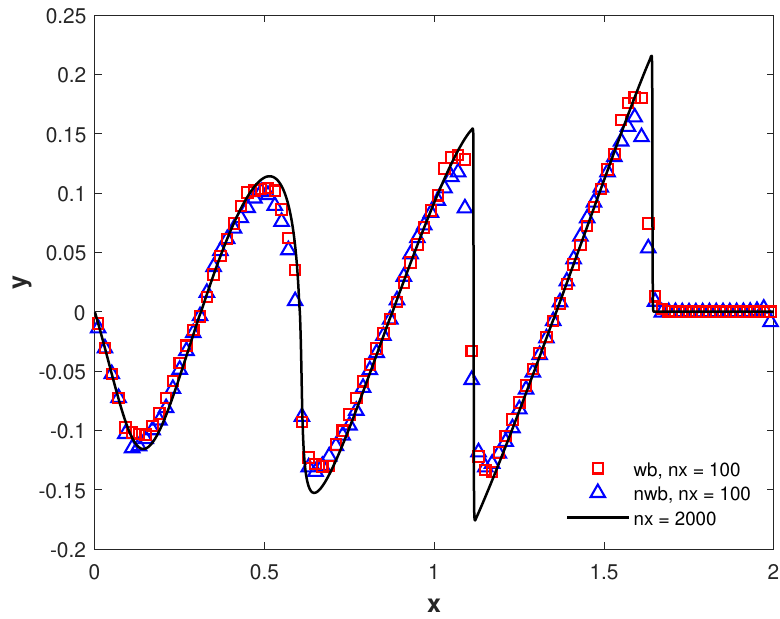}
  }
  \caption{Example \ref{euler:per_adi}: Pressure perturbation and velocity of the hydrostatic atmosphere ($M=0$) for the big amplitude wave propagation $A=0.1$ in (\ref{euler:vp}), using the well-balanced scheme with 100 and 2000 cells, the non-well-balanced scheme with 100 cells for comparison. }\label{euler:perbig}
\end{figure}

Next, we add a perturbation to the initial pressure
\begin{equation}\label{euler:adi_peradd}
(\rho,u,p)(x,0) = (\rho_{eq}(x), u_{eq}(x), p_{eq}(x)+A\exp(-100(x-\bar x)^2))
\end{equation}
with the location of the disturbance
\begin{equation}\label{euler:xloc}
\bar x =\left\{\begin{array}{lll}
    1.0,&\text{for} \  M=0, \\
    1.1,&\text{for} \  M=0.01,\\
    1.5,&\text{for} \  M=2.5.
    \end{array}\right.
\end{equation}
Here $ (\rho_{eq}, u_{eq}, p_{eq})$ is the same as (\ref{euler:adi_ini}) and a small amplitude $A=10^{-6}$ is considered. We take the stopping time $t=0.45$ for the hydrostatic and subsonic cases, and $t=0.25$ for the supersonic case. Figs. \ref{euler:persmall_hydro}, \ref{euler:persmall_sub} and \ref{euler:persmall_sup} demonstrate the three different numerical solutions by our developed method on a coarse mesh of 50 uniform cells. For comparison, we also run the same tests by the third order hydrostatically well-balanced DG scheme \cite{li2018welldge} (abbreviated as `wb$\_$Li') and the non-well-balanced scheme, respectively.
As expected, it shows a failed simulation of the non-well-balanced scheme for all Mach numbers. The hydrostatically well-balanced method only consists of our proposed method for the hydrostatic case $M=0$. In contrast, our well-balanced DG method gives well-resolved solutions that agree with the reference solutions under a relatively coarse mesh.
Furthermore, to show the performance for problems containing discontinuity, we consider a large amplitude  $A = 1$. The stopping time is chosen the same as before. We only show the pressure and velocity as well as their perturbation with 100 uniform cells for the subsonic atmosphere ($M=0.01$) in Fig. \ref{euler:perbig_sub}. Herein, the reference solutions are obtained by
our scheme with much refined 2000 cells. It demonstrates that our scheme works well for non-hydrostatic equilibrium states and provides oscillatory-free solutions.

\begin{figure}[htb!]
  \centering
  \subfigure[pressure perturbation]{
  \centering
  \includegraphics[width=5.5cm,scale=1]{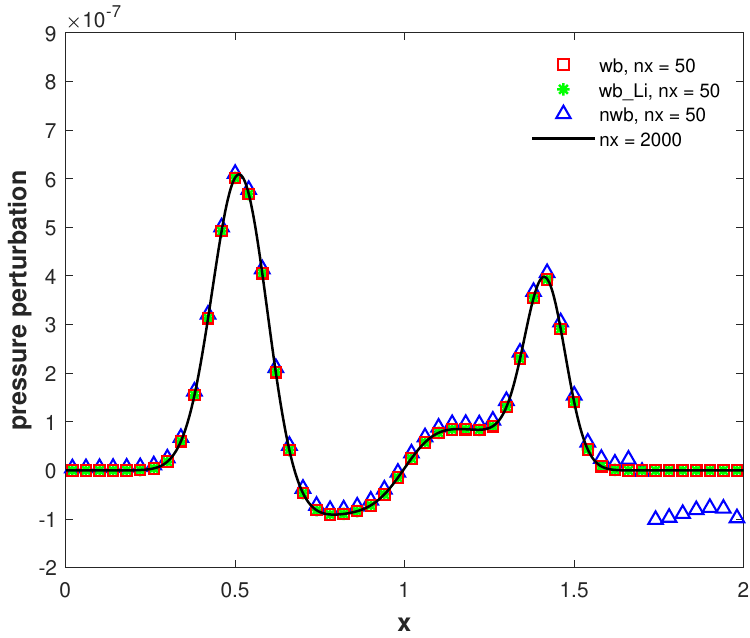}
  }
  \subfigure[velocity perturbation]{
  \centering
  \includegraphics[width=5.5cm,scale=1]{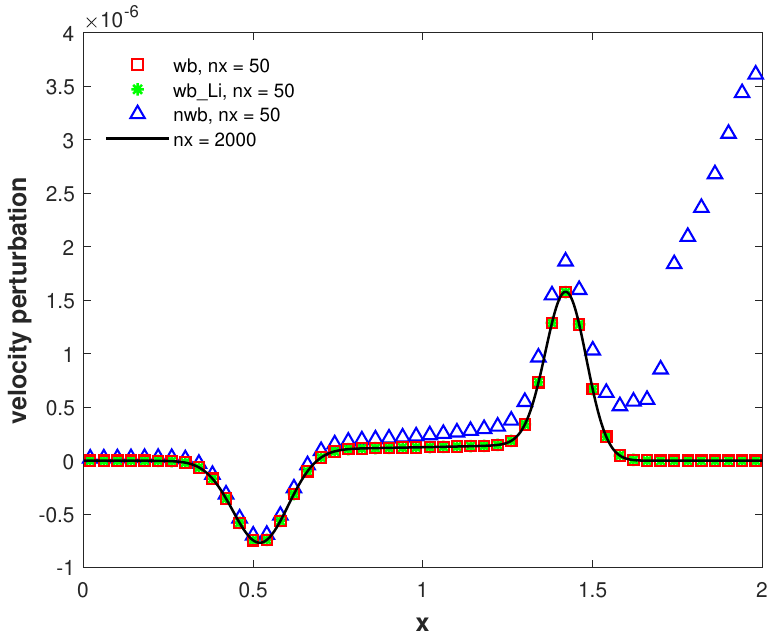}
  }
  \caption{Example \ref{euler:per_adi}: Pressure perturbation and velocity perturbation of the hydrostatic atmosphere ($M=0$) for the small amplitude wave propagation $A=10^{-6}$  in (\ref{euler:adi_peradd}). }\label{euler:persmall_hydro}
\end{figure}

\begin{figure}[htb!]
  \centering
  \subfigure[pressure perturbation]{
  \centering
  \includegraphics[width=5.5cm,scale=1]{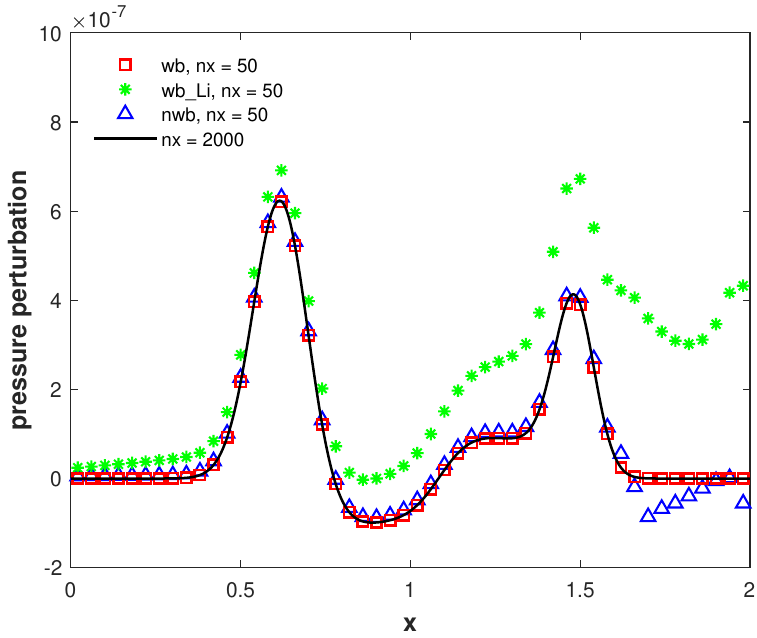}
  }
  \subfigure[velocity perturbation]{
  \centering
  \includegraphics[width=5.5cm,scale=1]{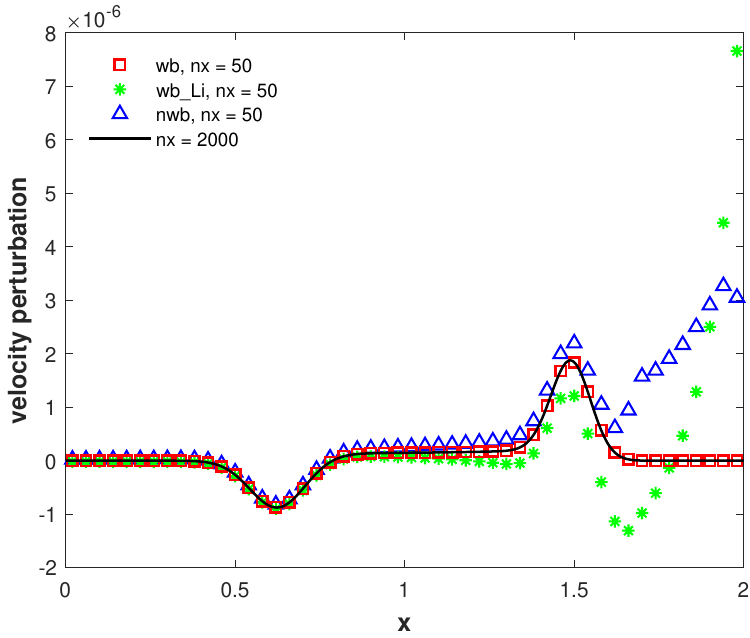}
  }
  \caption{Example \ref{euler:per_adi}: Pressure perturbation and velocity perturbation of the subsonic atmosphere ($M=0.01$) for the small amplitude wave propagation $A=10^{-6}$ in (\ref{euler:adi_peradd}). }\label{euler:persmall_sub}
\end{figure}

\begin{figure}[htb!]
  \centering
  \subfigure[pressure perturbation]{
  \centering
  \includegraphics[width=5.5cm,scale=1]{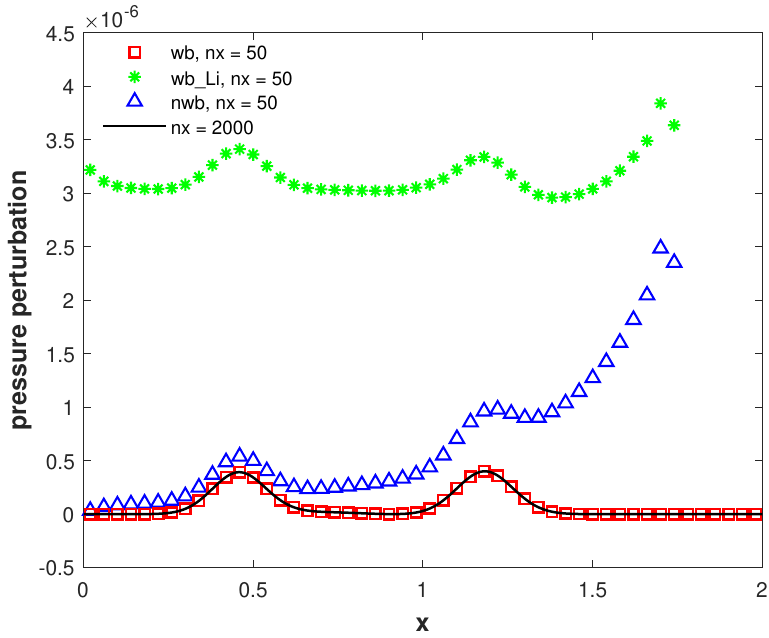}
  }
  \subfigure[velocity perturbation]{
  \centering
  \includegraphics[width=5.5cm,scale=1]{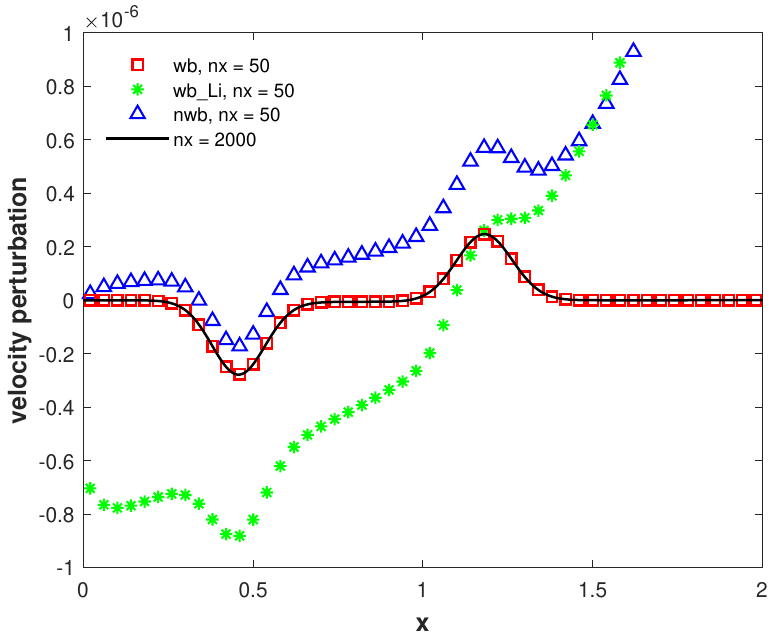}
  }
  \caption{Example \ref{euler:per_adi}: Pressure perturbation and velocity perturbation of the supersonic atmosphere ($M=2.5$) for the small amplitude wave propagation $A=10^{-6}$ in (\ref{euler:adi_peradd}). }\label{euler:persmall_sup}
\end{figure}

\begin{figure}[htb!]
  \centering
  \subfigure[pressure perturbation]{
  \centering
  \includegraphics[width=5.5cm,scale=1]{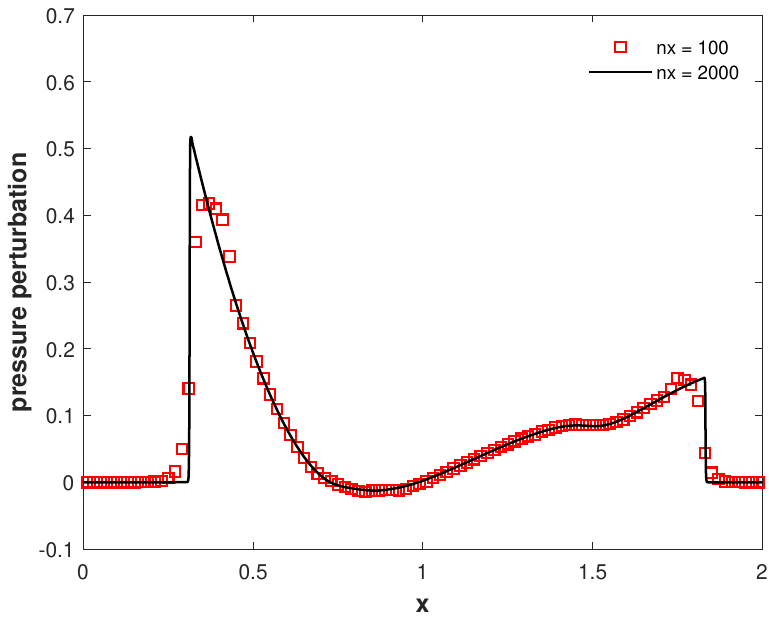}
  }
  \subfigure[velocity perturbation]{
  \centering
  \includegraphics[width=5.5cm,scale=1]{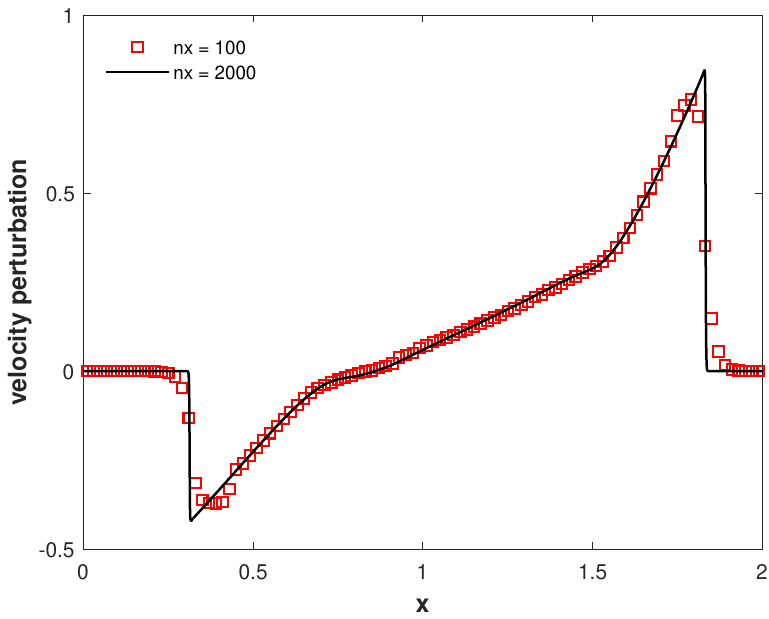}
  }
  \subfigure[pressure]{
  \centering
  \includegraphics[width=5.5cm,scale=1]{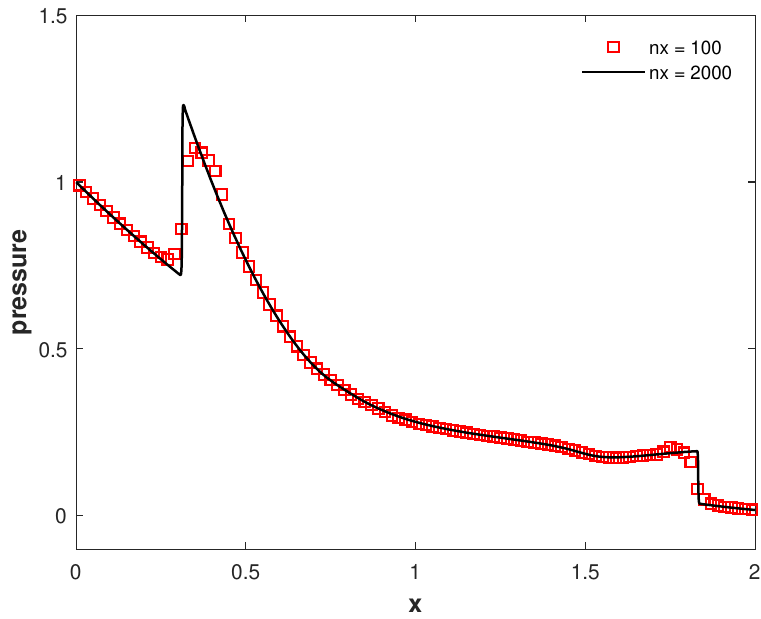}
  }
  \subfigure[velocity]{
  \centering
  \includegraphics[width=5.5cm,scale=1]{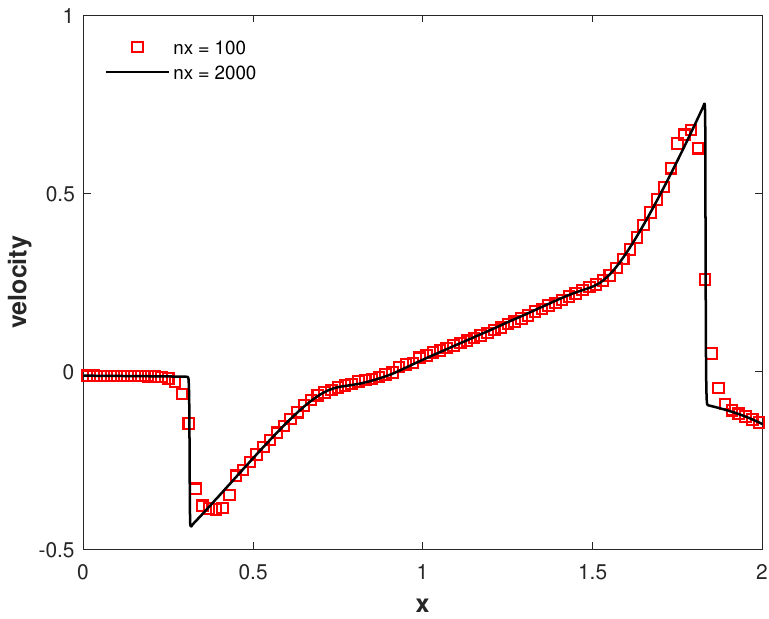}
  }
  \caption{Example \ref{euler:per_adi}: Small perturbation of the subsonic atmosphere ($M=0.01$) for the big amplitude wave propagation $A=1$ in (\ref{euler:adi_peradd}), using 100 and 2000 grid cells. Top: results of pressure and velocity perturbation, bottom: results of pressure and velocity.}\label{euler:perbig_sub}
\end{figure}

\begin{example}{\bf Well-balanced property test for a discontinuous equilibrium}\label{euler:wb_aditr}
\end{example}
To check the well-balanced property for a discontinuous equilibrium, we consider a subsonic flow in the lower half and supersonic in the upper half. They are joined by a stationary shock under the gravitational field $\phi(x)=\frac{1}{2}x^2$. Let the computational domain be $[0,2]$ and the shock is located in the middle of the domain $x_0=1$. The adiabatic index is $\gamma=5/3$. First, we define the pre-shock values
\begin{equation}\label{euler:trans_ini1}
\rho_{0,1} = 1, c_{0,1}^2 = \gamma, u_{0,1} = -M c_{0,1}
\end{equation}
with the Mach number $M = 2.5$. Then, by Rankine-Hugoniot jump conditions for a stationary shock \cite{LANDAU1987313}, we have the post-shock values
\begin{equation}\label{euler:trans_ini2}
\rho_{0,2} = \rho_{0,1}\dfrac{(\gamma+1)M^2}{(\gamma-1)M^2+2}, p_{0,2} = p_{0,1}\left(\dfrac{2\gamma M^2}{\gamma+1} - \dfrac{\gamma-1}{\gamma+1}\right), u_{0,2} = \dfrac{ \rho_{0,1}}{ \rho_{0,2}} u_{0,1}.
\end{equation}
The initial conditions in the upper $(k=1)$ and lower $(k=2)$ halves are given by
\begin{equation}\label{euler:aditr_ini}
(\rho,u,p)(x,0) = (\rho_{eq,k}(x), u_{eq,k}(x), p_{eq,k}(x)),
\end{equation}
where $ (\rho_{eq,k}, u_{eq,k}, p_{eq,k})$ is the equilibrium defined by $\rho_{0,k},u_{0,k},p_{0,k}$ at $x_0=1$. In our computation, we put the shock exactly at the cell boundary, and the Roe flux is needed for solving the approximate Riemann problem. We run the solution up to $t=1$ with 100 uniform cells, and list the $L^1$ and $L^{\infty}$ errors at double precision in Table \ref{euler:transwb1D}. For comparison, not only our proposed scheme, but also the non-well-balanced scheme and hydrostatically well-balanced scheme \cite{li2018welldge} together with piecewise constant $(P^0)$ and quadratic $(P^2)$ polynomial basis are applied for our simulation. We plot the corresponding density $\rho$, velocity $u$, pressure $p$ in Fig. \ref{euler:transcom1}. The difference between the numerical solutions at $t=1$ and the initial conditions are shown in Fig. \ref{euler:transcom}. For the `nwb' method and `wb$\_$Li' method, we can observe that first-order schemes fail to compute this discontinuous equilibrium state thoroughly, and third-order schemes accrue large errors near the stationary shock. In contrast, our well-balanced method can indeed maintain the moving equilibria with a round-off error.

\begin{table}[!ht]
  \centering
  \caption{Example \ref{euler:wb_aditr}: $L^1$ and $L^{\infty}$ errors for one-dimensional discontinuous equilibrium flow.}\label{euler:transwb1D}
  \begin{tabular}{ c c c c c c }
   \toprule
    {error}&{$\rho$}&{$\rho u$}&{$E$}&{$K$}&{$\varepsilon$}\\
    \midrule
    $L^1$        & 2.46E{-13} &  1.96E{-13}  &  1.36E{-13} & 1.63E{-13} & 5.21E{-13}\\
    $L^{\infty}$ & 9.33E{-12} &  3.50E{-12}  &  3.97E{-12} & 1.89E{-12} & 5.01E{-13}\\
    \bottomrule
  \end{tabular}
\end{table}

\begin{figure}[htb!]
  \centering
  \subfigure[ density $\rho$ ]{
  \centering
     \includegraphics[width= 4.8cm,scale=1]{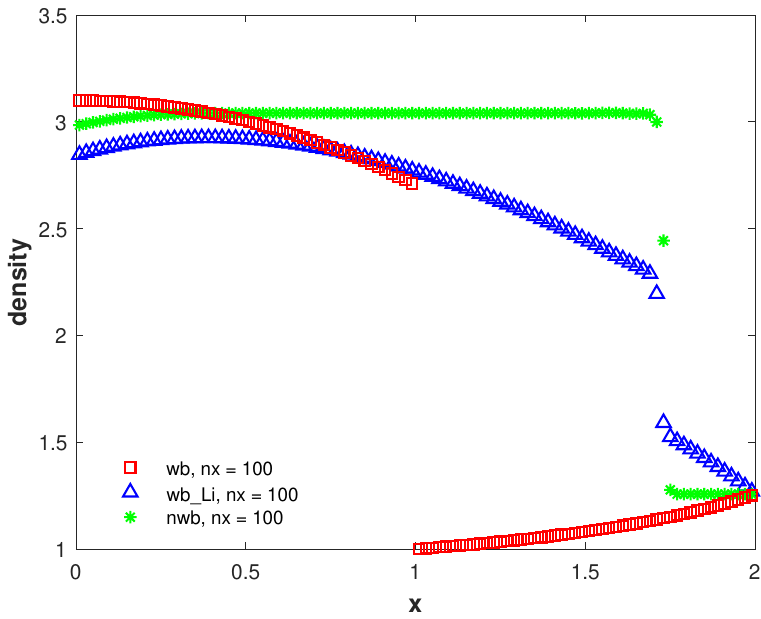}
  }
  \subfigure[velocity $u$]{
  \centering
     \includegraphics[width= 4.8cm,scale=1]{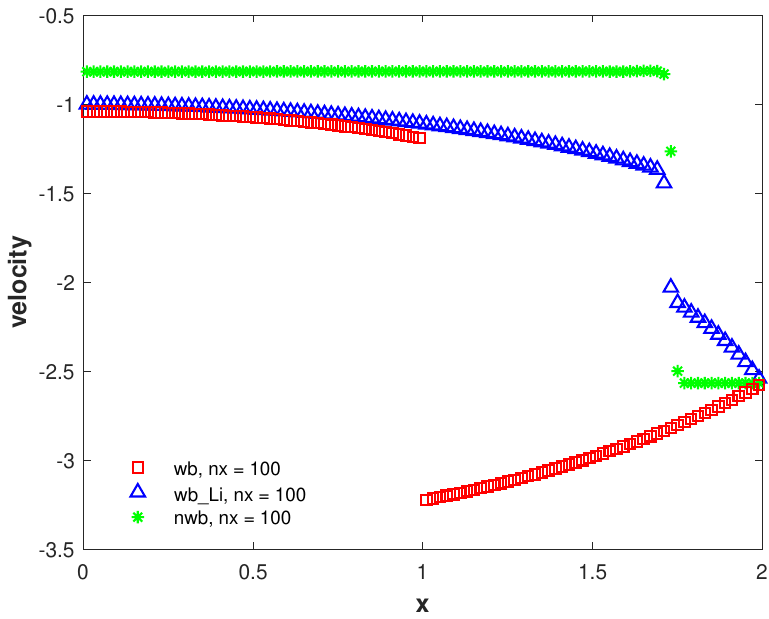}
  }
  \subfigure[pressure $p$]{
  \centering
     \includegraphics[width= 4.8cm,scale=1]{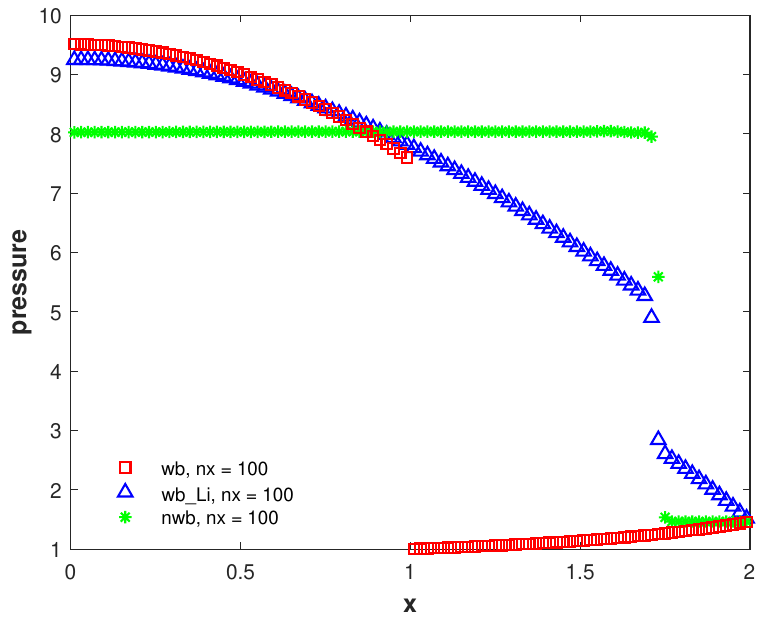}
  }
  \subfigure[ density $\rho$ ]{
  \centering
     \includegraphics[width= 4.8cm,scale=1]{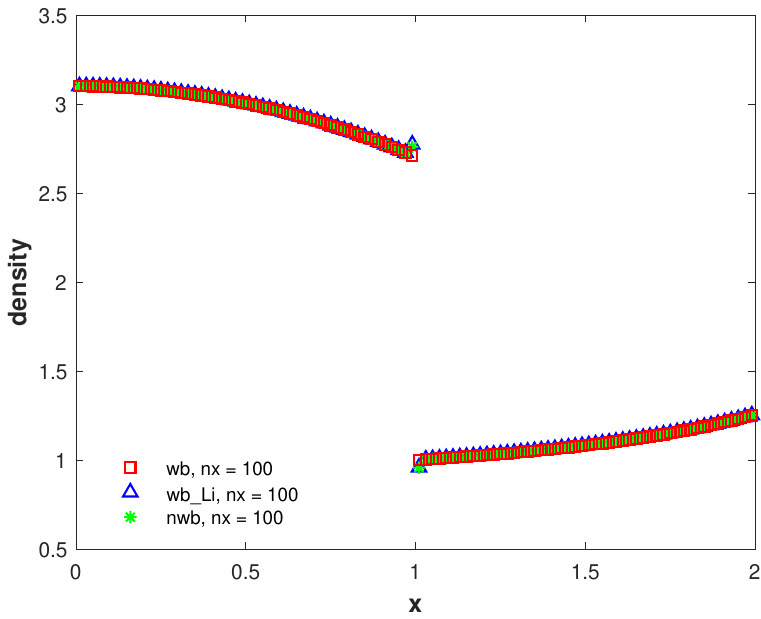}
  }
  \subfigure[velocity $u$]{
  \centering
     \includegraphics[width= 4.8cm,scale=1]{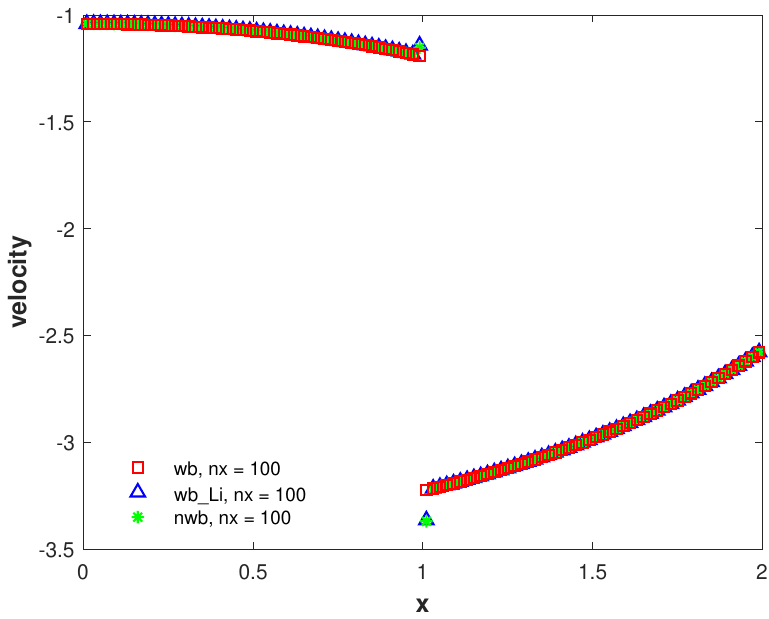}
  }
  \subfigure[pressure $p$]{
  \centering
     \includegraphics[width= 4.8cm,scale=1]{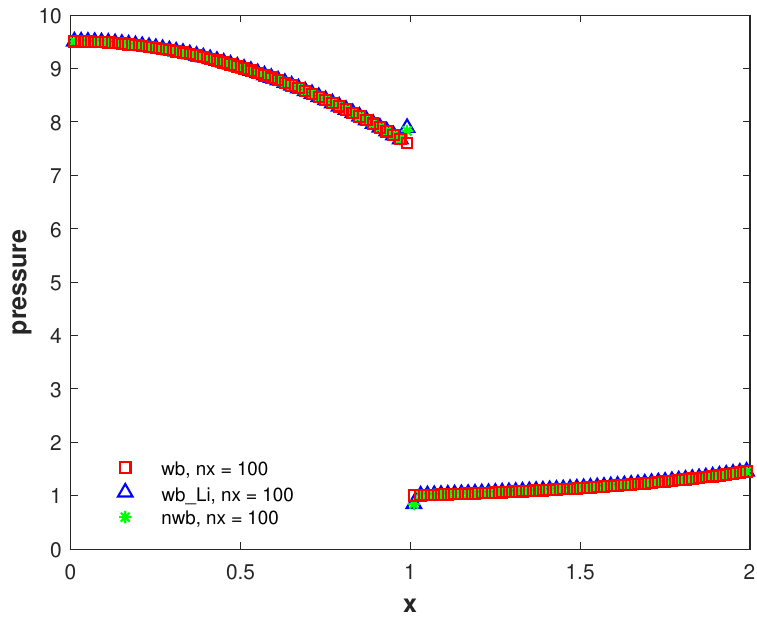}
  }
  \caption{Example \ref{euler:wb_aditr}: Numerical solutions of the density, velocity, and pressure (from left to right), using our well-balanced scheme, hydrostatically well-balanced scheme, and non-well-balanced scheme. Top: piecewise constant $(P^0)$ solution; bottom: quadratic $(P^2)$ solution.}\label{euler:transcom1}
\end{figure}

\begin{figure}[htb!]
  \centering
  \subfigure[ density $\rho$ ]{
  \centering
     \includegraphics[width= 4.8cm,scale=1]{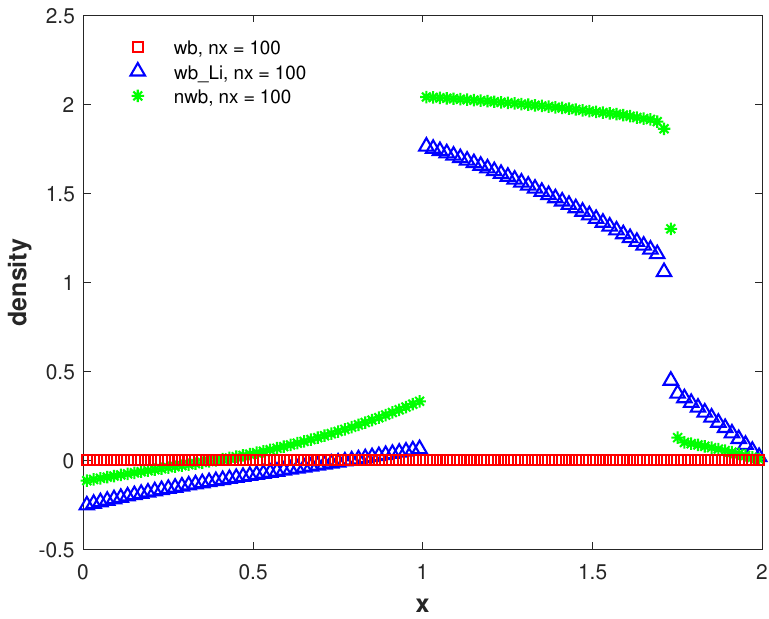}
  }
  \subfigure[velocity $u$]{
  \centering
     \includegraphics[width= 4.8cm,scale=1]{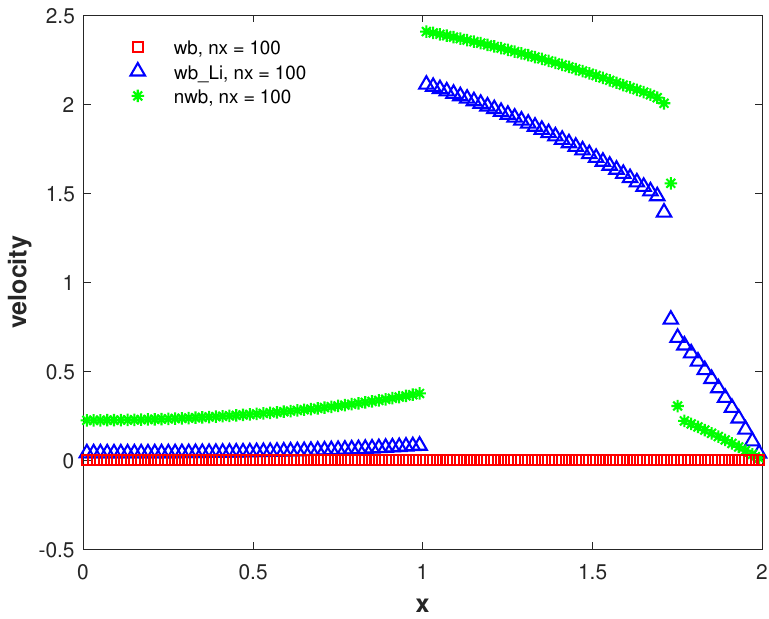}
  }
  \subfigure[pressure $p$]{
  \centering
     \includegraphics[width= 4.8cm,scale=1]{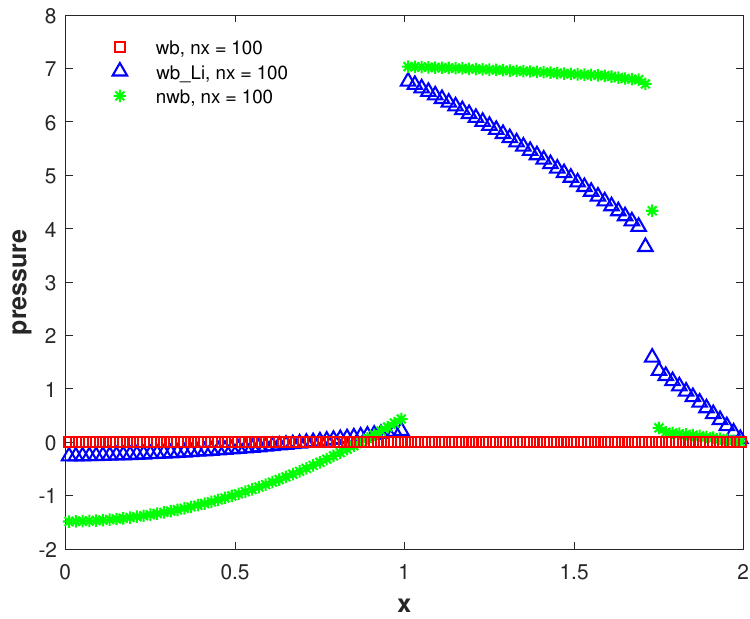}
  }
  \subfigure[ density $\rho$ ]{
  \centering
     \includegraphics[width= 4.8cm,scale=1]{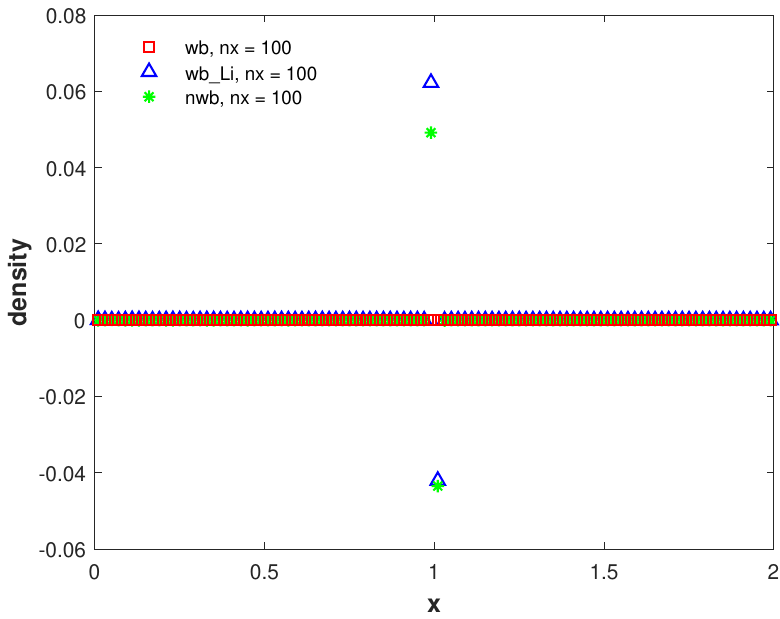}
  }
  \subfigure[velocity $u$]{
  \centering
     \includegraphics[width= 4.8cm,scale=1]{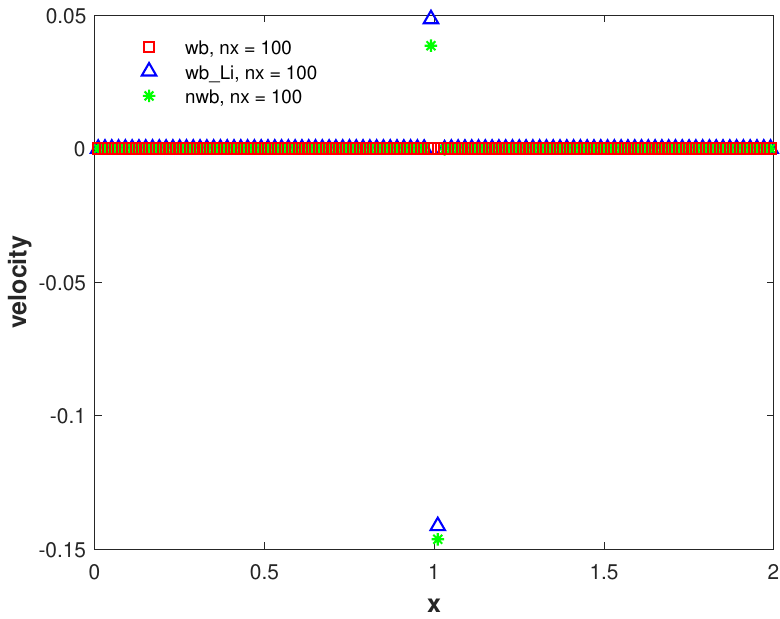}
  }
  \subfigure[pressure $p$]{
  \centering
     \includegraphics[width= 4.8cm,scale=1]{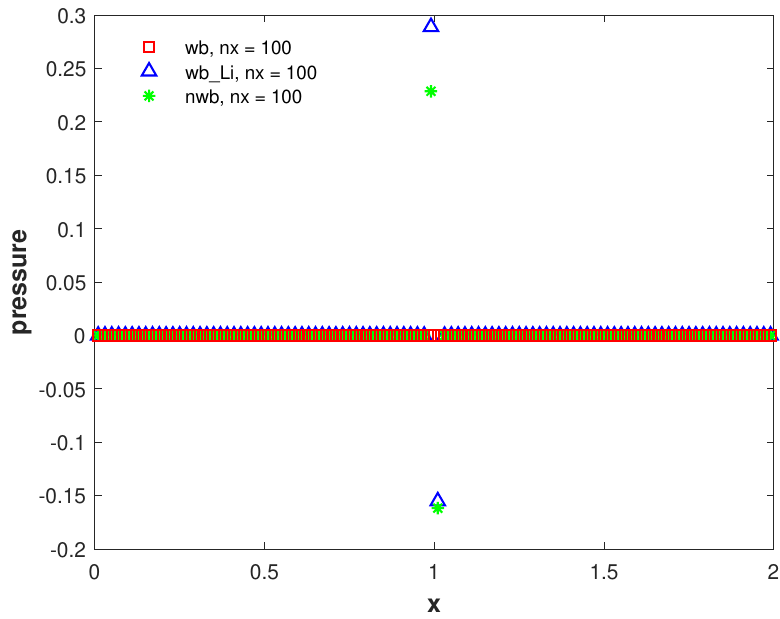}
  }
  \caption{Example \ref{euler:wb_aditr}: The difference between the numerical density, velocity, pressure (from left to right) and the original equilibrium state, using our well-balanced scheme, hydrostatically well-balanced scheme, and non-well-balanced scheme. Top: piecewise constant $(P^0)$ solution; bottom: quadratic $(P^2)$ solution.}\label{euler:transcom}
\end{figure}

\begin{example}{\bf Shock tube problem under gravitational fields }\label{euler:shock_tube}
\end{example}
We take the classical shock tube problem studied in \cite{li2016highfve}, which the initial conditions are given by
\begin{equation}\label{euler:sod_ini}
(\rho,u,p)(x,0) = \left\{\begin{array}{lll}
    (1,0,1),&  \text{if} \  x \leq 0.5, \\
    (0.125,0,0.1),&\text{otherwise},\\
    \end{array}\right.
\end{equation}
together with a linear gravitational field $\phi(x) = x$ on the computational domain $[0,1]$. The ratio of specific heats is $\gamma=1.4$. The boundary conditions are treated as solid walls. We run the simulation until $t=0.2$ with 200 and much refined 2000 cells for comparison. Fig. \ref{euler:shock2002} demonstrates the results of density $\rho$, velocity $u$, energy $E$, and pressure $p$, which achieve non-oscillatory and high-resolution and agree well with the reference solutions.
\begin{figure}[htb!]
  \centering
  \subfigure[density]{
  \centering
  \includegraphics[width=5cm,scale=1]{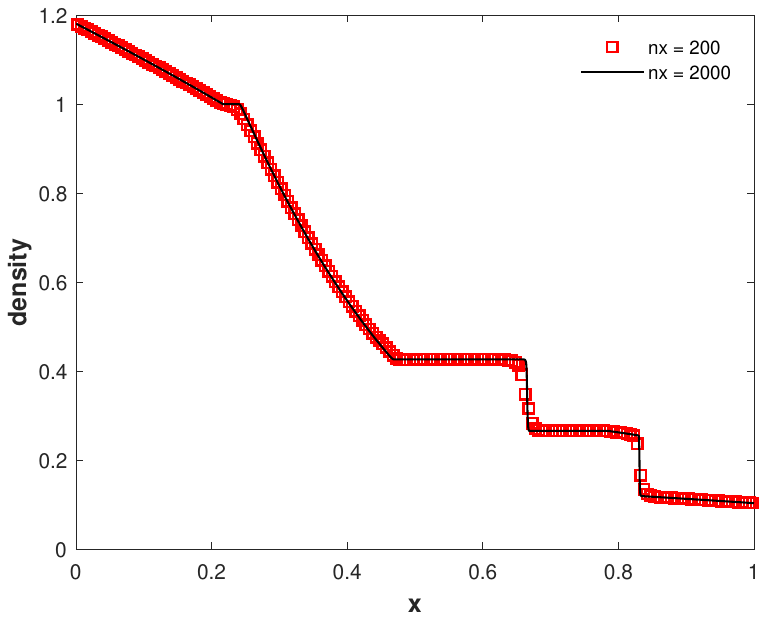}
  }
  \subfigure[velocity]{
  \centering
  \includegraphics[width=5cm,scale=1]{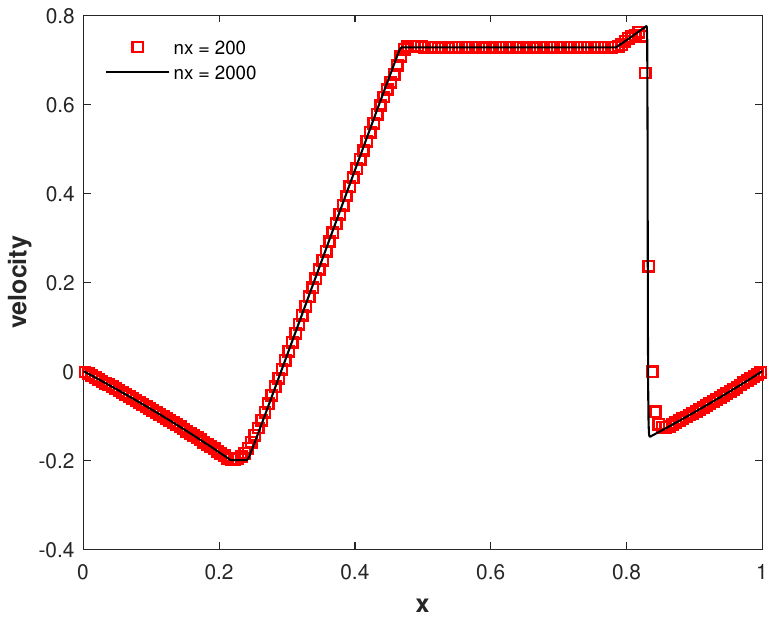}
  }
  \subfigure[energy]{
  \centering
  \includegraphics[width=5cm,scale=1]{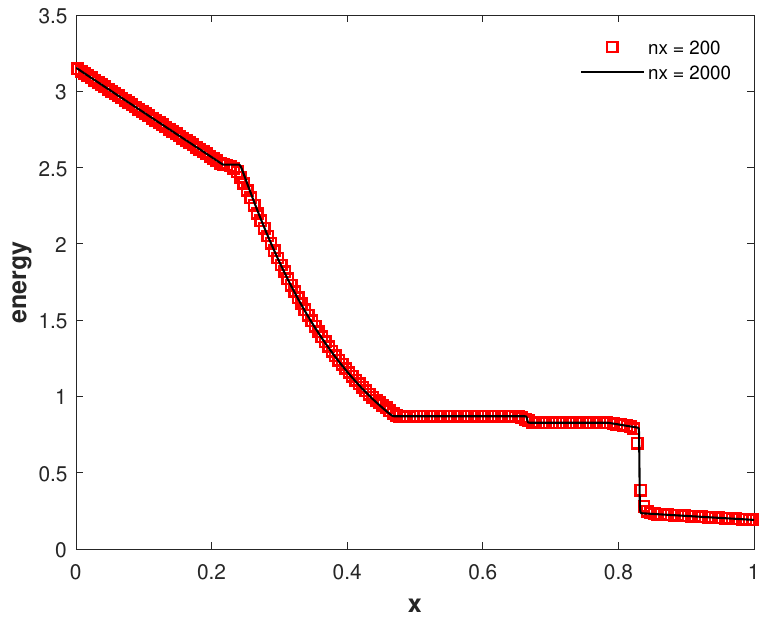}
  }
  \subfigure[pressure]{
  \centering
  \includegraphics[width=5cm,scale=1]{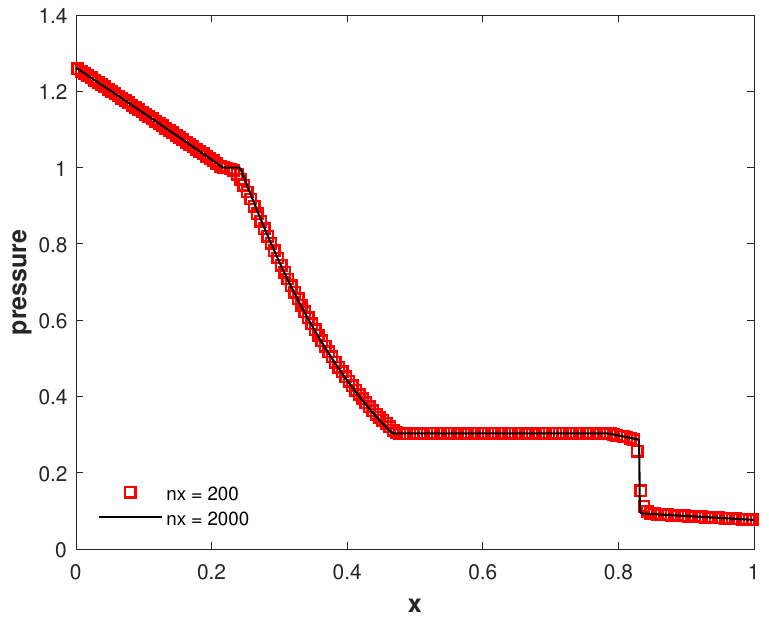}
  }
  \caption{Example \ref{euler:shock_tube}: Shock tube problem with initial value (\ref{euler:sod_ini}) at time $t=0.2$, using 200 and 2000 grid cells.}\label{euler:shock2002}
\end{figure}

\subsubsection{Two-dimensional tests}
\begin{example}{\bf Accuracy test}\label{euler:accuracy2D}
\end{example}
In this example, we test whether our scheme can achieve high order accuracy for a smooth solution in two-dimensional Euler equations.
The exact solutions on the computational domain $[0,2]\times[0,2]$ are given by
\begin{equation}\label{euler:acc2d}
  \begin{aligned}
    & \rho(x,y,t) = 1 + 0.2\sin(\pi(x+y-t(u_0+v_0))), \\
    & u(x,y,t) = u_0, \\
    & v(x,y,t) = v_0, \\
    & p(x,y,t) = p_0 + t(u_0+v_0) - x - y + 0.2\cos(\pi(x+y-t(u_0+v_0)))/\pi,
  \end{aligned}
\end{equation}
where the parameters $u_0=v_0=1,p_0=5.5$ and $\gamma=1.4$. The gravity is given by a linear gravitational potential $\phi_x = \phi_y=1$. The ghost points at the boundaries are set as the exact solutions. We compute the numerical solutions up to $t=0.1$ and show the  $L^1$ errors and the corresponding convergence rates in Table \ref{euler:DGacc2D}. It is found that the third order accuracy can be achieved by our proposed well-balanced DG scheme.
\begin{table}[htb]
  \centering
  \caption{Example \ref{euler:accuracy2D}:  $L^1$ errors and numerical orders of accuracy with initial condition (\ref{euler:acc2d}), $t = 0.1$.}\label{euler:DGacc2D}
  \begin{tabular}{c c c c c c c c }
  \toprule
    \multirow{2}{*} {$nx\times ny$} & \multicolumn{2}{c}{$\rho$}&\multicolumn{2}{c}{$ \rho u $}&\multicolumn{2}{c}{$ \rho v$}\\
   \cmidrule(lr){2-3} \cmidrule(lr){4-5} \cmidrule(lr){6-7}
    ~ &$L^1$ error &order&$L^1$ error &order&$L^1$ error &order\\
   \midrule
   10$\times$10 & 4.24E{-03} &     -- &  4.04E{-03} &  --  &  4.04E{-03} &  --  \\
   20$\times$20 & 5.71E{-04} &   2.89 &  4.80E{-04} & 3.07 &  4.80E{-04} & 3.07 \\
   40$\times$40 & 7.25E{-05} &   2.98 &  6.69E{-05} & 2.84 &  6.69E{-05} & 2.84 \\
   80$\times$80 & 8.69E{-06} &   3.06 &  9.98E{-06} & 2.74 &  9.98E{-06} & 2.74 \\
   160$\times$160 & 1.10E{-06} &   2.98 &  1.35E{-06} & 2.88 &  1.35E{-06} & 2.88 \\
   \midrule
   \multirow{2}{*} {$nx\times ny$} & \multicolumn{2}{c}{$E$}&\multicolumn{2}{c}{$ K $}&\multicolumn{2}{c}{$ \epsilon$}\\
   \cmidrule(lr){2-3} \cmidrule(lr){4-5} \cmidrule(lr){6-7}
    ~ &$L^1$ error &order&$L^1$ error &order&$L^1$ error &order\\
   \midrule
     10$\times$10 & 2.67E{-02} &     -- &  1.80E{-02} &    --  &  4.45E{-02} &     --  \\
     20$\times$20 & 3.48E{-03} &   2.94 &  2.31E{-03} &   2.96 &  5.72E{-03} &   2.96 \\
     40$\times$40 & 4.52E{-04} &   2.95 &  2.89E{-04} &   3.00 &  7.25E{-04} &   2.98 \\
     80$\times$80 & 5.71E{-05} &   2.98 &  3.55E{-05} &   3.03 &  9.12E{-05} &   2.99 \\
   160$\times$160 & 7.43E{-06} &   2.94 &  4.26E{-06} &   3.06 &  1.10E{-05} &   3.05 \\
  \bottomrule
  \end{tabular}
\end{table}

\begin{example}{\bf  Test for two-dimensional polytropic equilibrium}\label{euler:Poly1_2d}
\end{example}
To demonstrate the performance of our scheme for two-dimensional polytropic equilibrium, we consider an astrophysical problem studied in\cite{kappeli2014well}.
Let the computational domain be $[-0.5,0.5]\times[-0.5,0.5]$. The analytical hydrostatic equilibrium and the gravitation potential are given by
\begin{equation}\label{euler:hypoly_2d}
  \begin{aligned}
  &\rho(r) = \rho_c\dfrac{\sin(\alpha r)}{\alpha r}, \ u(r) =v(r) = 0, \ p(r) = \rho(r)^2, \\
  &\phi(r) = -2\rho_c\dfrac{\sin(\alpha r)}{\alpha r},
  \end{aligned}
\end{equation}
where $r = \sqrt{x^2+y^2}$ is the radial variable, the parameters are set as $\alpha = \sqrt{2\pi}$, $\rho_c=1$ and $\gamma=2$. We compute the numerical solutions up to $t = 0.5$ with 50$\times$50 uniform cells. The $L^1$ and $L^{\infty}$ errors at double precision are listed in Table \ref{euler:polywb_1002D}, from which the desired well-balanced property can be observed.

\begin{table}[htb]
  \centering
  \caption{Example \ref{euler:Poly1_2d}: $L^1$ and $L^{\infty}$ errors for two-dimensional polytropic hydrostatic equilibrium flow (\ref{euler:hypoly_2d}). }\label{euler:polywb_1002D}
  \begin{tabular}{ c c c c c c c }
   \toprule
    {error}&{$\rho$}&{$\rho u$}&{$\rho v$}&{$E$}&{$K$}&{$\varepsilon$}\\
    \midrule
    $L^1$        & 7.40E{-15} &  3.56E{-15} & 3.65E{-15} &  5.02E{-15} & 1.79E{-14} & 1.77E{-14}\\
    $L^{\infty}$ & 2.98E{-14} &  1.78E{-14} & 1.57E{-14} &  2.00E{-14} & 6.45E{-14} & 6.71E{-14}\\
    \bottomrule
  \end{tabular}
\end{table}

Next, we add a small Gaussian hump perturbation to the initial pressure profile
\begin{equation}\label{euler:hyp_ppres_2d}
  \begin{aligned}
  p(r) = \rho(r)^2 + A\exp(-100r^2)
  \end{aligned}
\end{equation}
with a small amplitude $A=10^{-6}$. We set the stopping time $t = 0.2$ before the excited waves reach the boundary. Transmission boundary conditions are imposed. In order to make a comparison, we run the same test both by our well-balanced method and the traditional third-order non-well-balanced scheme with 100$\times$100 uniform cells, and illustrate the contours of the pressure perturbation and velocity in Fig. \ref{euler:hypoly1_wbnwb2d}. We can see a failure simulation by the non-well-balanced scheme. In contrast, our resulting method works well and preserves the axial symmetry, which indicates the significance of the well-balanced scheme in capturing the small disturbance to the equilibrium state.

\begin{figure}[htb!]
  \centering
  \subfigure[pressure perturbation]{
  \centering
  \includegraphics[width=5.5cm,scale=1]{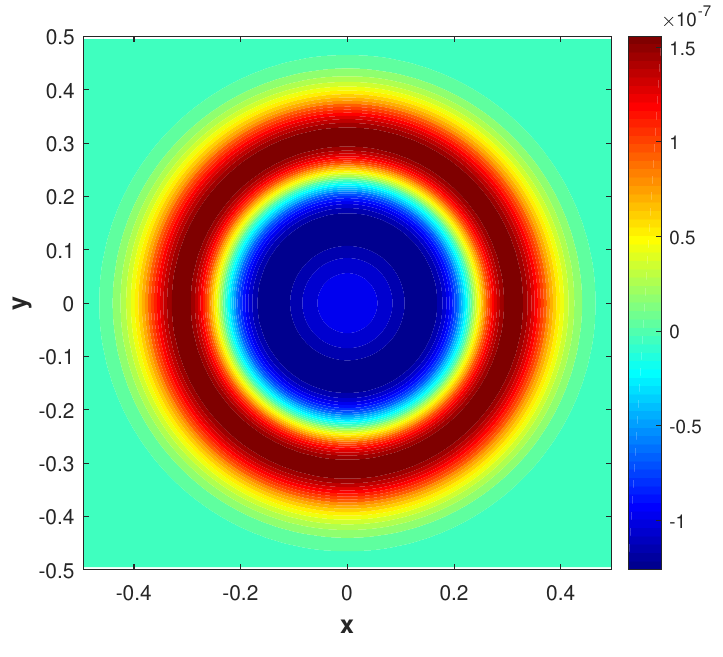}
  }
  \subfigure[velocity]{
  \centering
  \includegraphics[width=5.5cm,scale=1]{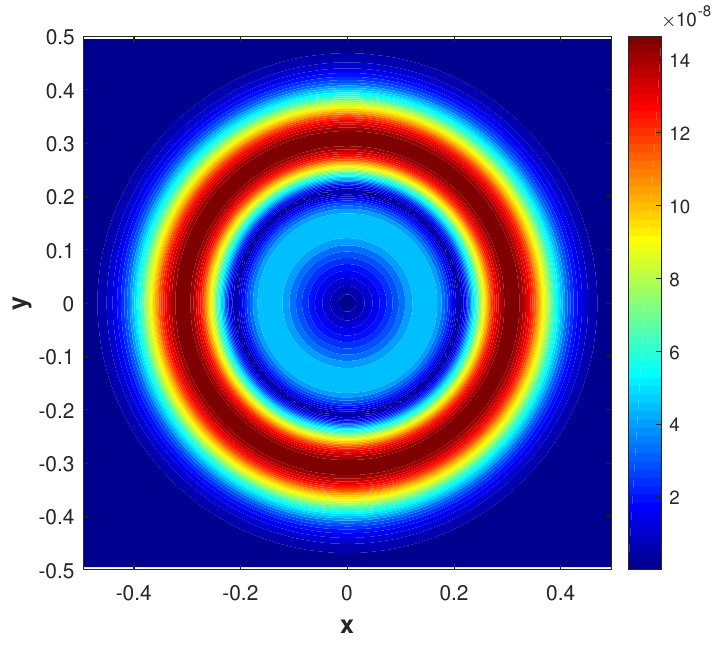}
  }
  \subfigure[pressure perturbation]{
  \centering
  \includegraphics[width=5.5cm,scale=1]{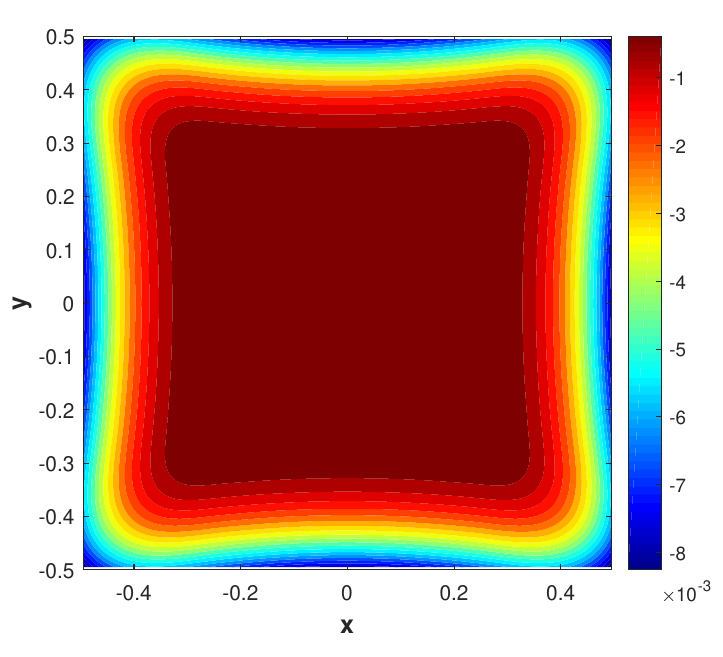}
  }
  \subfigure[velocity]{
  \centering
  \includegraphics[width=5.5cm,scale=1]{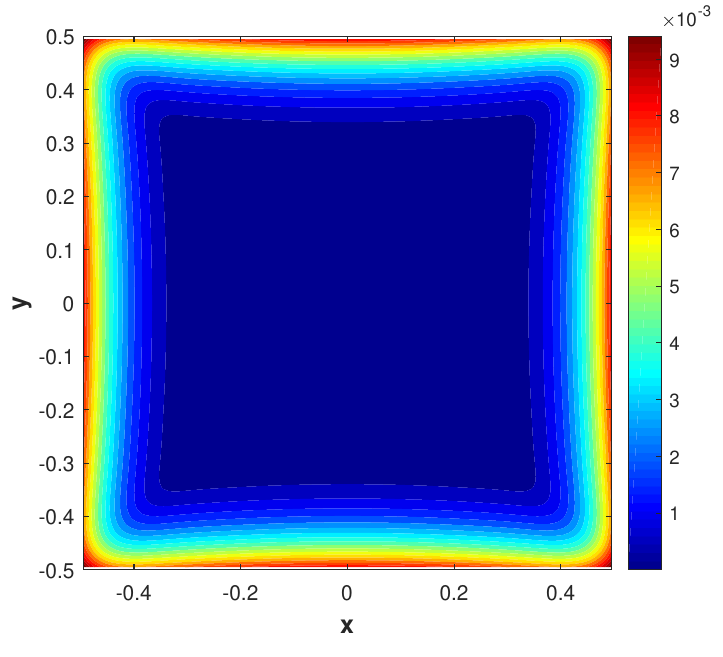}
  }
  \caption{Example \ref{euler:Poly1_2d}: The contours of the pressure perturbation and velocity $\sqrt{u^2+v^2}$ for the small amplitude $A=10^{-6}$ in (\ref{euler:hyp_ppres_2d}), with 20 uniformly spaced contour lines at time $t=0.2$, using 100$\times$100 cells. Top: results by our well-balanced DG scheme. Bottom: results by the non-well-balanced DG scheme. }\label{euler:hypoly1_wbnwb2d}
\end{figure}

\begin{example}{\bf Test for two-dimensional isentropic equilibrium}\label{euler:Poly2_2d}
\end{example}
This typical example we test is given by \cite{li2018welldge}, which shows the capability of our developed method for the preservation and perturbation of the isentropic hydrostatic equilibrium state in two dimensions.  The isentropic equilibrium state is given by
\begin{equation}\label{euler:isepoly_2d}
  \begin{aligned}
  &\rho(x,y) =  \left( 1-\frac{\gamma-1}{\gamma} \phi(x,y)\right)^{\frac{1}{\gamma-1}}, \\
  & u(x,y)=v(x,y) = 0,\\
  & p(x,y) = \left( 1-\frac{\gamma-1}{\gamma} \phi(x,y)\right)^{\frac{\gamma}{\gamma-1}},
  \end{aligned}
\end{equation}
coupled with the linear gravitation potential $ \phi(x,y) = x+y$ on the  computational domain $[0,1]\times[0,1]$. The ratio of specific heats is $\gamma=1.4$. We compute the numerical solutions up to $t = 0.5$ with 50$\times$50 uniform cells and present the $L^1$ and $L^{\infty}$ errors at double precision in Table \ref{euler:isewb_1002D}. We can observe that the well-balanced property is indeed maintained.

\begin{table}[htb]
  \centering
  \caption{Example \ref{euler:Poly2_2d}: $L^1$ and $L^{\infty}$ errors for two-dimensional isentropic hydrostatic equilibrium flow (\ref{euler:isepoly_2d}).}\label{euler:isewb_1002D}
  \begin{tabular}{ c c c c c c c}
   \toprule
    {error}&{$\rho$}&{$\rho u$}&{$\rho v$}&{$E$}&{$K$}&{$\varepsilon$}\\
    \midrule
    $L^1$        & 6.34E{-15} &  1.84E{-15} & 1.87E{-15} &  1.73E{-14} & 5.67E{-15} & 1.23E{-14}\\
    $L^{\infty}$ & 2.44E{-14} &  1.02E{-14} & 1.12E{-14} &  6.21E{-14} & 3.53E{-14} & 2.46E{-13}\\
    \bottomrule
  \end{tabular}
\end{table}

Subsequently, a small perturbation is added to the initial pressure profile
 \begin{equation}\label{euler:ise_ppres_2d}
  \begin{aligned}
  p(x,y,0) = p(x,y) + A\exp\left(-\frac{100\rho_0}{p_0}((x-0.3)^2+(y-0.3)^2)\right)
  \end{aligned}
\end{equation}
with the parameters $\rho_0=1.21,p_0=1$ and a small amplitude $A=10^{-6}$. Transmission boundary conditions are considered. The simulation is computed up to $t = 0.15$ with 100$\times$100 uniform cells. Fig. \ref{euler:hypoly2_wbnwb2d} displays the contours of the pressure perturbation and velocity calculated by our proposed method as well as the non-well-balanced method. It suggests that our well-balanced scheme performs better than the non-well-balanced scheme. It can achieve high resolution and capture the small features of the disturbance very well.

\begin{figure}[htb!]
  \centering
  \subfigure[pressure perturbation]{
  \centering
  \includegraphics[width=5.5cm,scale=1]{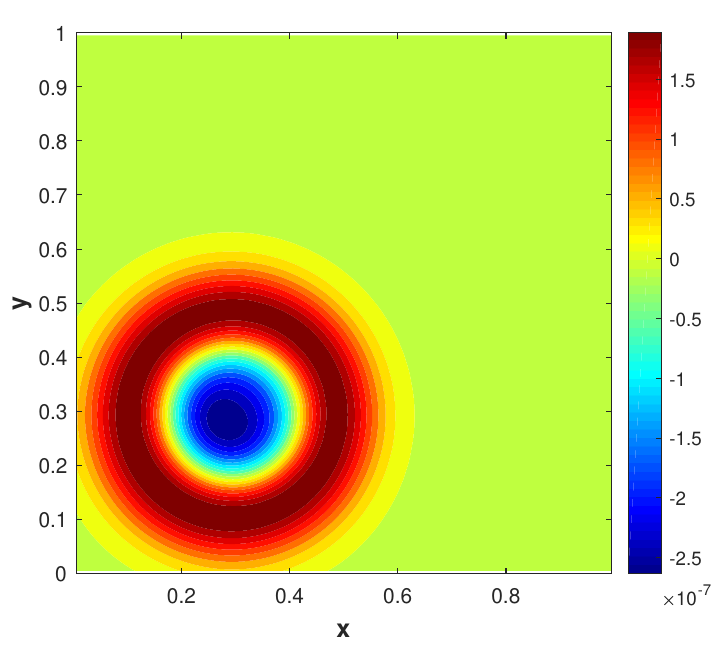}
  }
  \subfigure[velocity]{
  \centering
  \includegraphics[width=5.5cm,scale=1]{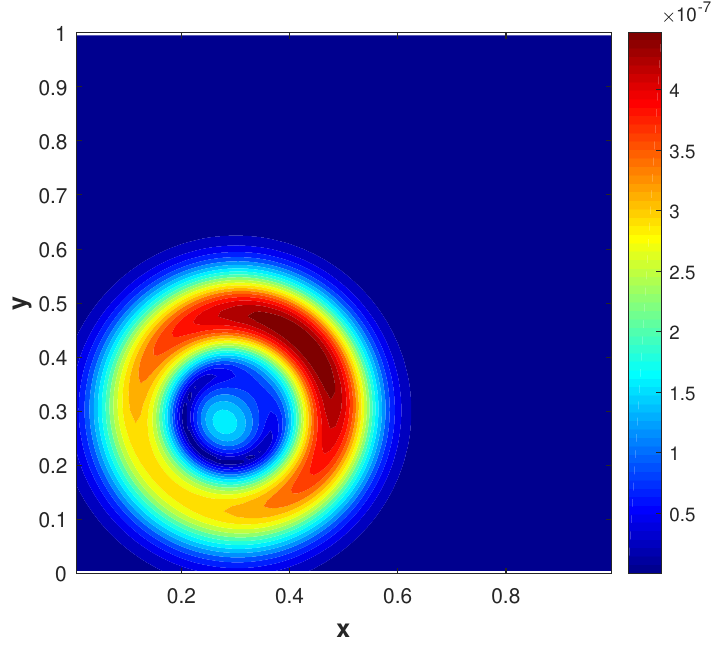}
  }
  \subfigure[pressure perturbation]{
  \centering
  \includegraphics[width=5.5cm,scale=1]{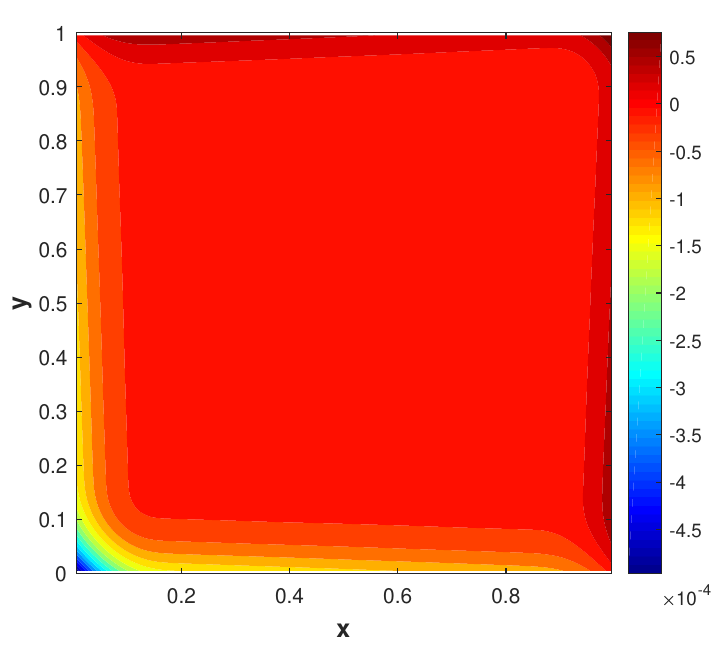}
  }
  \subfigure[velocity]{
  \centering
  \includegraphics[width=5.5cm,scale=1]{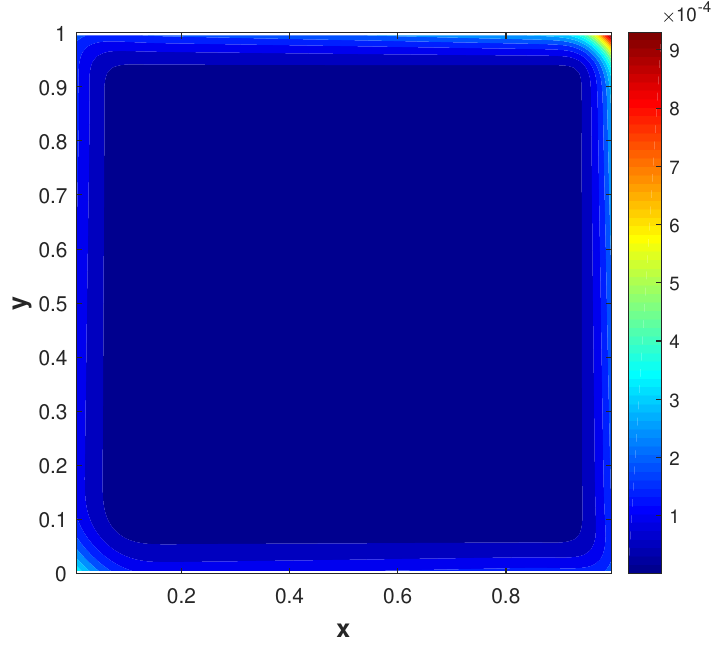}
  }
  \caption{Example \ref{euler:Poly2_2d}: The contours of the pressure perturbation and velocity $\sqrt{u^2+v^2}$ for the small amplitude $A=10^{-6}$ in (\ref{euler:ise_ppres_2d}), with 20 uniformly spaced contour lines at time $t=0.15$, using 100$\times$100 cells. Top: results by our well-balanced DG scheme. Bottom: results by the non-well-balanced DG scheme. }\label{euler:hypoly2_wbnwb2d}
\end{figure}

\subsection{Ripa model}
\subsubsection{One-dimensional tests}

\begin{example}{\bf Test for the moving water well-balanced property}\label{Ripa:exactC1D_moving}
\end{example}
We will verify that the DG method proposed maintains the moving water well-balanced property in this example. Three different steady-state solutions, which consist of subcritical, supercritical, and transcritical flow, will be investigated over the following bottom topography
\begin{equation}\label{Ripa:smo1D_moving}
  b(x)=\left\{\begin{array}{lll}
    0.2-0.05(x-10)^2,&\text{if}\ x\in[8,12], \\
    0,&\text{otherwise},\\
    \end{array}\right.
\end{equation}
for a channel of 25 meters. Taken from \cite{britton2020high}, they are extensions of the classical numerical examples used in \cite{xing2014exactly} for testing the moving water equilibria of the shallow water equations.

\smallskip
{\noindent \textbf{(a) Subcritical flow}}\label{Ripa:exc:sub}
The initial conditions are given by
\begin{equation}\label{Ripa:inisub}
  E(x,0) = 22.06605 \cdot 5,  \ hu(x,0)=4.42\sqrt{5}, \ \theta(x,0) = 5,
\end{equation}
together with the boundary conditions
\begin{equation*}
  hu(0,t)=4.42\sqrt{5}, \ h(25,t) = 2.
\end{equation*}
The Froude number is initially smaller than one throughout the computational domain.

\smallskip
{\noindent \textbf{(b) Supercritical flow}}\label{Ripa:exc:super}
The initial conditions are given by
\begin{equation}\label{Ripa:inisuper}
  E(x,0) = 91.624 \cdot 5,  \ hu(x,0)=24\sqrt{5}, \ \theta(x,0) = 5,
\end{equation}
together with the boundary conditions
\begin{equation*}
  hu(0,t)=24\sqrt{5}, \ h(0,t) = 2.
\end{equation*}
The Froude number here is larger than one everywhere.

\smallskip
{\noindent \textbf{(c) Transcritical flow}}\label{Ripa:exc:tranwithout}
The initial conditions are given by
\begin{equation}\label{Ripa:initrwithout}
  E(x,0) = 11.0907140397782 \cdot 5, \ hu(x,0)=1.53\sqrt{5},\ \theta(x,0) = 5,
\end{equation}
together with the boundary conditions
\begin{equation*}
  hu(0,t)=1.53\sqrt{5}, \ h(25,t) = 0.405748088283403 \  \text{for the subsonic flow}.
\end{equation*}
The Froude number at the beginning is chosen as
\begin{equation*}
\left\{\begin{array}{lll}
    Fr_j(0)<1,&  \text{if} \   x_j <10, \\
    Fr_j(0)>1,&  \text{if} \   x_j >10.
    \end{array}\right.
\end{equation*}
As we can see, the above three cases are all given in terms of the equilibrium variables and are exactly in equilibrium for the initial data. We calculate the numerical solutions with 200 uniform cells and set the stopping time $t=1$. Table \ref{Ripa:wbmoving_smo1D} shows the $L^1$ and $L^{\infty}$ errors of the above three cases at double precision. We can see that the moving water equilibrium state is exactly preserved up to the round-off error by our proposed well-balanced DG scheme.

\begin{table}[htb]
  \centering
  \caption{Example \ref{Ripa:exactC1D_moving}: $L^1$ and $L^{\infty}$ errors for one-dimensional moving water equilibrium state.}\label{Ripa:wbmoving_smo1D}
  \begin{tabular}{ c c c c c c c}
   \toprule
    case &{error}&{$h$}&{$hu$}&{$h\theta$}&{$E$}&{$\theta$}\\
    \midrule
    \multirow{2}{*}{(a) }&$L^1$        & 6.24E{-14} & 4.36E{-14} &  3.25E{-13} &  2.37E{-12} &  1.64E{-14} \\
                        ~&$L^{\infty}$ & 3.38E{-15} & 6.57E{-15} &  1.71E{-14} &  1.29E{-13} &  4.12E{-15} \\
    \midrule
    \multirow{2}{*}{(b) }&$L^1$        & 5.98E{-14} & 3.38E{-13} &  2.82E{-13} &  1.61E{-11} &  1.10E{-14} \\
                       ~ &$L^{\infty}$ & 6.29E{-15} & 4.50E{-14} &  3.21E{-14} &  1.37E{-12} &  1.71E{-15} \\
    \midrule
    \multirow{2}{*}{(c) }&$L^1$        & 2.44E{-13} & 9.11E{-14} &  1.22E{-12} &  1.08E{-12} &  3.67E{-14} \\
                       ~ &$L^{\infty}$ & 7.79E{-13} & 2.75E{-13} &  3.90E{-12} &  8.72E{-13} &  1.38E{-14} \\
    \bottomrule
  \end{tabular}
\end{table}

\begin{example}{\bf Small perturbation of moving water equilibrium}\label{Ripa:permoving_smo1D}
\end{example}
The same example in \cite{britton2020high} demonstrates the ability to capture the propagation of small perturbations of moving water equilibria, which cannot be realized by a non-well-balanced scheme. Almost the same setup in Example \ref{Ripa:exactC1D_moving}, we impose a small pulse $A = 0.0001$ on the water height in the interval $[5.75,6.25]$, the discharge $hu$ and the temperature $\theta$ keep the same as in (\ref{Ripa:inisub}) - (\ref{Ripa:initrwithout}). We simulate with transmissive boundary conditions until stopping time $t=0.75$ for the first and third cases, $t=0.45$ for the second case, using 200 uniform grid cells and a finer mesh of 1000 cells for comparison. Figs. \ref{Ripa:persubsmall}, \ref{Ripa:persupsmall} and \ref{Ripa:pertranssmall} display the numerical results of the surface level $h+b$, discharge $hu$ and $h\theta$ for the subcritical, supercritical and transcritical flow respectively. It can be seen that our developed method is successful in giving well-resolved and oscillatory-free solutions.

\begin{figure}[htb!]
  \centering
  \subfigure[surface level $h+b$]{
  \centering
  \includegraphics[width=4.9cm,scale=1]{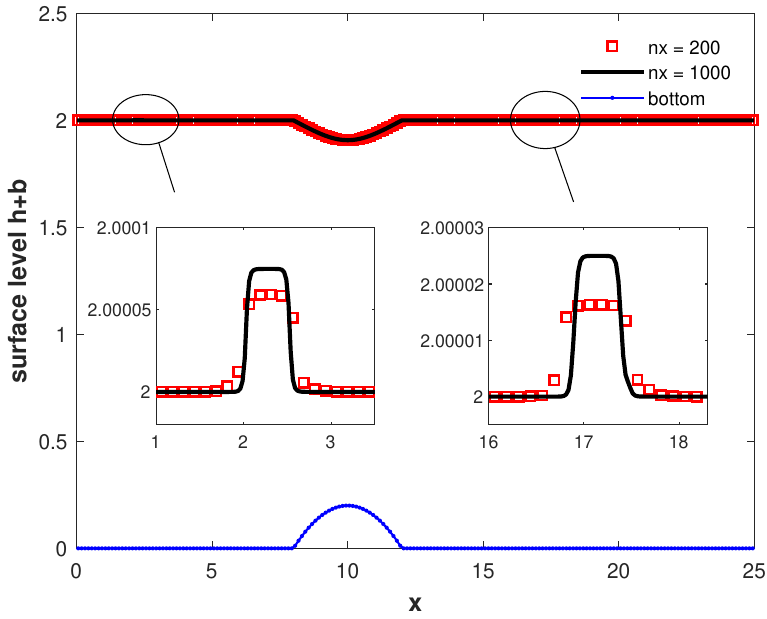}
  }
  \subfigure[discharge $hu$]{
  \centering
  \includegraphics[width=4.9cm,scale=1]{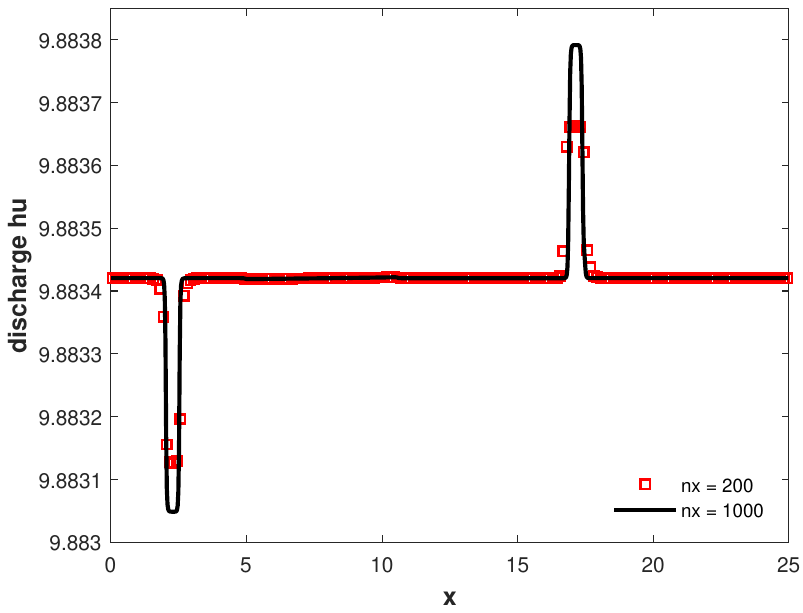}
  }
  \subfigure[ $h\theta$]{
  \centering
  \includegraphics[width=4.9cm,scale=1]{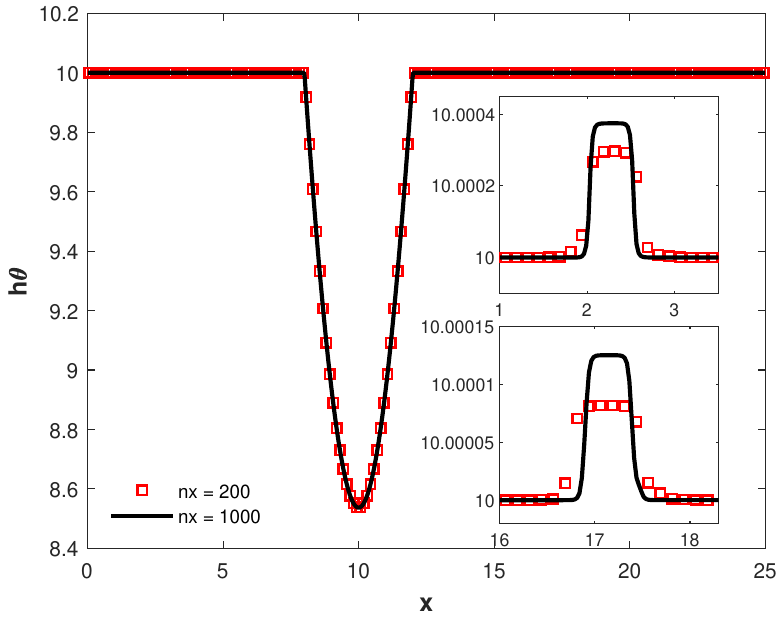}
  }
  \caption{Example \ref{Ripa:permoving_smo1D}: Small perturbation of the subcritical flow for the small pulse $A=0.0001$, using 200 and 1000 grid cells at time $t=0.75$. }\label{Ripa:persubsmall}
\end{figure}

\begin{figure}[htb!]
  \centering
  \subfigure[surface level $h+b$]{
  \centering
  \includegraphics[width=4.9cm,scale=1]{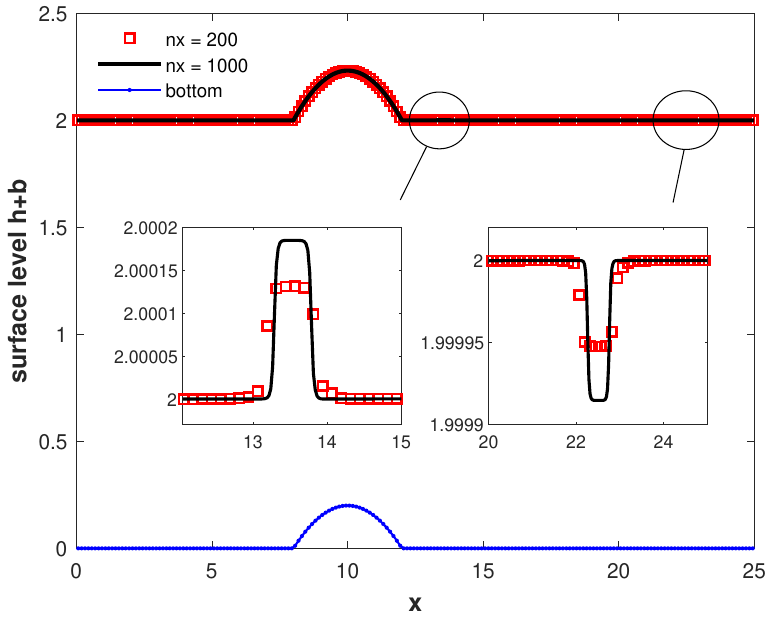}
  }
  \subfigure[discharge $hu$]{
  \centering
  \includegraphics[width=4.9cm,scale=1]{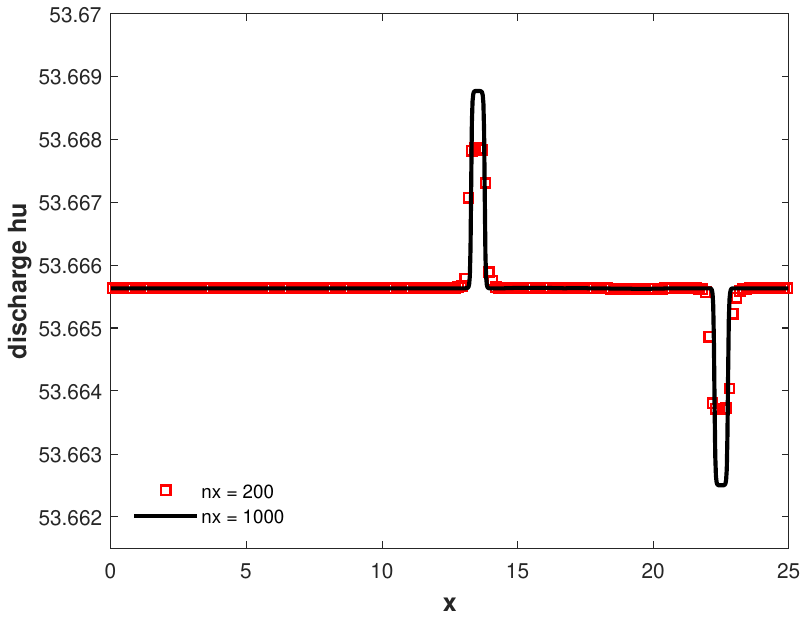}
  }
  \subfigure[ $h\theta$]{
  \centering
  \includegraphics[width=4.9cm,scale=1]{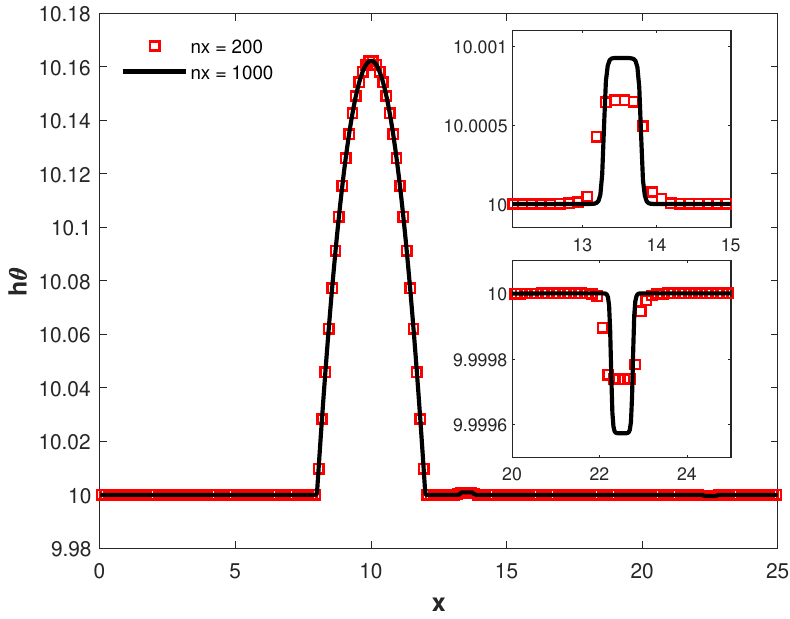}
  }
  \caption{Example \ref{Ripa:permoving_smo1D}: Small perturbation of the supercritical flow for the small pulse $A=0.0001$, using 200 and 1000 grid cells at time $t=0.45$. }\label{Ripa:persupsmall}
\end{figure}

\begin{figure}[htb!]
  \centering
  \subfigure[surface level $h+b$]{
  \centering
  \includegraphics[width=4.9cm,scale=1]{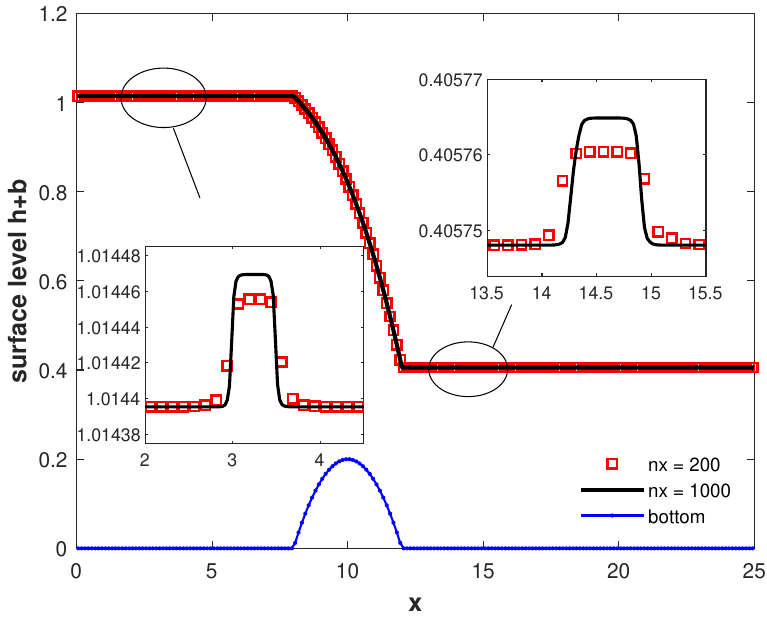}
  }
  \subfigure[discharge $hu$]{
  \centering
  \includegraphics[width=4.9cm,scale=1]{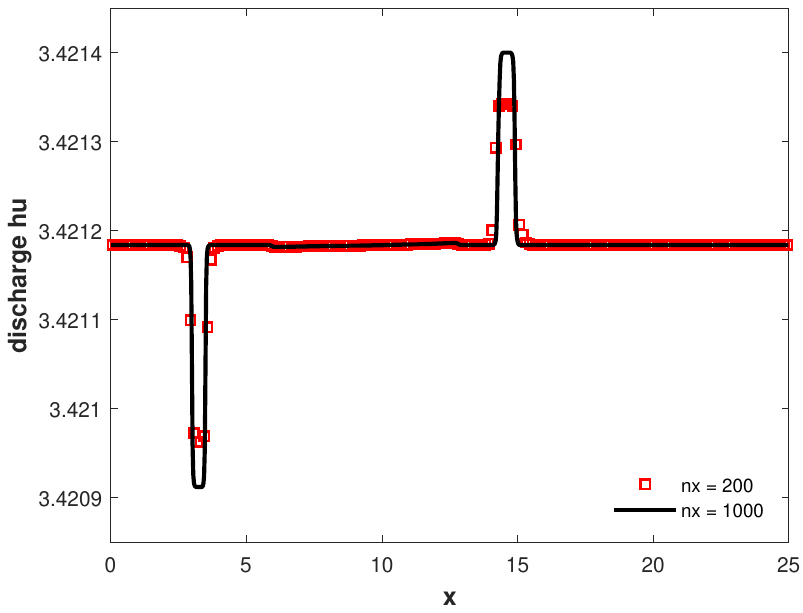}
  }
  \subfigure[ $h\theta$]{
  \centering
  \includegraphics[width=4.9cm,scale=1]{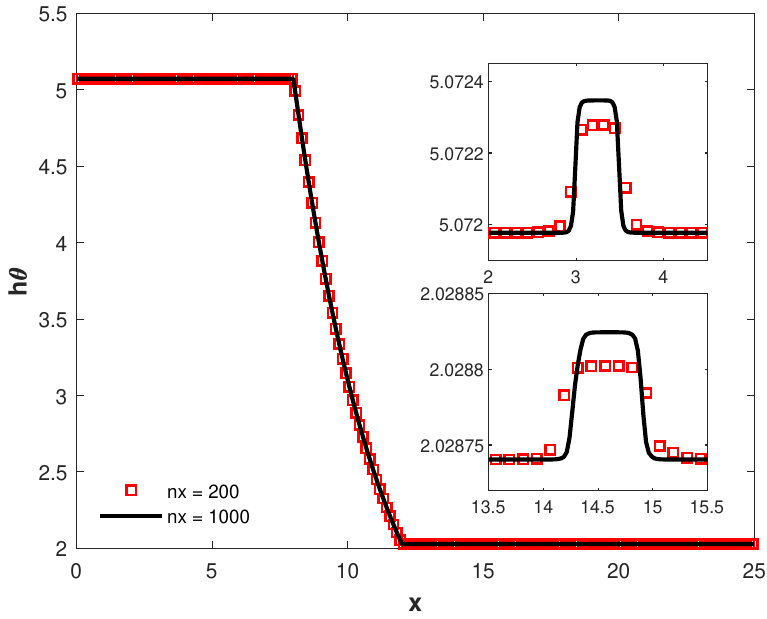}
  }
  \caption{Example \ref{Ripa:permoving_smo1D}: Small perturbation of the transcritical flow for the small pulse $A=0.0001$, using 200 and 1000 grid cells at time $t=0.75$.
}\label{Ripa:pertranssmall}
\end{figure}

\begin{example}{\bf The dam breaking problems}\label{Ripa:dam}
\end{example}
To demonstrate the capability of our developed well-balanced DG scheme when containing discontinuous solutions, we consider three dam breaking problems.
In our simulation, we treat the solution obtained by the third order still water equilibria preserving DG scheme \cite{qian2018high} (denoted by `wb$\_$still') on 3000 uniform grid cells as the reference solution.

First, multiple Riemann problems with nonzero initial velocity over a flat bottom ($b(x)=0$) are considered on a computational domain $[-1,1]$. The initial conditions are given in the following way
\begin{equation}\label{Ripa:dam_flat2}
(h,u,\theta)(x,0) = \left\{\begin{array}{lll}
    (2,0.75,1),&  \text{if} \  |x| \leq 0.5, \\
    (1,0.5,1.5),&\text{otherwise}.\\
    \end{array}\right.
\end{equation}
We compute the solutions up to $t=0.075$ with transmissive boundary conditions. The numerical solutions of the surface level $h+b$, discharge $hu$, $h\theta$, velocity $u$, temperature $\theta$ and pressure $\frac{1}{2}gh^2\theta$ with 200 uniform cells are displayed in Fig. \ref{Ripa:fd}.  We can observe that the numerical results reach an amicable agreement with the reference solution, and the discontinuity is well captured without oscillation.

\begin{figure}[htb!]
  \centering
  \subfigure[surface level $h+b$]{
  \centering
  \includegraphics[width=4.5cm,scale=1]{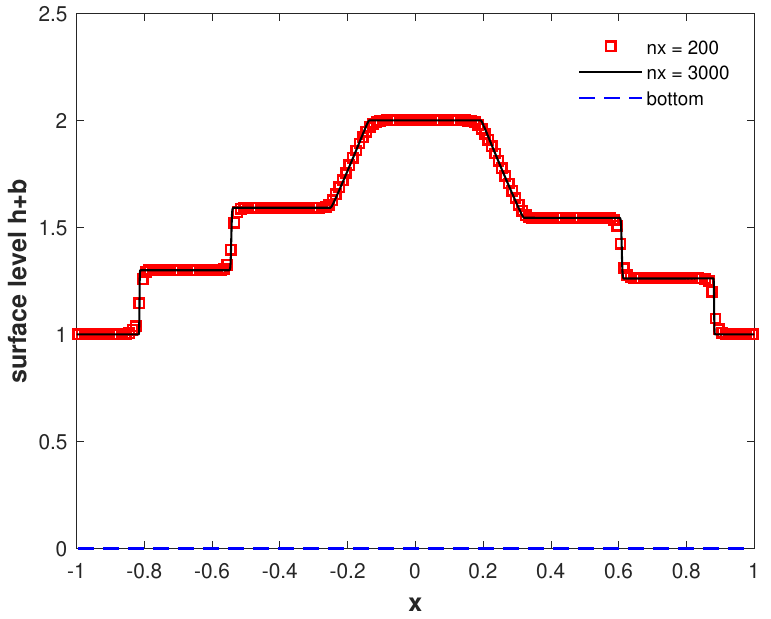}
  }
  \subfigure[discharge $hu$]{
  \centering
  \includegraphics[width=4.5cm,scale=1]{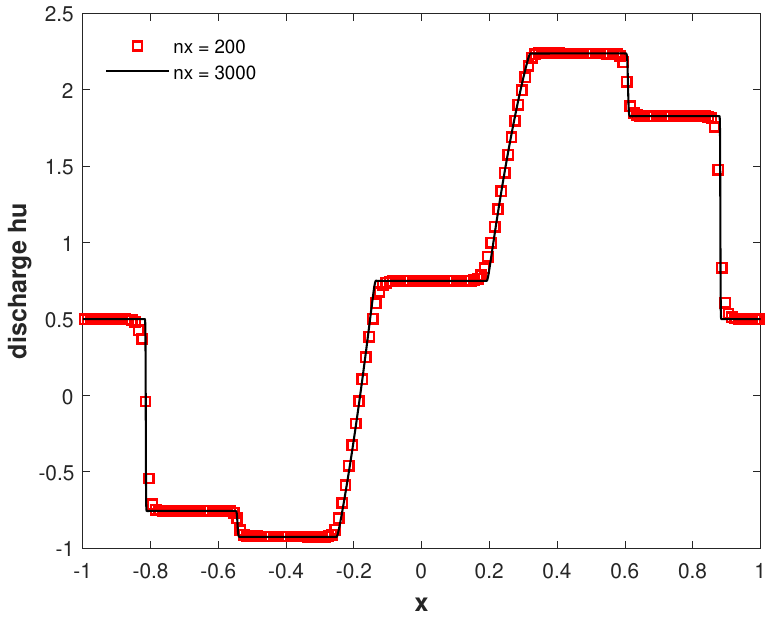}
  }
  \subfigure[velocity $u$]{
  \centering
  \includegraphics[width=4.5cm,scale=1]{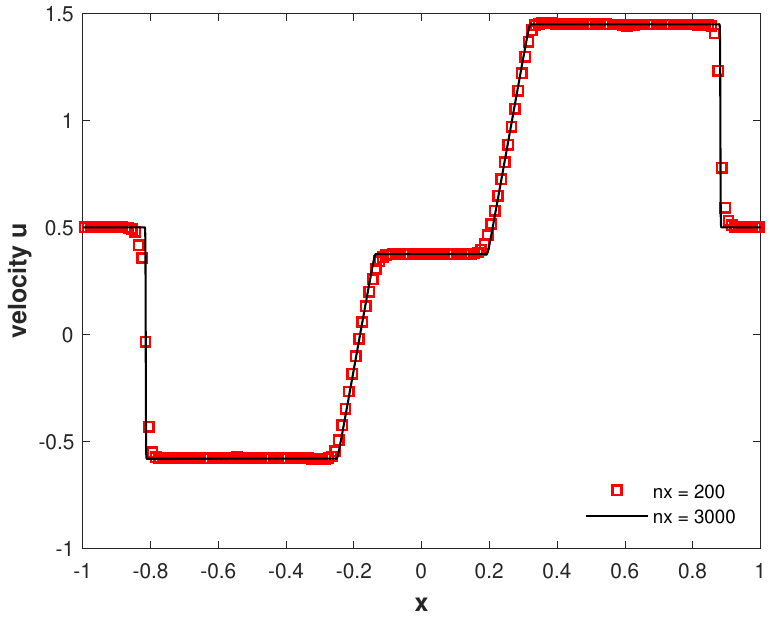}
  }
  \subfigure[ $h\theta$]{
  \centering
  \includegraphics[width=4.5cm,scale=1]{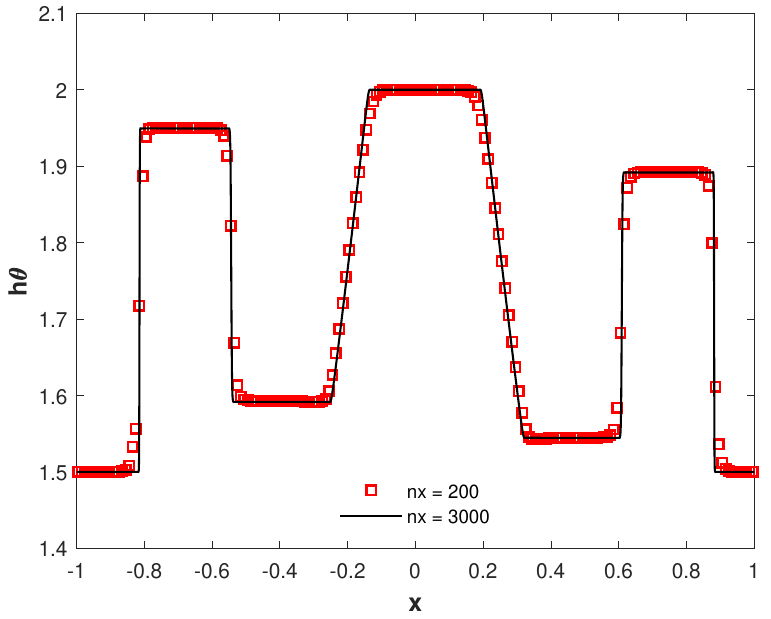}
  }
  \subfigure[temperature $\theta$]{
  \centering
  \includegraphics[width=4.5cm,scale=1]{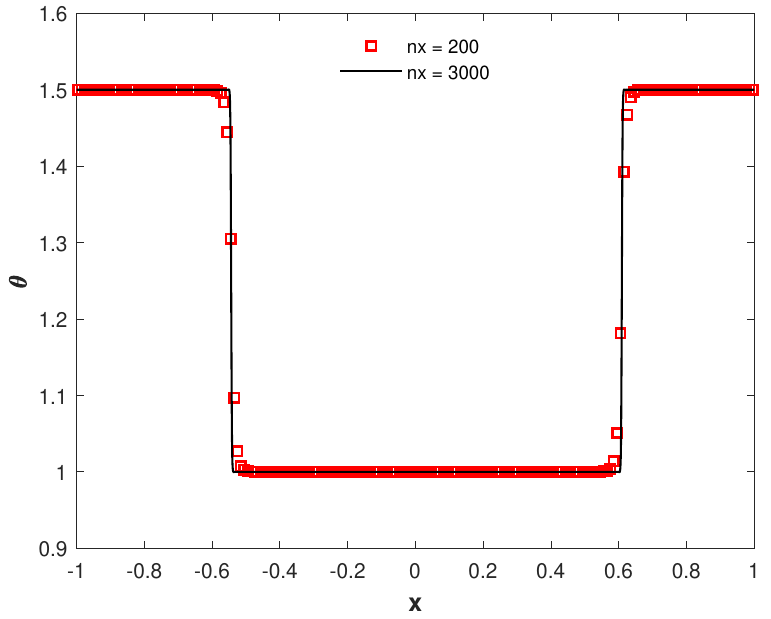}
  }
  \subfigure[ pressure $\frac{1}{2}gh^2\theta$]{
  \centering
  \includegraphics[width=4.5cm,scale=1]{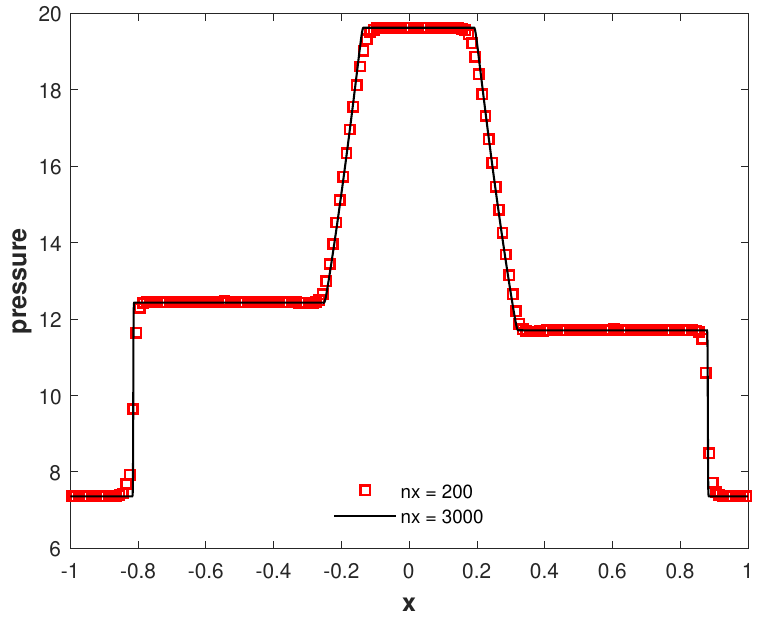}
  }
  \caption{Example \ref{Ripa:dam}: Numerical solutions for the dam breaking problem over a flat bottom topography with the initial value (\ref{Ripa:dam_flat2}) at time $t=0.075$.  }\label{Ripa:fd}
\end{figure}

Subsequently, another dam breaking problem over a non-flat bottom topography is considered on the computational domain $[-1,1]$. The bottom function which includes two bumps is defined as
\begin{equation}\label{Ripa:condam_bottom}
  b(x)=\left\{\begin{array}{lll}
    0.5(\cos(10\pi(x+0.3))+1), &\text{if}\ x\in[-0.4,-0.2], \\
    0.75(\cos(10\pi(x-0.3))+1),&\text{if}\ x\in[0.2,0.4],\\
    0,&\text{otherwise},\\
    \end{array}\right.
\end{equation}
and the initial conditions are given by
\begin{equation}\label{Ripa:condam_ini}
(h,u,\theta)(x,0) = \left\{\begin{array}{lll}
    (5-b(x),0,3),&  \text{if} \  x \leq 0, \\
    (2-b(x),0,5),&\text{otherwise}.\\
    \end{array}\right.
\end{equation}
We compute the solutions up to $t=0.045$ with transmissive boundary conditions. Fig. \ref{Ripa:cd} presents the corresponding solutions with 200 uniform cells. The numerical results realize non-oscillatory and match the reference ones well, which indicates the effectiveness of our proposed well-balanced DG scheme.
\begin{figure}[htb!]
  \centering
  \subfigure[surface level $h+b$]{
  \centering
  \includegraphics[width=4.5cm,scale=1]{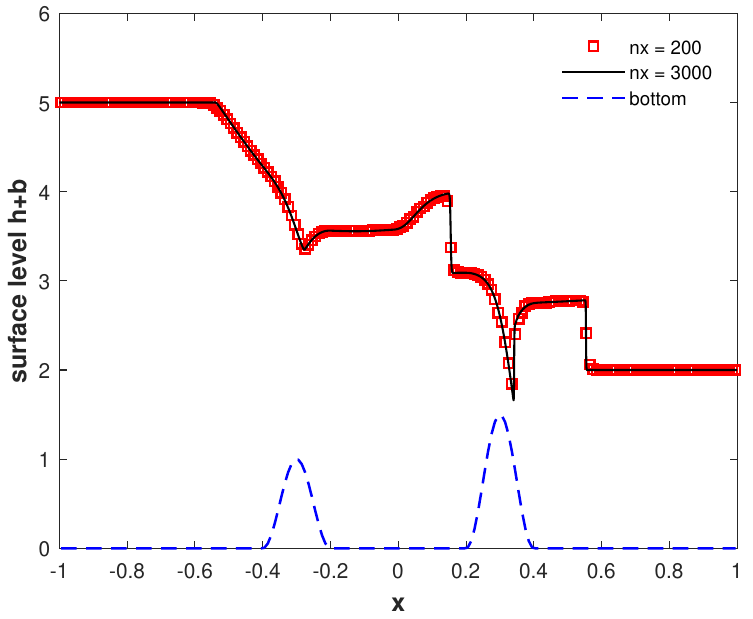}
  } 
  \subfigure[ temperature $\theta$]{
  \centering
  \includegraphics[width=4.5cm,scale=1]{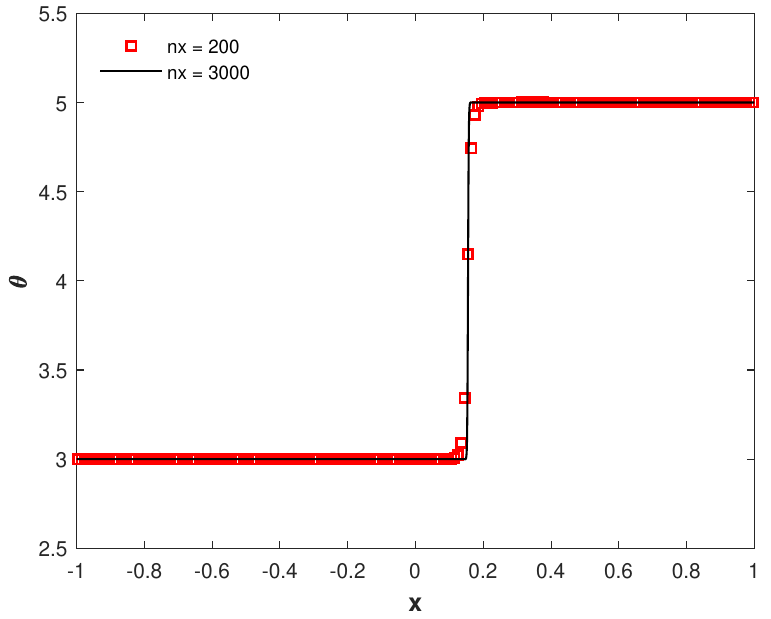}
  }
  \subfigure[pressure $\frac{1}{2}gh^2\theta$]{
  \centering
  \includegraphics[width=4.5cm,scale=1]{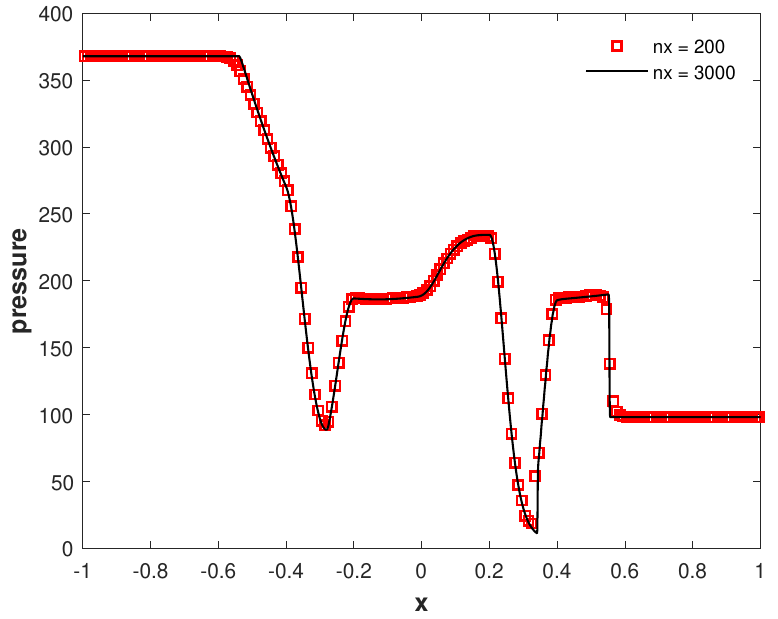}
  }
  \caption{Example \ref{Ripa:dam}: Numerical solutions for the dam breaking problem over a continuous bottom topography (\ref{Ripa:condam_bottom}) with the initial value (\ref{Ripa:condam_ini}) at time $t=0.045$.  }\label{Ripa:cd}
\end{figure}

Lastly, we implement a dam breaking test case used in \cite{britton2020high} over a rectangular bump, which involves rapidly varying flow over a discontinuous bottom topography on a domain $[0,600]$
\begin{equation}\label{Ripa:bdam}
  b(x)=\left\{\begin{array}{lll}
    8,&   \text{if} \ \   225\leq x \leq 375, \\
    0,&\text{otherwise},\\
    \end{array}\right.
\end{equation}
and the initial conditions are given by
\begin{equation}\label{Ripa:inidam}
 (h,u,\theta)(x,0) = \left\{\begin{array}{lll}
    (20-b(x),1,10),&  \text{if} \  x \leq 300, \\
    (15-b(x),5,5),&\text{otherwise}.\\
    \end{array}\right.
\end{equation}
The simulation is calculated up to $t=3$ with transmissive boundary conditions. We illustrate the results with 200 uniform cells in Fig. \ref{Ripa:rdu}. For comparison, much refined 3000 cell solutions by our proposed moving equilibria preserving scheme (solid black line) are also presented in the figure. We can see some numerical oscillations located at the points $x=225$ and $x=375$ for the discharge $hu$ calculated by the scheme \cite{qian2018high}, since the correct capturing of the water discharge is more difficult than the water height. By contrast, our scheme provides an oscillation-free solution for the discharge and realizes high-resolution interface tracking of fluid motion.
\begin{figure}[!ht]
  \centering
  \subfigure[surface level $h+b$]{
  \centering
  \includegraphics[width=4.5cm,scale=1]{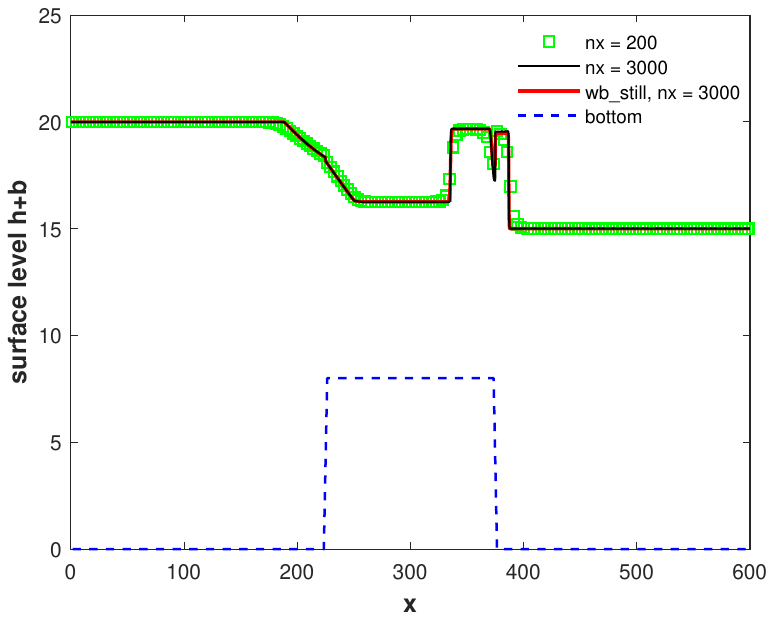}
  }
  \subfigure[discharge $hu$]{
  \centering
  \includegraphics[width=4.5cm,scale=1]{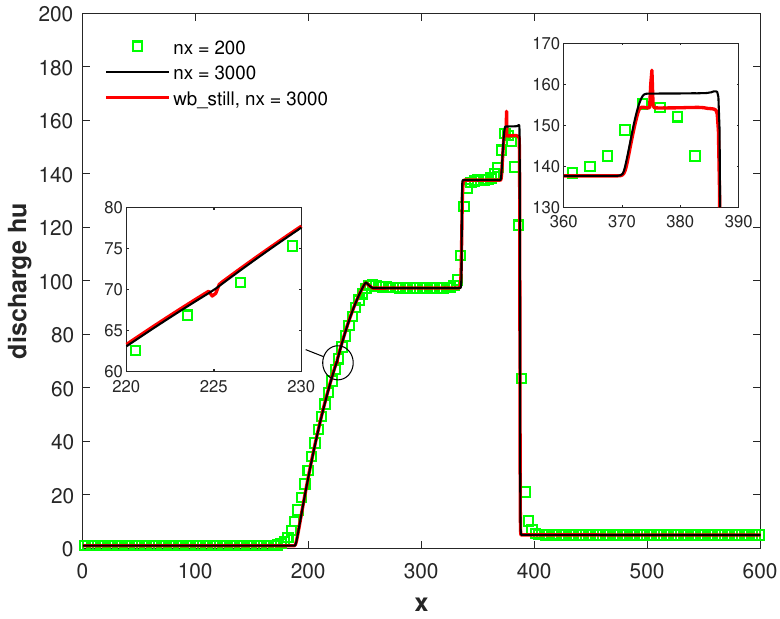}
  }
  \subfigure[velocity $u$]{
  \centering
  \includegraphics[width=4.5cm,scale=1]{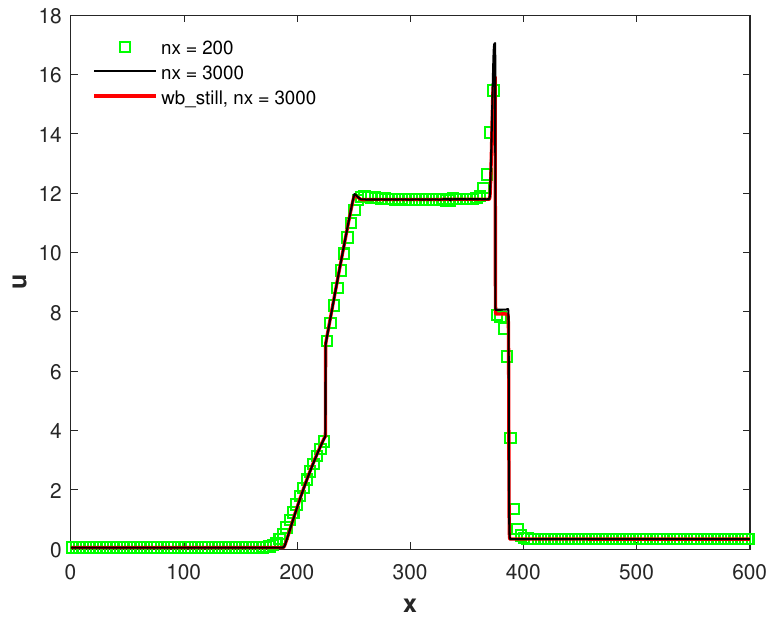}
  }
  \subfigure[ $h\theta$]{
  \centering
  \includegraphics[width=4.5cm,scale=1]{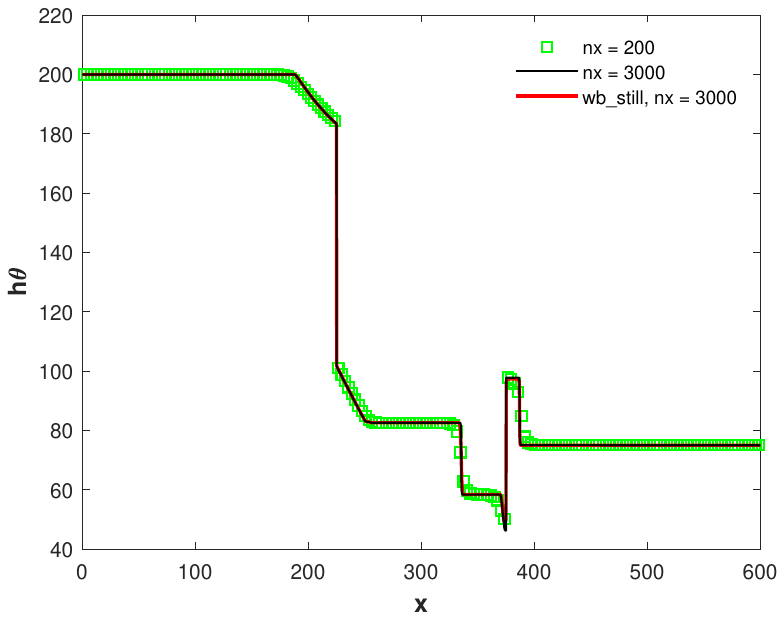}
  }
  \subfigure[temperature $\theta$]{
  \centering
  \includegraphics[width=4.5cm,scale=1]{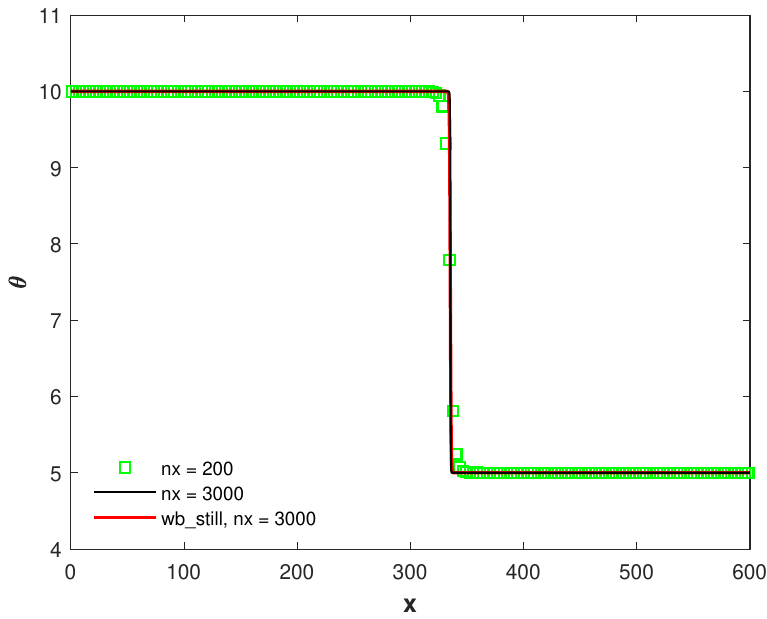}
  }
  \subfigure[pressure $\frac{1}{2}gh^2\theta$]{
  \centering
  \includegraphics[width=4.5cm,scale=1]{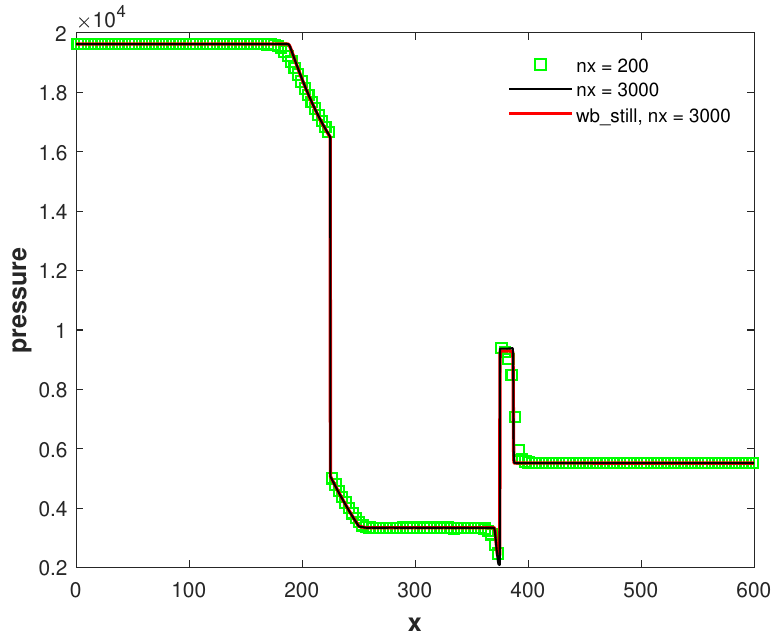}
  }
  \caption{Example \ref{Ripa:dam}: Numerical solutions for the dam breaking problem over a discontinuous bottom topography (\ref{Ripa:bdam}) with the initial value (\ref{Ripa:inidam}) at time $t=3$.}\label{Ripa:rdu}
\end{figure}

\begin{example}{\bf Pressure oscillation free}\label{Ripa:isobaric1D}
\end{example}
We take the example in \cite{chertock2014central} to testify the isobaric equilibrium preserving of our proposed DG method. The gravitation constant is set as $g=1$ in this test. The initial conditions on the computational domain $[-1000,1000]$ with the bottom function $b(x)=0$ are given by
\begin{equation}\label{Ripa:isoini}
(h,u,\theta)(x,0) = \left\{\begin{array}{lll}
    (2\sqrt{2},0,1),&  \text{if} \  x < 0, \\
    (1,0,8),&\text{otherwise}.\\
    \end{array}\right.
\end{equation}
It can be seen the pressure $p$ equals to 4 in the entire domain. We compute the solutions up to $t=10$ with 2000 uniform cells and display the corresponding surface level $h+b$, velocity $u$, and pressure $p$ in Fig. \ref{Ripa:pof}. Both the moving equilibria preserving DG scheme \eqref{scheme2d}-\eqref{flux2d} (denoted by `moving') and the isobaric equilibria preserving DG scheme \eqref{scheme:ripaiso2d}-\eqref{flux:ripaiso2d} (denoted by `isobaric') are shown for comparison. We can observe that the former produces spurious oscillations that do not disappear even with highly refined meshes. While the latter provides oscillation-free solutions, which validate the rationality of our isobaric equilibria preserving DG method.

\begin{figure}[htb!]
  \centering
  \subfigure[surface level $h+b$]{
  \centering
  \includegraphics[width=4.5cm,scale=1]{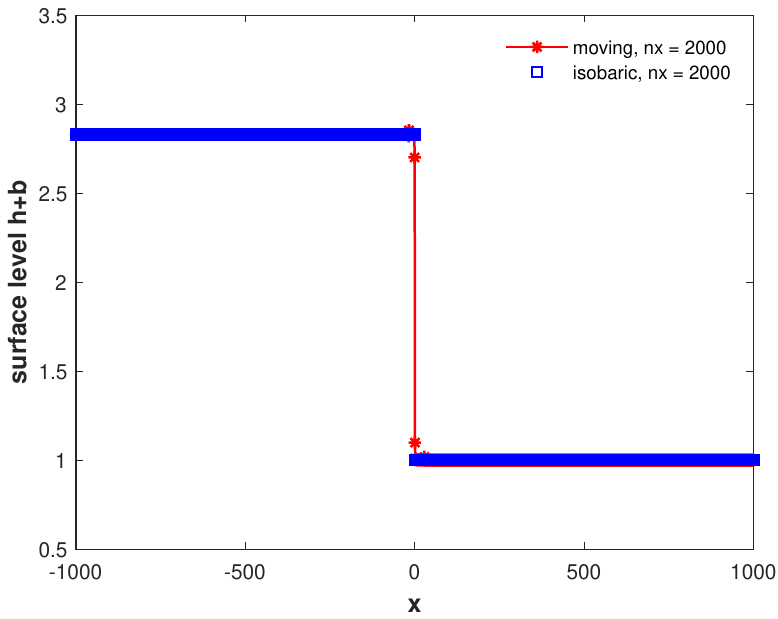}
  }
  \subfigure[velocity $u$]{
  \centering
  \includegraphics[width=4.5cm,scale=1]{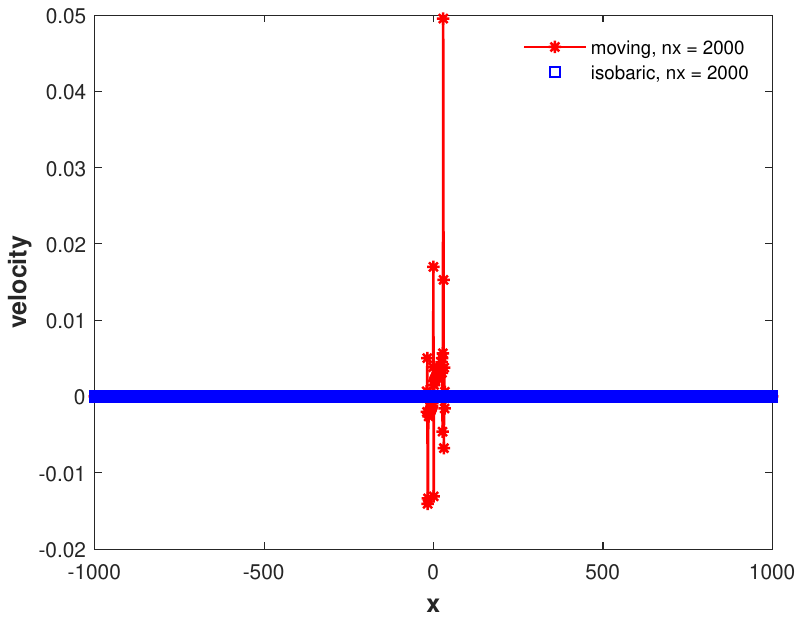}
  }
  \subfigure[ pressure $\frac{1}{2}gh^2\theta$]{
  \centering
  \includegraphics[width=4.5cm,scale=1]{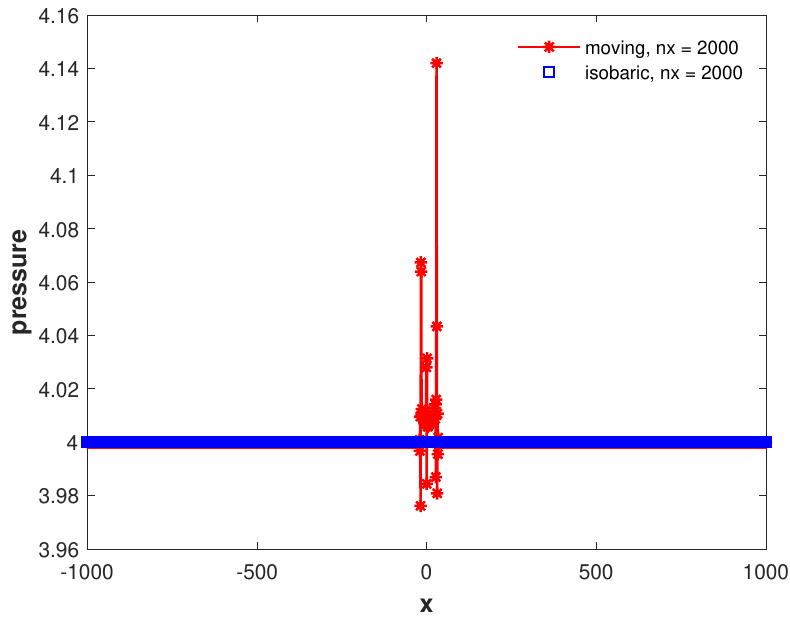}
  }
  \caption{Example \ref{Ripa:isobaric1D}: Numerical solutions of the surface level $h+b$, velocity $u$, pressure $p$ (from left to right) with the initial value (\ref{Ripa:isoini}) at time $t=10$, using 2000 grid cells.  }\label{Ripa:pof}
\end{figure}

\subsubsection{Two-dimensional tests}
\begin{example}{\bf Accuracy test }\label{Ripa:accuracy2D}
\end{example}
In this example, we test the high-order accuracy of our constructed well-balanced DG method for the 2D Ripa model. The bottom function and initial conditions on the domain $[0,1]\times[0,1]$ are given by
\begin{equation}\label{Ripa:acc2d}
\begin{aligned}
 & b(x,y) = \sin(2\pi x)+\cos(2\pi y), \\
 & h(x,y,0)=10+e^{\sin(2\pi x)}\cos(2\pi y),\quad hu(x,y,0) = \sin(\cos(2\pi x))\sin(2\pi y),\\
 & hv(x,y,0) = \cos(2\pi x)\cos(\sin(2\pi y)), \quad\theta(x,y,0) = 2 + \sin(2\pi x)\cos(2\pi y).
\end{aligned}
\end{equation}
We impose periodic boundary conditions and compute until  $t = 0.01$ to avoid the appearance of the shocks in the solution. The reference solution is obtained by the third-order still-water equilibria preserving DG scheme \cite{qian2018high} on 800$\times$800 uniform cells. The $L^1$ errors and convergence rates for the conservative variables $h$, $hu$, $hv$, $h\theta$ as well as the equilibrium variables $E$, $\theta$ are listed in Table \ref{Ripa:DGacc2D}, which show the high order accuracy as we expected.
\begin{table}[htb]
  \centering
  \caption{Example \ref{Ripa:accuracy2D}:  $L^1$ errors and numerical orders of accuracy with initial condition (\ref{Ripa:acc2d}), $t = 0.01$.}\label{Ripa:DGacc2D}
  \begin{tabular}{c c c c c c c c }
  \toprule
    \multirow{2}{*} {$nx\times ny$} & \multicolumn{2}{c}{$h$}&\multicolumn{2}{c}{$ h u $}&\multicolumn{2}{c}{$ h v$}\\
   \cmidrule(lr){2-3} \cmidrule(lr){4-5} \cmidrule(lr){6-7}
    ~ &$L^1$ error &order&$L^1$ error &order&$L^1$ error &order\\
   \midrule
   25$\times$25  & 6.98E{-04} &     -- &  6.77E{-03} &  --  &  7.14E{-03} &  --  \\
   50$\times$50  & 9.83E{-05} &   2.83 &  8.57E{-04} & 2.98 &  9.08E{-04} & 2.98 \\
  100$\times$100 & 1.37E{-05} &   2.84 &  1.11E{-04} & 2.95 &  1.18E{-04} & 2.95 \\
  200$\times$200 & 1.89E{-06} &   2.86 &  1.49E{-05} & 2.89 &  1.59E{-05} & 2.88 \\
  400$\times$400 & 2.48E{-07} &   2.93 &  2.01E{-06} & 2.89 &  2.17E{-06} & 2.88 \\
   \midrule
   \multirow{2}{*} {$nx\times ny$} & \multicolumn{2}{c}{$h\theta$}&\multicolumn{2}{c}{$ E $}&\multicolumn{2}{c}{$ \theta$}\\
   \cmidrule(lr){2-3} \cmidrule(lr){4-5} \cmidrule(lr){6-7}
    ~ &$L^1$ error &order&$L^1$ error &order&$L^1$ error &order\\
   \midrule
     25$\times$25  & 1.48E{-03} &     -- &  2.26E{-04} &    --  &  1.53E{-02} &     --  \\
     50$\times$50  & 1.93E{-04} &   2.94 &  3.10E{-05} &   2.86 &  2.03E{-03} &   2.91 \\
    100$\times$100 & 2.53E{-05} &   2.93 &  4.28E{-06} &   2.86 &  2.67E{-04} &   2.92 \\
    200$\times$200 & 3.39E{-06} &   2.90 &  6.04E{-07} &   2.83 &  3.49E{-05} &   2.93 \\
    400$\times$400 & 4.51E{-07} &   2.91 &  8.08E{-08} &   2.90 &  4.47E{-06} &   2.97 \\
  \bottomrule
  \end{tabular}
\end{table}

\begin{example}{\bf Test for two-dimensional still water well-balanced  property}\label{Ripa:exactC2D_still}
\end{example}
To verify our developed DG method can be extended to two dimensions and indeed maintain the still water well-balanced property over a non-flat bottom, we choose the bottom topography on a domain $[-1,1]\times[-1,1]$,
\begin{equation}\label{Ripa:smobot2D}
  b(x,y)=\left\{\begin{array}{lll}
    0.5e^{-100\left((x+0.5)^2+(y+0.5)^2\right)},&\text{if} \ x<0,\\
    0.6e^{-100\left((x-0.5)^2+(y-0.5)^2\right)},&\text{otherwise},\\
    \end{array}\right.
\end{equation}
and the initial conditions
\begin{equation}\label{Ripa:smoini2D}
 h(x,y,0)+b(x,y) = 3,\quad hu(x,y,0)=0, \quad hv(x,y,0)=0, \quad \theta(x,y,0) = \dfrac{4}{3}.
\end{equation}
We apply our resulting scheme to calculate the solutions up to  $t = 0.1$ with 20$\times $20 and 100$\times $100 uniform rectangular meshes respectively. The $L^1$ and $L^{\infty}$ errors at double precision for the conservative variables $h$, $hu$, $hv$, $h\theta$ as well as the equilibrium variables $E$, $\theta$ are listed in Table \ref{Ripa:wbstill_smo2D}. It can be seen clearly that even under a relatively coarse mesh, the hydrostatic equilibrium state can be exactly preserved in two dimensions.

\begin{table}[htb]
  \centering
  \caption{Example \ref{Ripa:exactC2D_still}: $L^1$ and $L^{\infty}$ errors for two-dimensional still water equilibrium state.}\label{Ripa:wbstill_smo2D}
  \begin{tabular}{c c c c c c c c}
   \toprule
    {error}&{$nx\times ny$}&{$h$}&{$hu$}&{$hv$}&{$h\theta$}&{$E$}&{$\theta$}\\
    \midrule
   \multirow{2}{*} {$L^1$}        &  20$\times$20  & 1.55E{-13} &  1.09E{-14} & 1.08E{-14} & 7.44E{-14} & 7.84E{-14}  & 9.13E{-14} \\
                   ~              & 100$\times$100 & 6.05E{-13} &  2.72E{-14} & 2.68E{-14} & 6.17E{-13} & 6.21E{-13}  & 4.63E{-13} \\
   \cmidrule(lr){2-8}
   \multirow{2}{*} {$L^{\infty}$} &  20$\times$20  & 5.73E{-14} &  9.24E{-15} & 1.04E{-14} & 3.51E{-14} & 3.51E{-14}  & 3.00E{-14} \\
                   ~              & 100$\times$100 & 2.38E{-13} &  3.18E{-14} & 3.24E{-14} & 2.98E{-13} & 3.04E{-13}  & 1.49E{-13} \\
    \bottomrule
  \end{tabular}
\end{table}

\begin{example}{\bf Small perturbation of two-dimensional still water equilibrium}\label{Ripa:perturbation2D}
\end{example}
This example we test is the extension of the classical shallow water equations problem given by LeVeque \cite{leveque1998balancing}, which shows the capability of our proposed well-balanced DG scheme for the perturbation of the hydrostatic state in two dimensions Ripa model. Initial conditions are given by
\begin{equation}\label{Ripa:per_ht2d}
(h,u,v,\theta)(x,y,0) = \left\{\begin{array}{lll}
    (6-b(x,y)+A, 0, 0, \frac{24}{6+A}),&  \text{if} \  0.05\leq x \leq 0.15, \\
    (6-b(x,y),\quad \quad 0, 0, 4), &\text{otherwise},\\
    \end{array}\right.
\end{equation}
with an elliptical hump bottom topography
$$b(x,y)= 3e^{-5(x-0.9)^2-50(y-0.5)^2},$$
on the computational domain $[-2,2]\times [0,1]$. The disturbance is imposed on both the water height and the temperature, and the small pulse $A= 0.001$ is tested with reflective boundary conditions. We compute the solutions on 200$\times$100 cells at times $t=0.05,0.08$, and present the contours of the surface level $h+b$, discharge $hu$, $hv$ and $h\theta$ in Fig. \ref{Ripa:smallper_2D}. The corresponding results of the surface level $h+b$ and discharge $hu$ along the line $y=0.5$ with 200$\times$100 and 400$\times$200 for comparison are displayed in Figs. \ref{Ripa:small5comp_2D_y0.5} and \ref{Ripa:small8comp_2D_y0.5}, respectively. We can observe that the centered contact discontinuity wave remains unmoved in the perturbed region over time, and its amplitude reduces to approximately 6.0005. It illustrates that our scheme can achieve high-resolution interface tracking and capture the complex small features of the flow very well.
\begin{figure}[htb!]
  \centering
  \subfigure[ surface level $h+b$ ]{
  \centering
     \includegraphics[width= 7.5cm,scale=1]{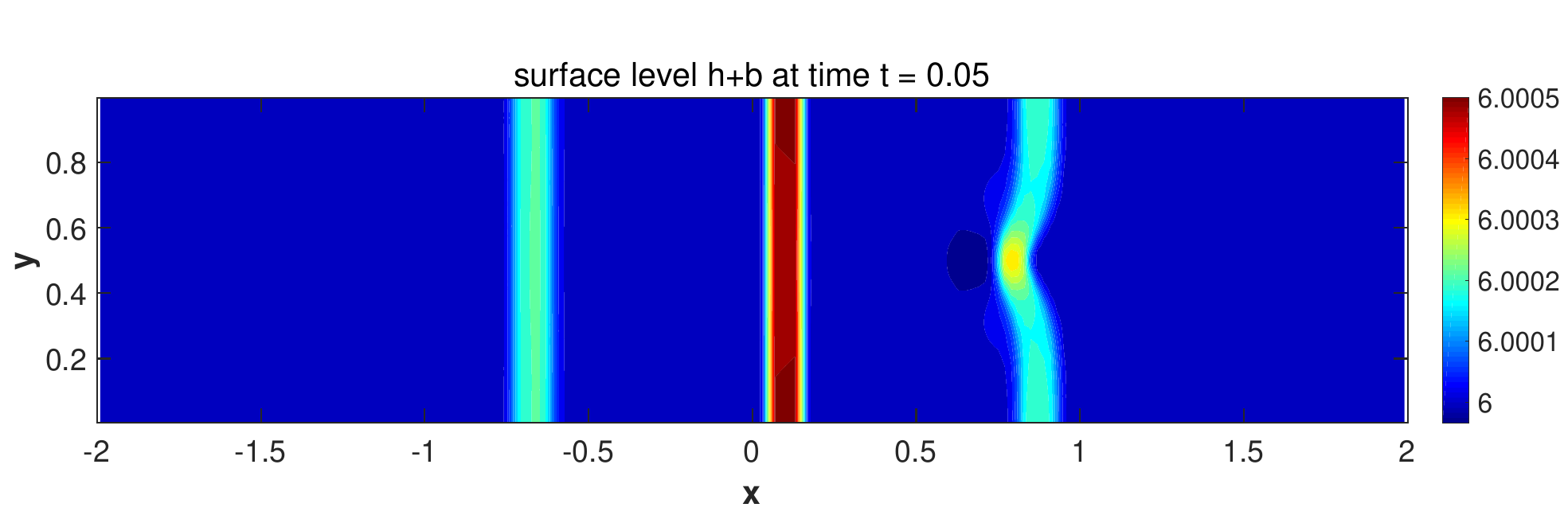}
  }
  \subfigure[ surface level $h+b$ ]{
  \centering
     \includegraphics[width= 7.5cm,scale=1]{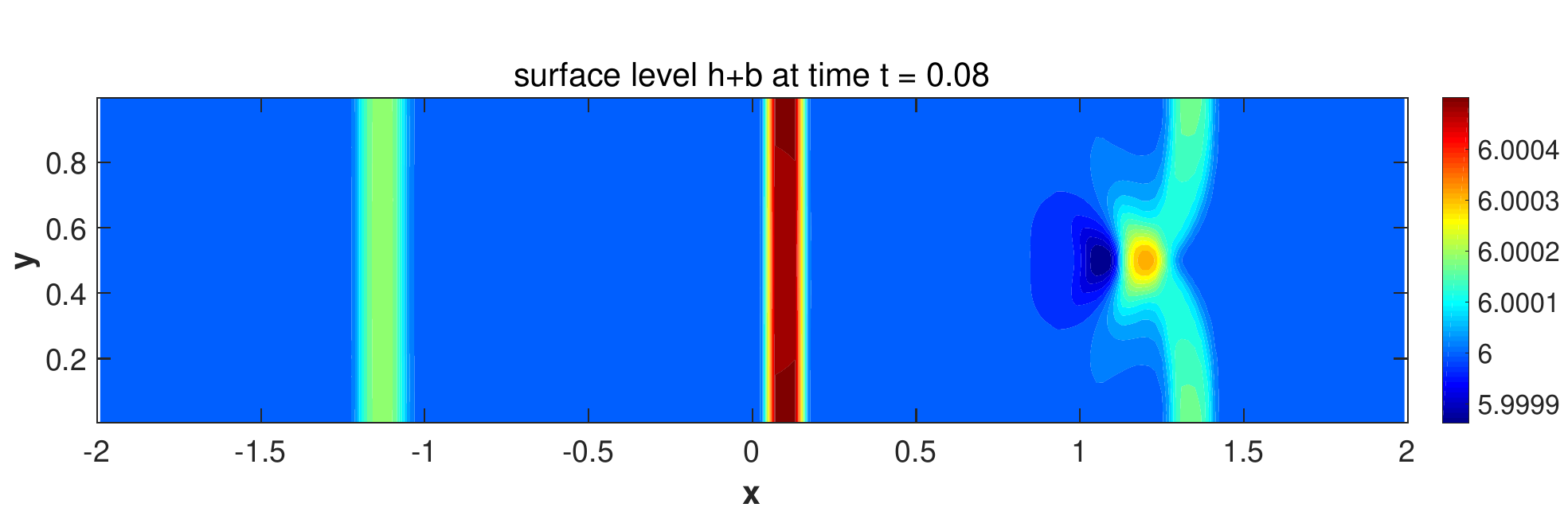}
  }
  \subfigure[discharge $hu$]{
  \centering
     \includegraphics[width= 7.5cm,scale=1]{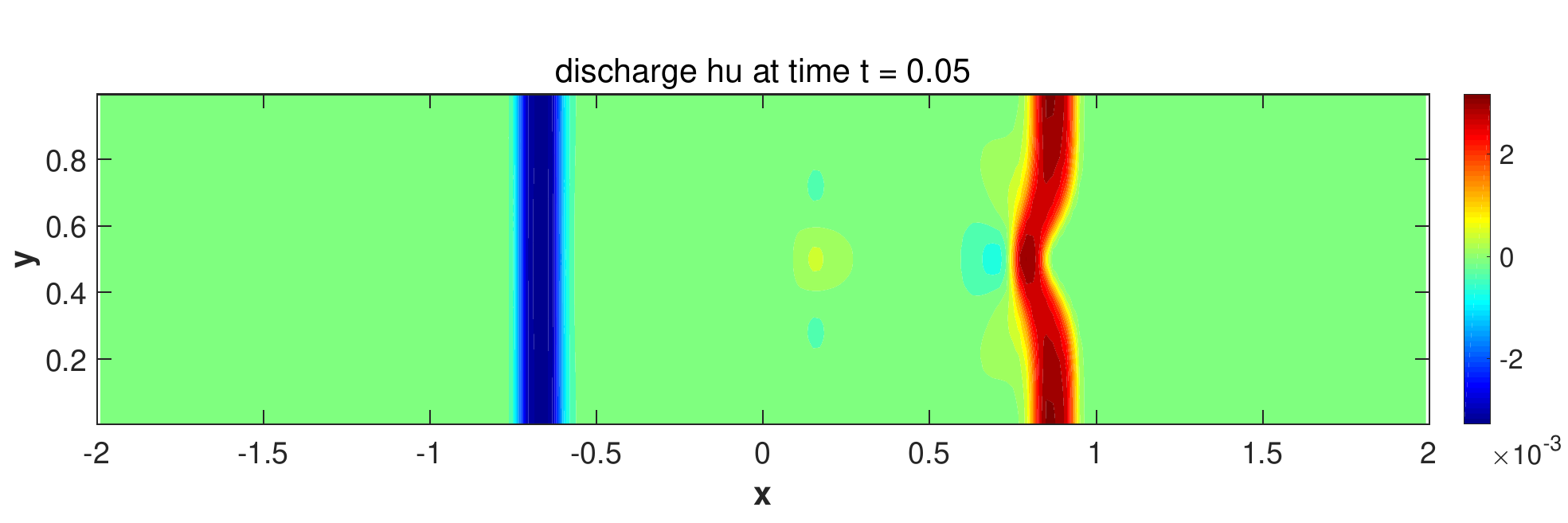}
  }
  \subfigure[discharge $hu$]{
  \centering
     \includegraphics[width= 7.5cm,scale=1]{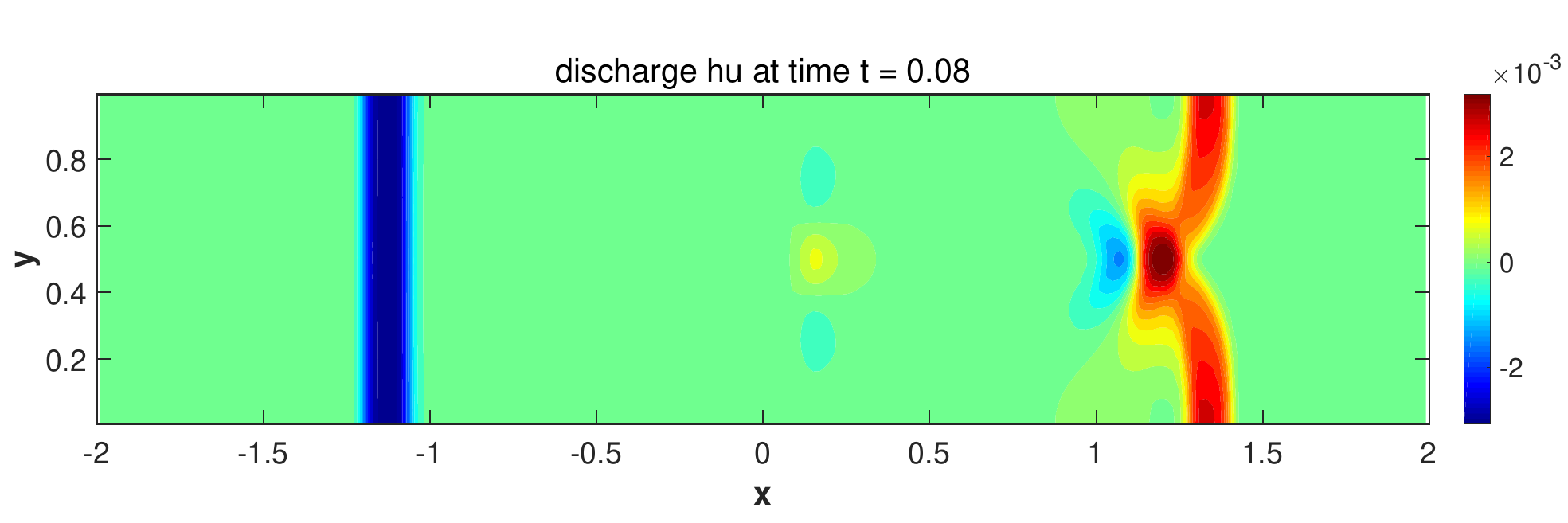}
  }
  \subfigure[discharge $hv$]{
  \centering
     \includegraphics[width= 7.5cm,scale=1]{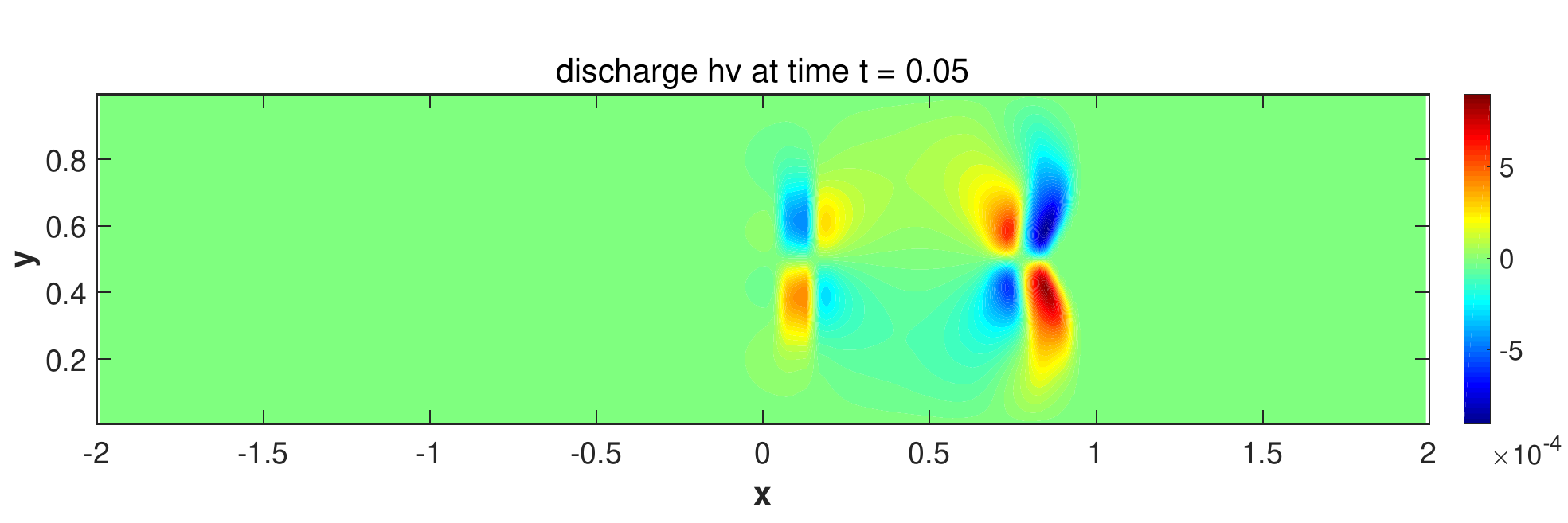}
  }
  \subfigure[discharge $hv$]{
  \centering
     \includegraphics[width= 7.5cm,scale=1]{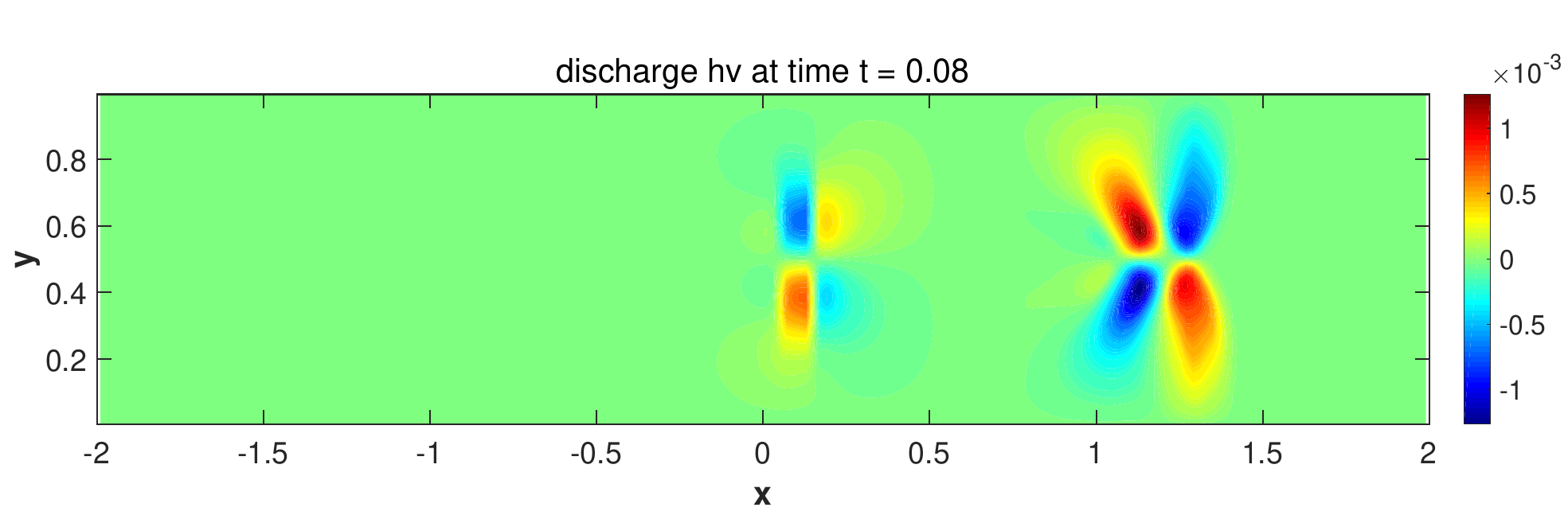}
  }
  \subfigure[$h\theta$]{
  \centering
     \includegraphics[width= 7.5cm,scale=1]{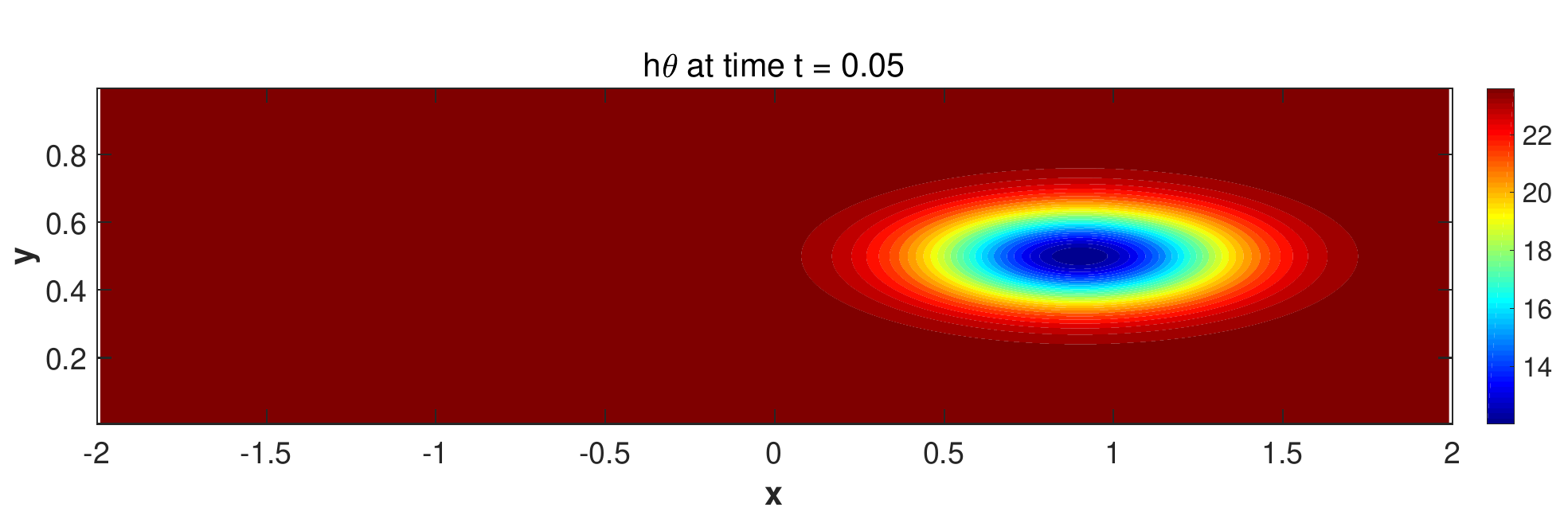}
  }
  \subfigure[$h\theta$]{
  \centering
     \includegraphics[width= 7.5cm,scale=1]{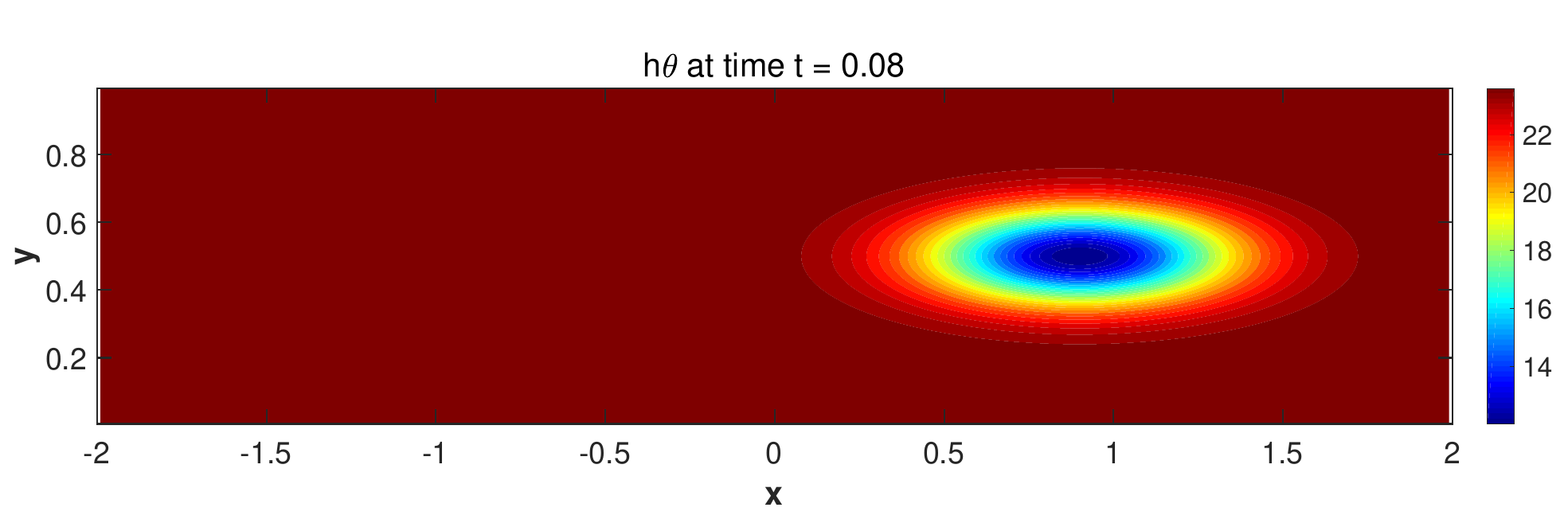}
  }
  \caption{Example \ref{Ripa:perturbation2D}:  From top to bottom: the contours of the surface level $h+b$ from 5.9998 to 6.0005; discharge $hu$ from -0.004 to 0.004; discharge $hv$ from -0.0015 to 0.0015; $h\theta$ from 12 to 24.001. 30 uniformly spaced contour lines at different times $t=0.05$ (left) and $t=0.08$ (right) with 200$\times$100 cells. }\label{Ripa:smallper_2D}
\end{figure}

\begin{figure}[htb!]
  \centering
  \subfigure[ surface level $h+b$  ]{
  \centering
     \includegraphics[width= 5.5cm,scale=1]{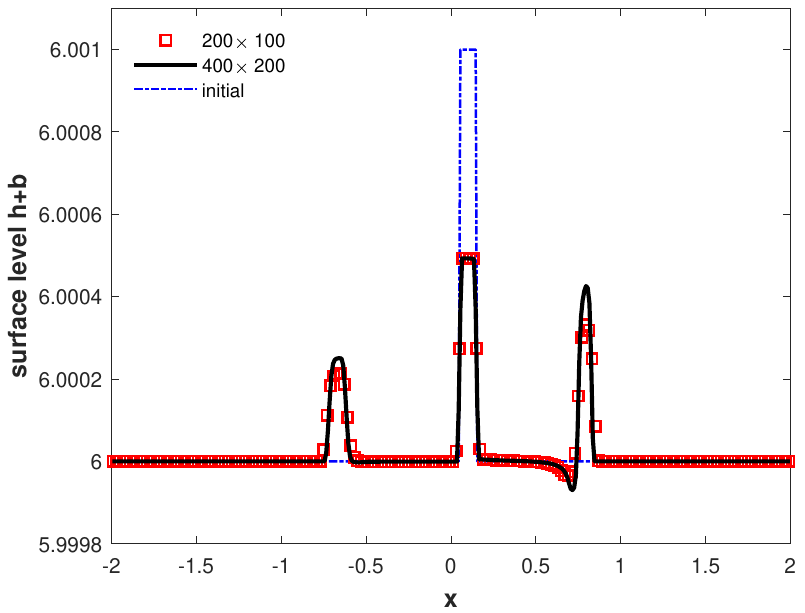}
  }
  \subfigure[discharge $hu$]{
  \centering
     \includegraphics[width= 5.5cm,scale=1]{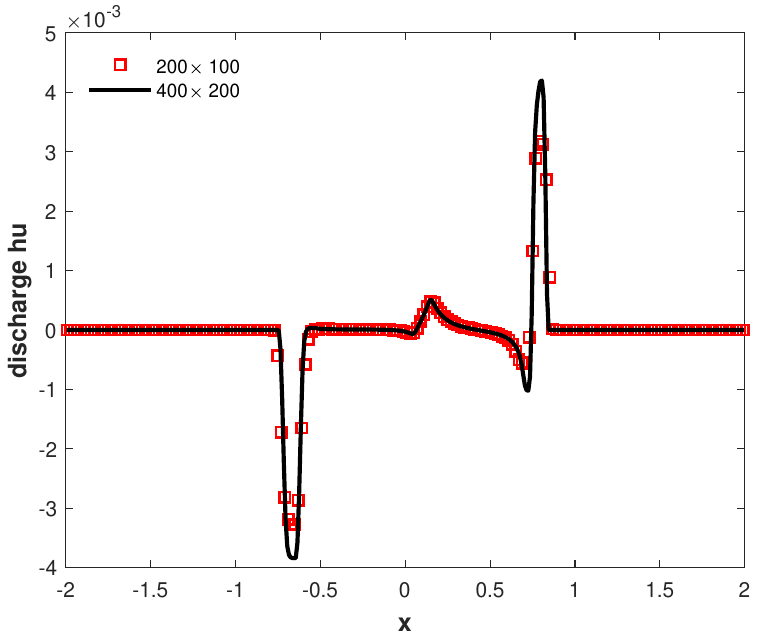}
  }
  \caption{Example \ref{Ripa:perturbation2D}:  Numerical solutions of the surface level $h+b$ and discharge $hu$ with initial value (\ref{Ripa:per_ht2d}) along the line $y=0.5$ at time $t=0.05$,  using 200$\times$100 and 400$\times$200  cells for comparison. }\label{Ripa:small5comp_2D_y0.5}
\end{figure}

\begin{figure}[htb!]
  \centering
  \subfigure[ surface level $h+b$  ]{
  \centering
     \includegraphics[width= 5.5cm,scale=1]{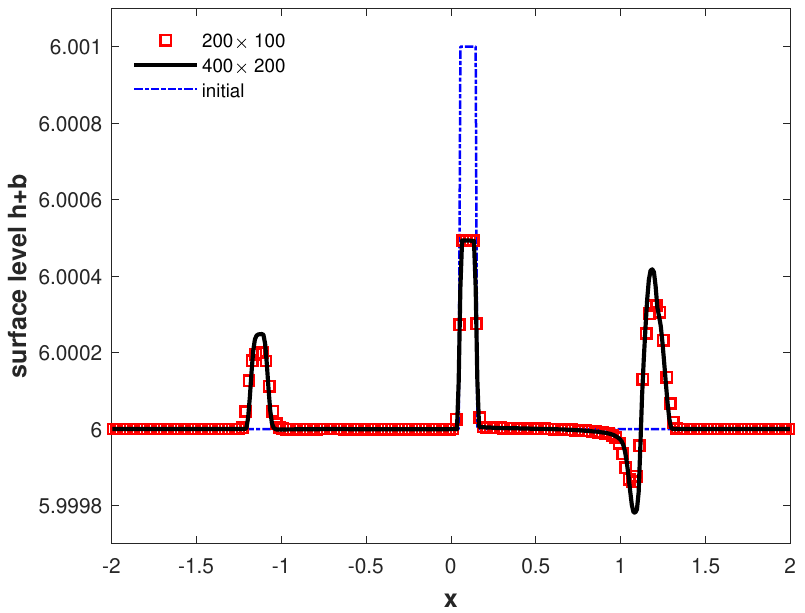}
  }
  \subfigure[discharge $hu$]{
  \centering
     \includegraphics[width= 5.5cm,scale=1]{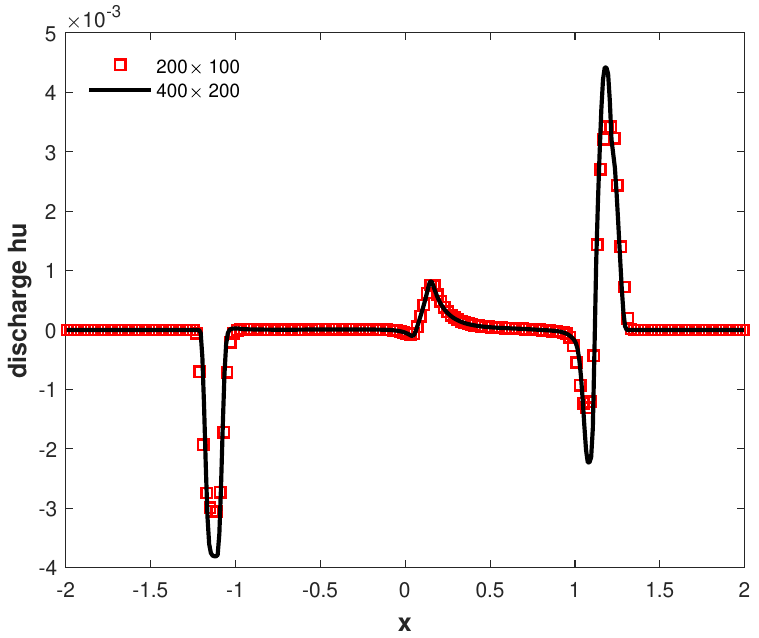}
  }
  \caption{Example \ref{Ripa:perturbation2D}:  Numerical solutions of the surface level $h+b$ and discharge $hu$ with initial value (\ref{Ripa:per_ht2d}) along the line $y=0.5$ at time $t=0.08$,  using 200$\times$100 and 400$\times$200 cells for comparison. }\label{Ripa:small8comp_2D_y0.5}
\end{figure}

\begin{example}{\bf The radial dam break over a flat bottom topography}\label{Ripa:flatdam2D}
\end{example}
We test the same dam breaking problem taken from \cite{touma2015well} over a flat bottom topography $b(x,y)=0$. The initial conditions are defined as
\begin{equation}\label{Ripa:radial_dam2d}
(h,u,v,\theta)(x,y,0) = \left\{\begin{array}{lll}
    (2,0,0,1),&  \text{if} \  x^2+y^2 \leq 0.25, \\
    (1,0,0,1.5),&\text{otherwise},\\
    \end{array}\right.
\end{equation}
on the domain $[-1,1]\times [-1,1]$. Transmissive boundary conditions are employed and the stopping time is set as $t=0.05$.
We compute the solutions on 100$\times$100 uniform cells, and display the profiles of the water height $h$, temperature $\theta$ and pressure $\frac{1}{2}gh^2\theta$ in Fig. \ref{Ripa:damflat2D}, as well as the scatter plot of the cross sections of the water height along the line $x=0$ and $y=x$ in Fig. \ref{Ripa:damflat2D_x0}. Oscillation-free solutions can be observed by our well-balanced DG method for the two-dimensional Ripa model.

\begin{figure}[htb!]
  \centering
  \subfigure[ water height $h$ ]{
  \centering
     \includegraphics[width= 4.8cm,scale=1]{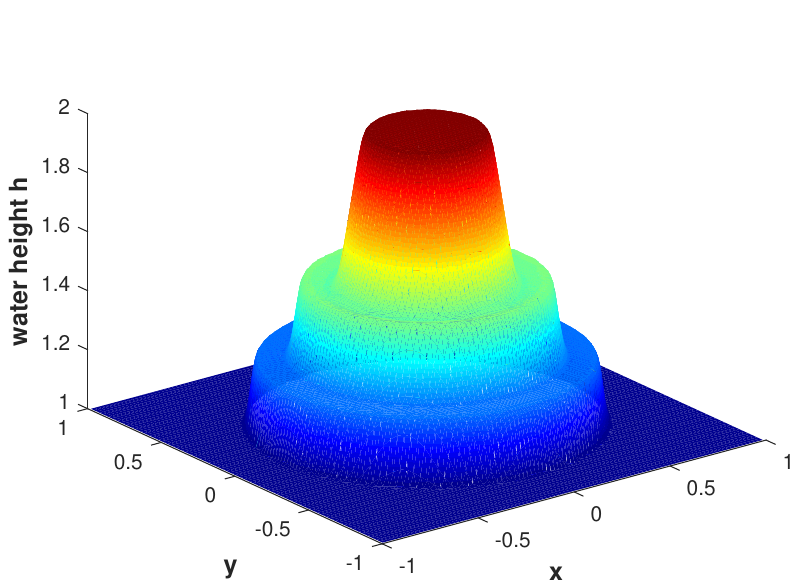}
  }
  \subfigure[temperature $\theta$]{
  \centering
     \includegraphics[width= 4.8cm,scale=1]{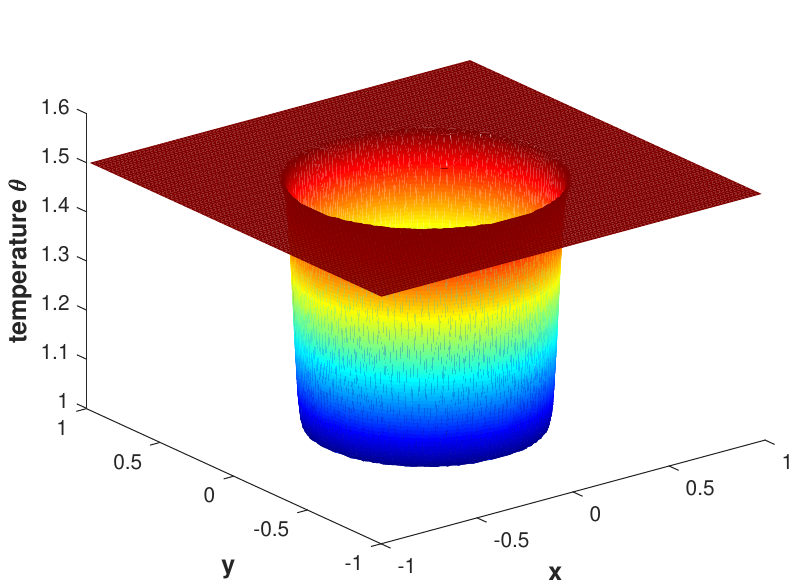}
  }
  \subfigure[pressure $\frac{1}{2}gh^2\theta$]{
  \centering
     \includegraphics[width= 4.8cm,scale=1]{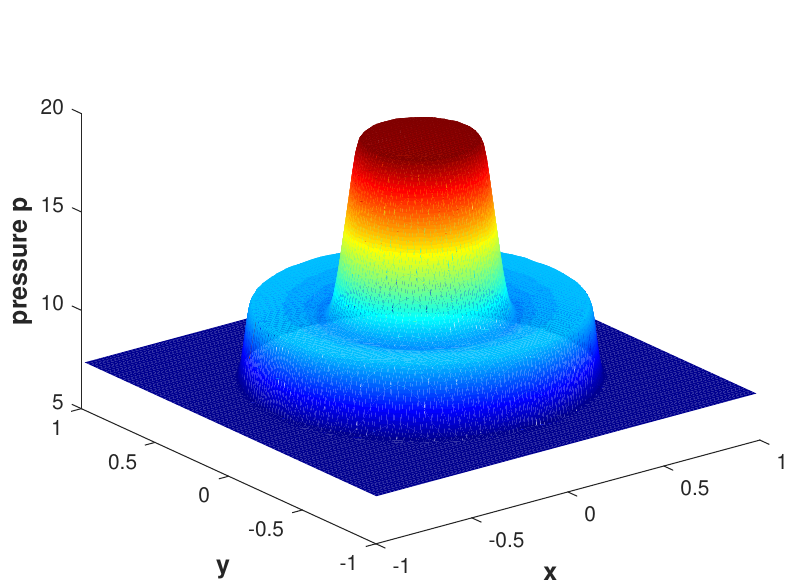}
  }
  \caption{Example \ref{Ripa:flatdam2D}: From left to right: 3D view of the water height $h$, temperature $\theta$ and pressure $\frac{1}{2}gh^2\theta$, with the initial value (\ref{Ripa:radial_dam2d}) at time $t=0.05$, using $100\times 100$ cells.}\label{Ripa:damflat2D}
\end{figure}

\begin{figure}[htb!]
  \centering
  \subfigure[ water height $h$, $x=0$]{
  \centering
     \includegraphics[width= 5.5cm,scale=1]{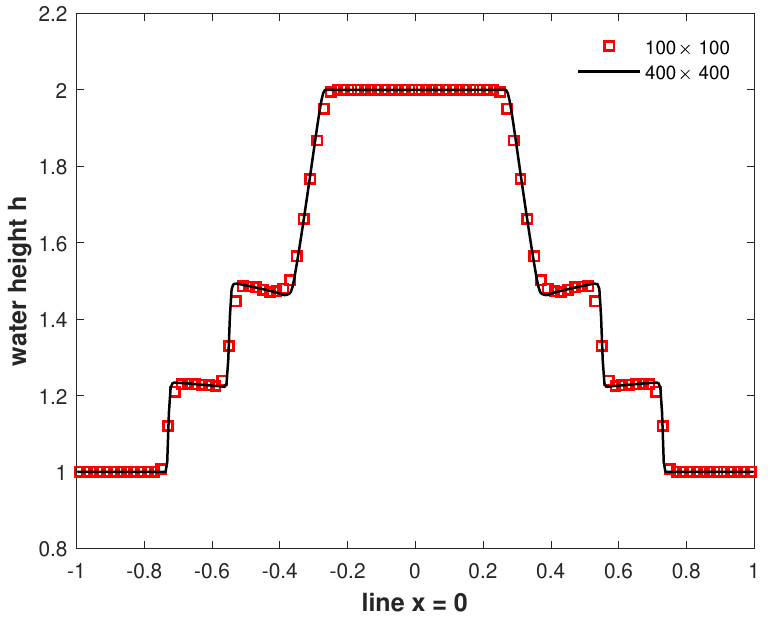}
  }
  \subfigure[ water height $h$, $y=x$]{
  \centering
     \includegraphics[width= 5.5cm,scale=1]{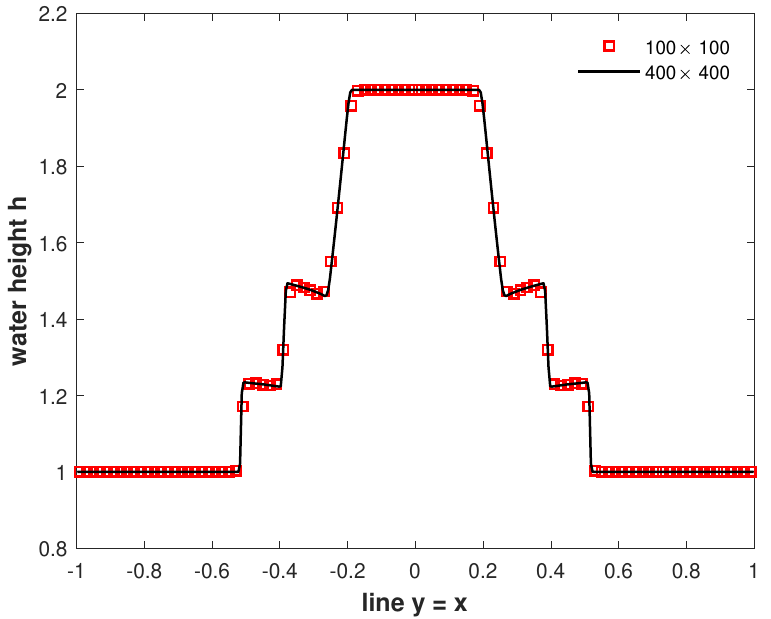}
  }
  \caption{Example \ref{Ripa:flatdam2D}:  Scatter plot of the cross sections of the water height along the line $x=0$ (left), $y=x$ (right), with the initial value (\ref{Ripa:radial_dam2d}) at time $t=0.05$, using $100\times 100$ and $400\times 400$ cells for comparison.}\label{Ripa:damflat2D_x0}
\end{figure}

\begin{example}{\bf Steady state solution}\label{Ripa:steady2D}
\end{example}
In this example, we consider a steady state solution which consists of two still water states \eqref{still_water:Ripa} connected through the temperature jump. On the computational domain $[-1,1]\times[-1,1]$, the initial conditions are given by
\begin{equation}\label{wbli2d}
(h,u,v,\theta)(x,y,0) = \left\{\begin{array}{lll}
    (3-b(x,y), 0, 0, \frac{4}{3}),&  \text{if} \  x^2+y^2 \leq 0.25, \\
    (2-b(x,y), 0, 0, 3), &\text{otherwise},\\
    \end{array}\right.
\end{equation}
where the bottom topography is the same as in (\ref{Ripa:smobot2D}). Here the gravitation constant is taken as $g = 1$. For comparison, we run the simulation on $100\times 100$ uniform cells by the still water equilibria preserving DG method \cite{qian2018high} and isobaric equilibria preserving DG method \eqref{scheme:ripaiso2d}-\eqref{flux:ripaiso2d}, respectively. The final time is set as $t=0.12$, and the results are plotted in Fig. \ref{Ripa:wbli2d_2D}. We can observe the circular-shape pressure oscillations by the `wb$\_$still' scheme and the $\theta$ component is over smeared around the contact wave. By contrast,  the isobaric equilibria preserving DG scheme preserves the steady state without spurious oscillations.

\begin{figure}[htb!]
  \subfigure[ surface level $h+b$ ]{
  \centering
     \includegraphics[width= 3.5cm,scale=1]{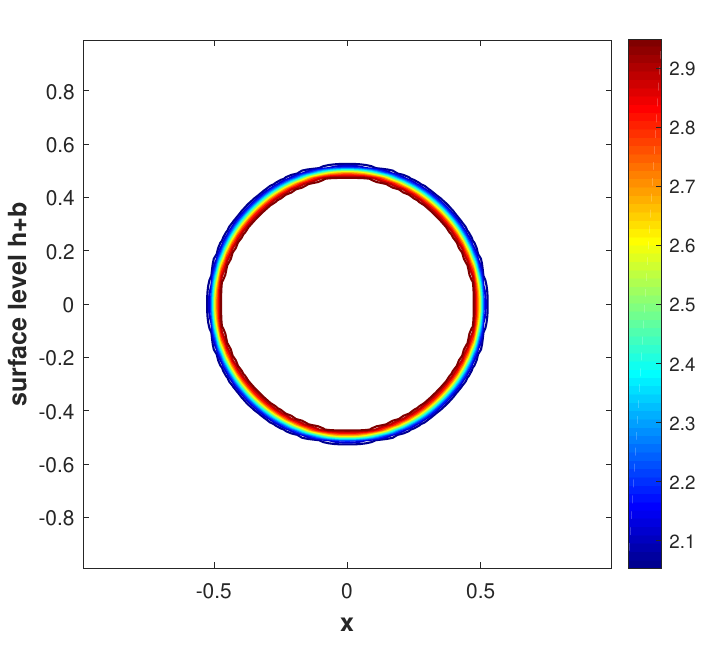}
  }
  \subfigure[temperature $\theta$]{
  \centering
     \includegraphics[width= 3.5cm,scale=1]{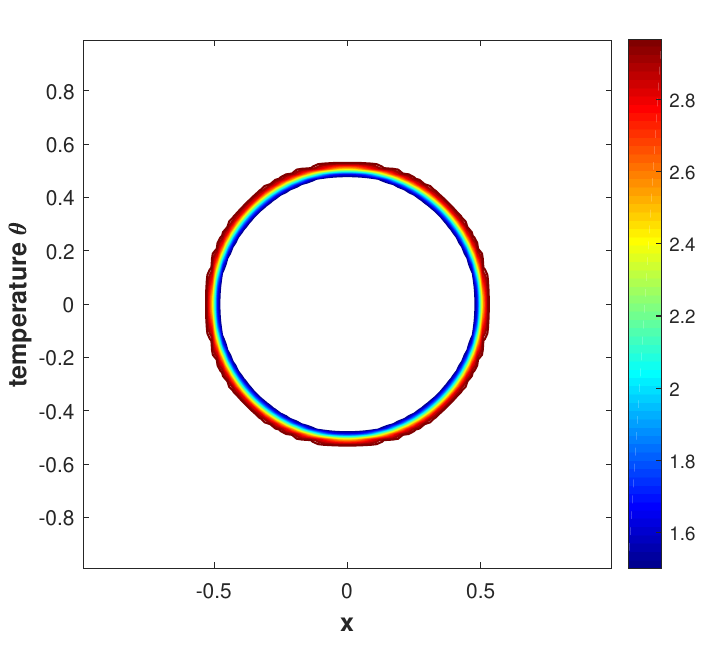}
  }
  \subfigure[pressure $\frac{1}{2}gh^2\theta$]{
  \centering
     \includegraphics[width= 3.5cm,scale=1]{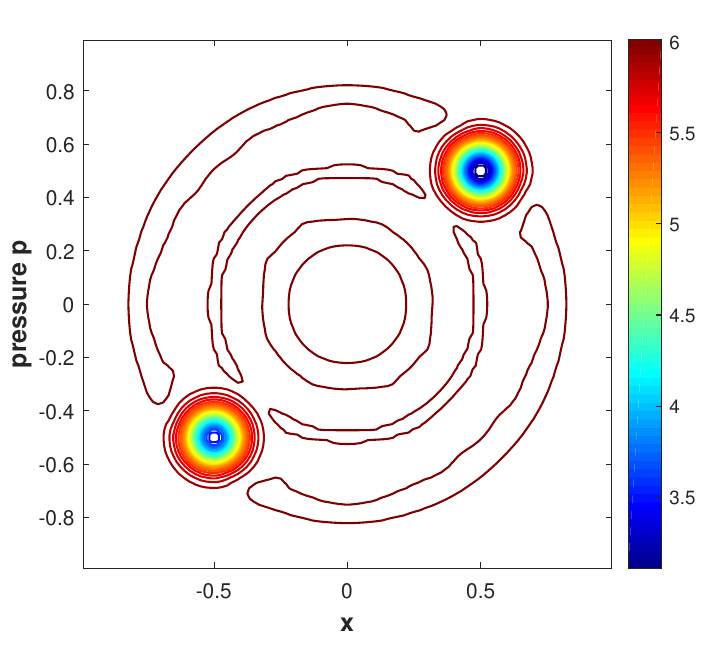}
  }
  \subfigure[pressure $\frac{1}{2}gh^2\theta$]{
  \centering
     \includegraphics[width= 3.5cm,scale=1]{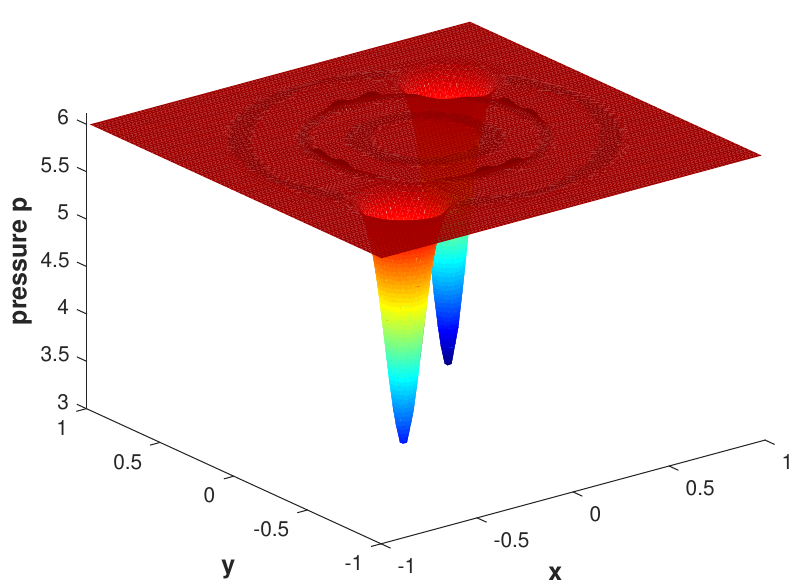}
  }

  \centering
  \subfigure[ surface level $h+b$ ]{
  \centering
     \includegraphics[width= 3.5cm,scale=1]{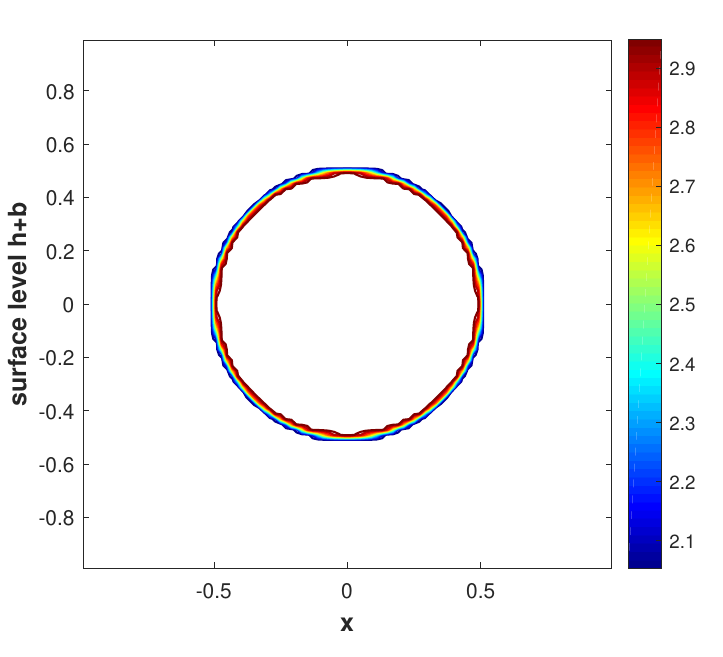}
  }
  \subfigure[temperature $\theta$]{
  \centering
     \includegraphics[width= 3.5cm,scale=1]{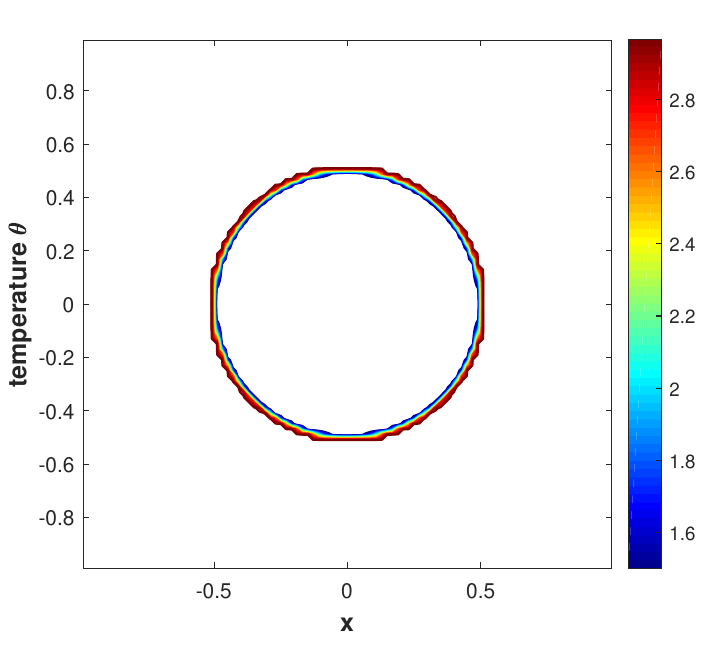}
  }
  \subfigure[pressure $\frac{1}{2}gh^2\theta$]{
  \centering
     \includegraphics[width= 3.5cm,scale=1]{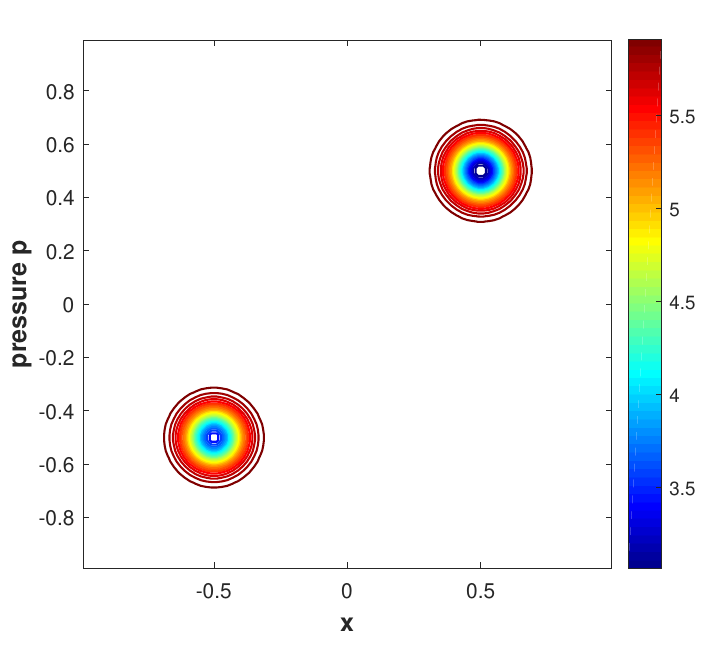}
  }
  \subfigure[pressure $\frac{1}{2}gh^2\theta$]{
  \centering
     \includegraphics[width= 3.5cm,scale=1]{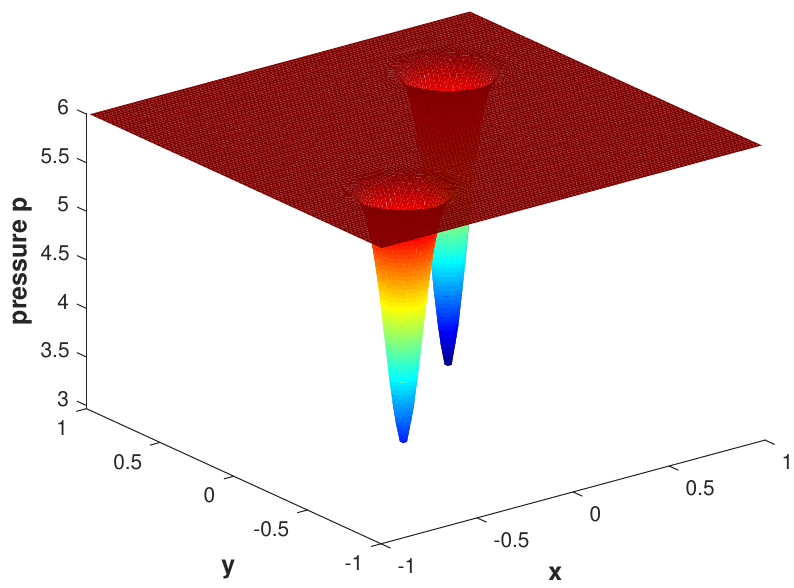}
  }
  \caption{Example \ref{Ripa:steady2D}: Steady state solution with initial value (\ref{wbli2d}) at time $t=0.12$, using $100\times 100$ cells. From left to right: contour plots of the surface level, temperature, and pressure, 3D view of pressure. From top to bottom: still water and isobaric equilibria preserving DG scheme. }\label{Ripa:wbli2d_2D}
\end{figure}

Next, we add a perturbation to the water height
\begin{equation}\label{perli2d}
(h,u,v,\theta)(x,y,0) = \left\{\begin{array}{lll}
    (3-b(x,y)+A, 0, 0, \frac{4}{3}),&\text{if} \  0.01 \leq x^2+y^2 \leq 0.09, \\
    (3-b(x,y), 0, 0, \frac{4}{3}),&\text{if} \  x^2+y^2<0.01 \ \text{or} \ 0.09 < x^2+y^2 < 0.25, \\
    (2-b(x,y), 0, 0, 3), &\text{otherwise},\\
    \end{array}\right.
\end{equation}
with the parameter $A=0.1$. We simulate up to $t=0.15$ with $100\times 100$ uniform cells. The solutions computed by the same two well-balanced DG methods are presented in Fig. \ref{Ripa:perli2d_2D}. For the `wb$\_$still' scheme, the interaction between the circular-shaped pressure oscillations and the perturbation leads to the appearance of parasitic waves. While the isobaric equilibria preserving scheme captures the perturbation more accurately.

\begin{figure}[htb!]
  \centering
  \subfigure[ surface level $h+b$ ]{
  \centering
     \includegraphics[width= 3.5cm,scale=1]{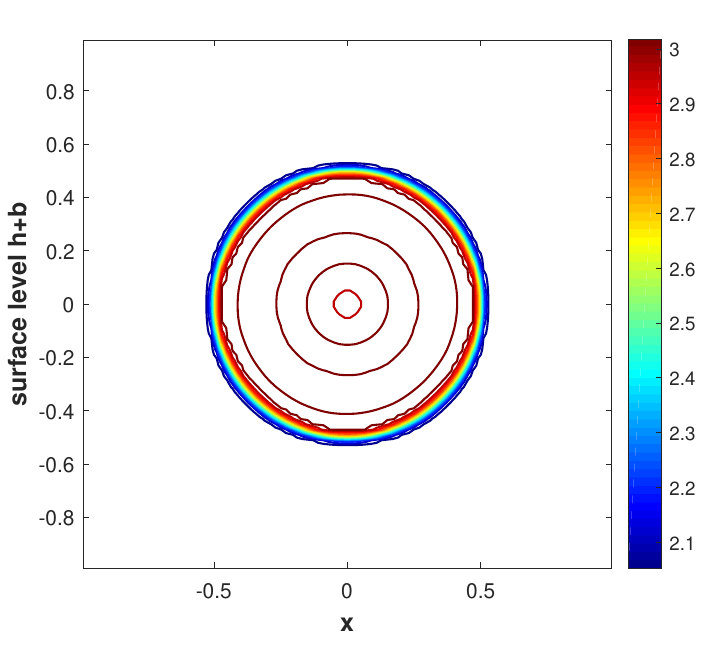}
  }
  \subfigure[temperature $\theta$]{
  \centering
     \includegraphics[width= 3.5cm,scale=1]{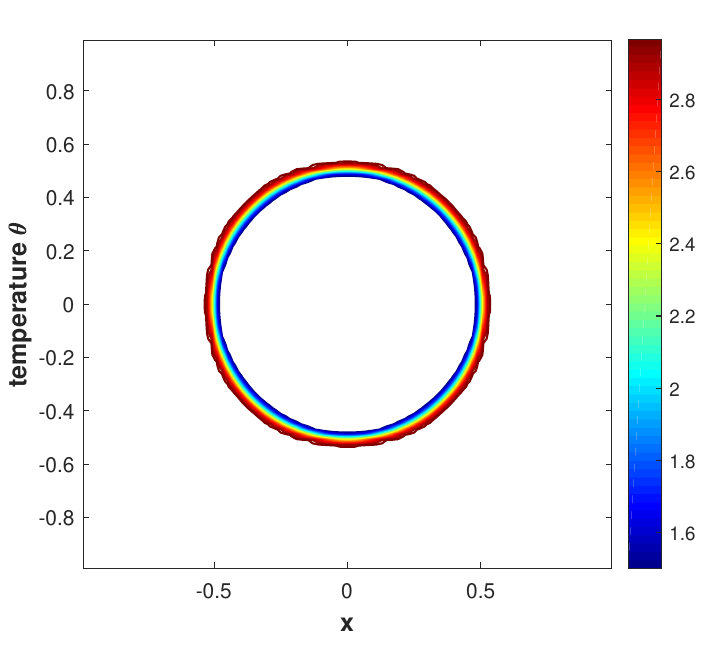}
  }
  \subfigure[pressure $\frac{1}{2}gh^2\theta$]{
  \centering
     \includegraphics[width= 3.5cm,scale=1]{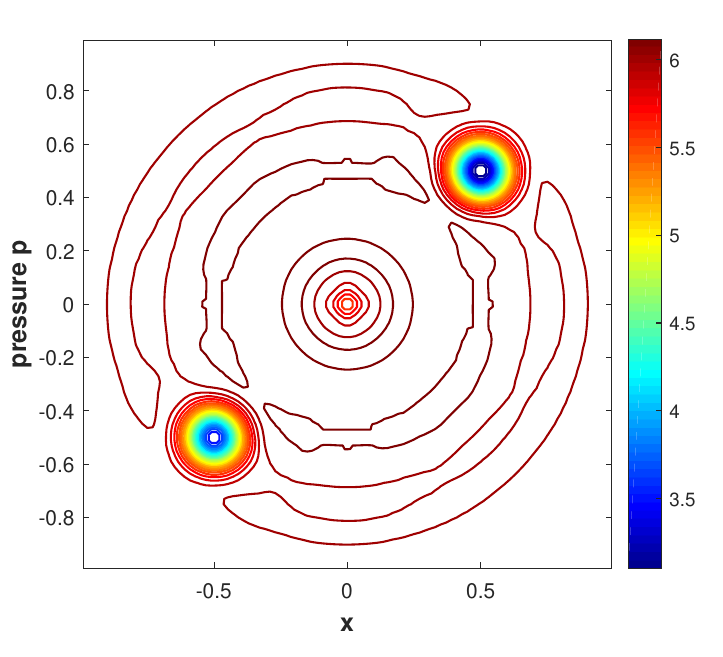}
  }
  \subfigure[pressure $\frac{1}{2}gh^2\theta$]{
  \centering
     \includegraphics[width= 3.5cm,scale=1]{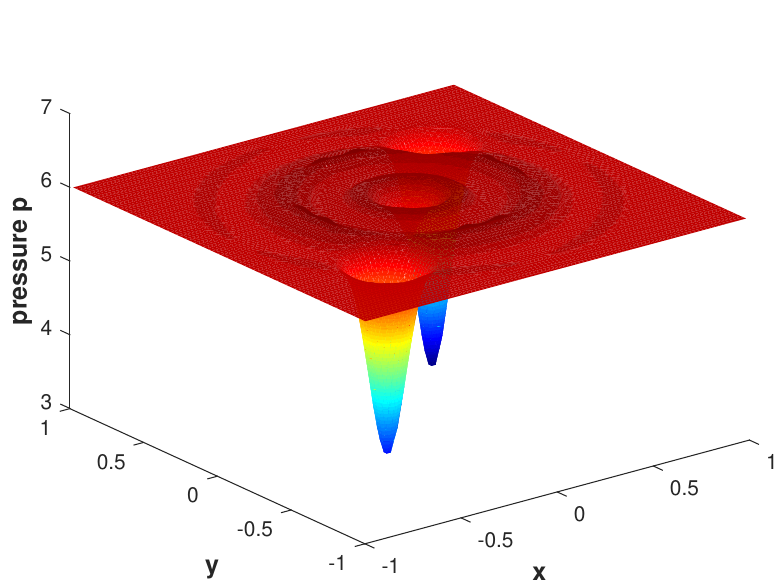}
  }

  \subfigure[ surface level $h+b$ ]{
  \centering
     \includegraphics[width= 3.5cm,scale=1]{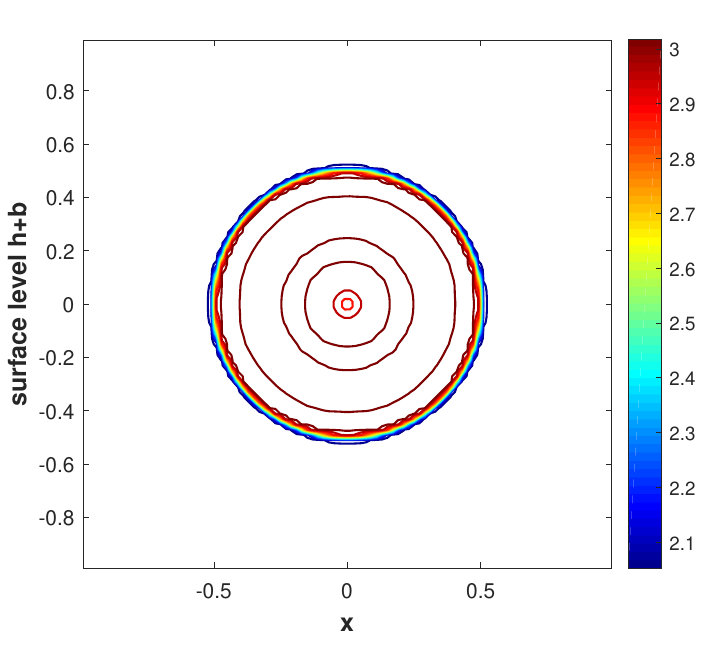}
  }
  \subfigure[temperature $\theta$]{
  \centering
     \includegraphics[width= 3.5cm,scale=1]{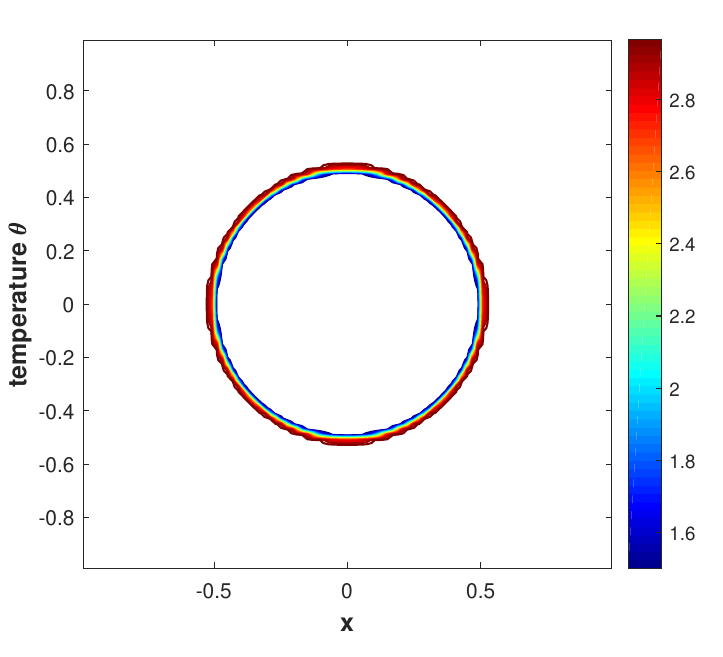}
  }
  \subfigure[pressure $\frac{1}{2}gh^2\theta$]{
  \centering
     \includegraphics[width= 3.5cm,scale=1]{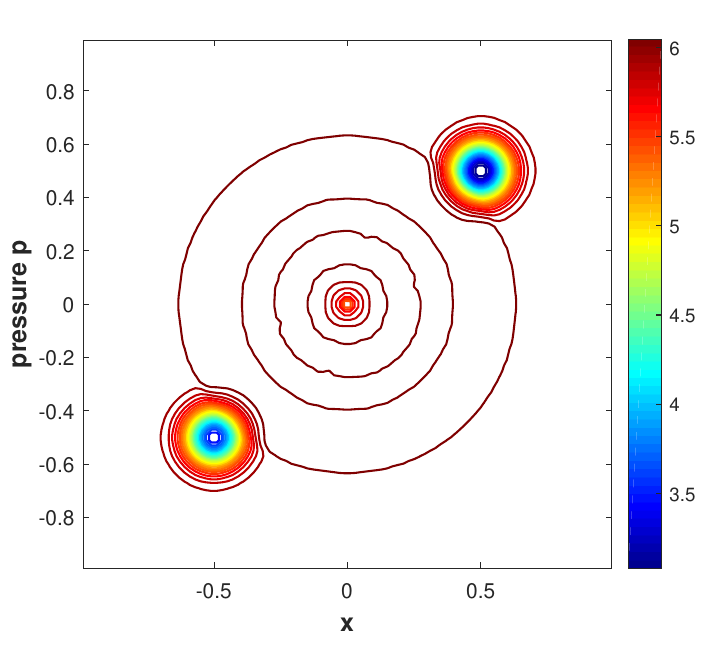}
  }
  \subfigure[pressure $\frac{1}{2}gh^2\theta$]{
  \centering
     \includegraphics[width= 3.5cm,scale=1]{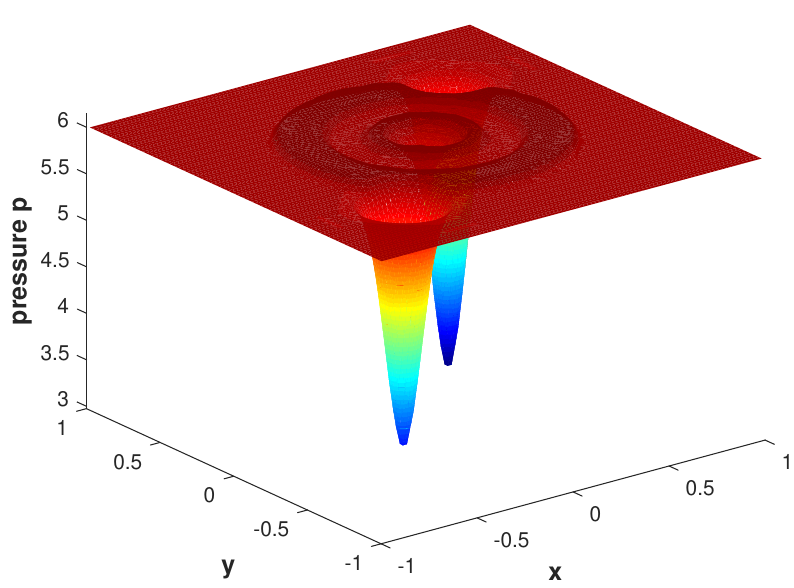}
  }
  \caption{Example \ref{Ripa:steady2D}: Small perturbation of the steady state solution with initial value (\ref{perli2d}) at time $t=0.15$, using $100\times 100$ cells. From left to right: contour plots of the surface level, temperature and pressure, 3D view of pressure. From top to bottom: still water and isobaric equilibria preserving DG scheme.}\label{Ripa:perli2d_2D}
\end{figure}

\section{Conclusion}\label{se:co}
We have developed a novel framework for designing the well-balanced discontinuous Galerkin method for hyperbolic balance laws with the equilibrium preserving space. This is achieved by approximating the equilibrium variables using piecewise polynomial spaces in the DG method and employing hydrostatic reconstruction for well-balanced numerical fluxes. The well-balanced property is thus ensured for general equilibrium steady states, without the need for reference equilibrium state recovery or the ODE solver for a stationary solution. The scheme has been theoretically analyzed and numerically validated under various conditions, demonstrating its ability to accurately preserve both hydrostatic and moving equilibrium states.  Additionally, it exhibits high-order accuracy and essentially non-oscillatory properties, and performs well in capturing small perturbations of nearly equilibria flow on coarse meshes. We anticipate that this framework can be extended for the development of equilibrium-preserving methods in general scenarios. In future work, we plan to explore the corresponding equilibrium states with general convex equations in different coordinate systems.

\section*{Declarations}

\subsection*{Data Availability}
The datasets generated during the current study are available from the corresponding author upon reasonable request. They support
our published claims and comply with field standards.

\subsection*{Competing of interest}
The authors declare that they have no known competing financial interests or personal relationships that could have appeared to
influence the work reported in this paper.

\subsection*{Funding}
The research of Yinhua Xia was partially supported by National Key R\&D Program of China No. 2022YFA1005202/2022YFA1005200, and NSFC grant No. 12271498. The research of Yan Xu was partially supported by  NSFC grant No. 12071455.
%
%

\addcontentsline{toc}{section}{References}
\bibliographystyle{abbrv}
\normalem
\bibliography{shortreference}

\end{document}